\documentclass[11pt]{article}

\newcommand{\E}{\mathbb{E}}
\newcommand{\bbP}{\mathbb{P}}
\newcommand{\bsbracket}[1]{{\Big[ #1 \Big]}}

\newcommand{\mcal}[1]{\mathcal{#1}}

\usepackage{graphicx}
\usepackage{bm}
\usepackage{epstopdf}
\usepackage{booktabs} 
\usepackage{array} 
\usepackage{paralist} 
\usepackage{verbatim} 
\usepackage{caption}
\usepackage{subcaption}
\usepackage{fancyhdr} 
\usepackage[nottoc,notlof,notlot]{tocbibind}
\usepackage[titles,subfigure]{tocloft}

\usepackage{amsmath,amsfonts,amsthm,mathrsfs,amssymb,cite}
\usepackage[usenames]{color}
\usepackage{mathtools}
\mathtoolsset{showonlyrefs}

\newtheorem{thm}{Theorem}[section]

\newtheorem{lem}{Lemma}[section]

\newtheorem{ass}{Assumption}[section]

\theoremstyle{definition}

\newtheorem{rem}{Remark}[section]
\newtheorem{example}{Example}
\numberwithin{equation}{section}

\usepackage{algorithm}%
\usepackage{algorithmicx}%
\usepackage{algpseudocode}%
\usepackage{listings}%
\usepackage{ulem}

\allowdisplaybreaks
\title{\bf Nonparametric Schr\"odinger Bridge Time Series Generator: Algorithm, Convergence Analysis and Applications}

\author{
Daili Sheng
\thanks{School of Mathematics, Harbin Institute of Technology, Harbin, P. R. China.
Email: {\tt dlsheng@stu.hit.edu.cn}}
\and
Minghui Song
\thanks{School of Mathematics, Harbin Institute of Technology,
Harbin, P. R. China. Email: {\tt songmh@hit.edu.cn}}
\and
Hui Sun
\thanks{School of Mathematics and Physics, Xi'an Jiaotong-Liverpool University,
Suzhou, P. R. China. Email: {\tt Hui.Sun@xjtlu.edu.cn}}
}

\date{} 
\begin{document}
\maketitle

\begin{abstract}
We conduct a convergence analysis for the Schr\"{o}dinger Bridge Time Series (SBTS) data generator. Starting from a regularized formulation in which the data ensemble is mixed with a standard multivariate Gaussian distribution with a prescribed probability, we prove that the Euler--Maruyama discretization converges to the mixed target distribution with half-order convergence rate, provided that the ensemble size and kernel bandwidth are chosen appropriately. We further show that the regularized distribution converges to the original target distribution as the mixing probability tends to zero. The analysis simultaneously accounts for the ensemble approximation error, kernel approximation error, and time-discretization error, and therefore provides a full distributional convergence result for the SBTS generator. Empirically, we further examine the flexibility of the method by replacing the Wiener reference measure with the path measure induced by a more general SDE. The numerical experiments show that the schemes based on both the original Wiener reference measure and the SDE-induced reference measure achieve comparable performance, demonstrating the robustness and stability of the SBTS framework.

\medskip

\medskip

\noindent{\bf Keywords:}~~generative models, time series, Schr\"{o}dinger Bridge, convergence analysis


\end{abstract}

\section{Introduction}
Sequential data generation is an important task with many industrial applications. For example, in finance, stock prices, market dynamics, and economic scenarios evolve over time, and accurately capturing their statistical behavior is crucial for risk management. In natural language processing, speech and text are inherently sequential, since their meaning depends on the order of words and the surrounding context. Therefore, generating high-quality sequential data is an essential objective in modern data science, particularly for modeling, prediction, simulation, and decision-making in dynamic systems. With the dawning of deep learning/generative models, a variety of sequential data generators have emerged. Early deep-learning methods for sequential data generation mainly relied on recurrent architectures, including RNNs, LSTMs, and sequence-to-sequence models \cite{mikolov2010recurrent,graves2013generating,cho2014learning,sutskever2014sequence}. Subsequent neural network design is then powered by generative adversarial neural networks (GANs), including discrete sequences, continuous sequences, and time series \cite{goodfellow2014generative,yu2017seqgan,mogren2016crnngan,yoon2019timeseries}, latent-variable and adversarially trained recurrent models \cite{chung2015recurrent,lamb2016professor}, and autoregressive forecasting models such as DeepAR \cite{salinas2020deepar}. More recently, autoregressive convolutional models, Transformers, large-scale language models, and diffusion models have further advanced sequence generation in domains such as audio, text, and time series \cite{oord2016wavenet,vaswani2017attention,radford2019language,ho2020denoising,rasul2021autoregressive,kong2021diffwave}. Some other methods include the Schr{\"o}dinger bridge based methods \cite{hamdouche2026nonparametric} which formulate the timeseries generation task as entropic interpolation of the time series distribution, see also \cite{alouadi2026sbbts}. One outstanding feature of such method is that generation of data samples simply relies on the simulation of solution of a path-dependent SDE via the classical Euler-Maruyama scheme, and it requires no training of the neural networks. 

While there are a variety of the computational methods available, the convergence and numerical analysis of sequential/time-series generation is still less developed than the empirical modeling literature. In this work, we study the framework and algorithm of \cite{hamdouche2026nonparametric} and establish a rigorous numerical convergence analysis for the proposed scheme. The algorithm is cast as a generalized Schr{\"o}dinger bridge problem, where the path-space law is required to match prescribed marginal distributions not only at the terminal time, but also at intermediate time points. Its solution is thus an entropic optimal-transport interpolation between a reference path measure and a target path measure representing the joint law of the time series. Moreover, \cite{hamdouche2026nonparametric} shows that the solution can be characterized by a finite-horizon stochastic differential equation with a path-dependent drift:
\begin{equation}
\mathrm{d}X_t = \alpha\bigl(t,X_t;X_{t_1:\eta(t)}\bigr),\mathrm{d}t + \mathrm{d}W_t,
\qquad X_0=0.
\nonumber
\end{equation}
The drift $\alpha\bigl(t,X_t;X_{t_1:\eta(t)}\bigr)$, defined in \eqref{drift}, depends on both the current state and the past observations, and hence captures the sequential dependence in the generated process. 

However, the direct application of the Euler--Maruyama scheme is challenging for two main reasons. From a computational perspective, the conditional expectations arising at each time step must be approximated. From a theoretical perspective, the drift of the resulting SDE may fail to satisfy the standard Lipschitz conditions required for classical convergence analysis. In \cite{hamdouche2026nonparametric}, the first difficulty is addressed by using kernel density estimation based on a sequential product of density kernels, while the second is handled by regularizing the target distribution through a mixture with the joint law of Brownian motion. Both techniques have found important applications in financial mathematics and statistics. As a related example, \cite{reisinger2024convergence} studies a local stochastic volatility model for which the original McKean--Vlasov formulation is not directly well-posed, partly because the diffusion coefficient contains a conditional expectation in the denominator. To address this issue, the authors introduce both kernel density approximation and regularization, thereby obtaining a well-posed regularized SDE suitable for Euler--Maruyama discretization. A similar idea appears in \cite{jiao2026sfsconvexity}, where the target measure is mixed with a Gaussian distribution to ensure that the denominator in the regularized drift remains bounded away from zero. These works illustrate the usefulness of kernel-based approximation and distributional regularization. Nevertheless, they do not address genuinely sequential data generation. In addition, the convergence result in \cite{reisinger2024convergence} is established only for the regularized kernel-density approximation with fixed regularization parameters, rather than for the limiting regime in which the bandwidth tends to zero. In \cite{wang2026multimodal, huang2025sfs, jiao2026sfsconvexity}, the authors also performed convergence analysis on the Schr{\"o}dinger F{\"o}llmer sampler but under a different setup. They assumed that the density function of the target measure is available, and the drift takes the form under a change of measure which at each time step can be estimated by performing averaging based on Gaussian ensembles. We consider a more practical setting in which the target distribution is accessible only through samples and the data are sequential in nature. The closest analytical framework to ours is \cite{mazhar2026direct}; however, their work focuses on direct drift-level estimation under H{\"o}lder smoothness, a marginal density floor condition, and bounded support assumptions, which differ from the setting considered here. More specifically, our work provides a rigorous convergence analysis for the regularized problem. For each fixed $\epsilon>0$, and under suitable choices of the empirical sample size $M$ and kernel bandwidth $H$, we prove that the Euler--Maruyama approximation converges at rate $\sqrt{h}$. We further justify the regularization limit by showing that, as $\epsilon \to 0$, the law of the generated sample sequence converges to the law of the target time series in the $\mathcal{W}_2$ metric.

The main contributions of the current work are as follows: 
\begin{itemize}
    \item Under suitable growth assumptions on the density ratio between the target distribution and a Gaussian reference distribution, we prove convergence results for the SBTS model. In particular, our analysis simultaneously accounts for the temporal discretization error, the finite-ensemble error, the kernel approximation error, and the regularization error. To the best of our knowledge, this provides the first systematic numerical convergence treatment of Schr{\"o}dinger bridge-based time-series generation that incorporates all principal approximation parameters. 

    \item As a further empirical contribution, we extend the Schr{\"o}dinger bridge time-series method by allowing the reference measure to be induced by an SDE with a more general drift structure. We then compare the resulting data generator with the existing Brownian-reference formulation. The numerical results show that the samples generated under the SDE-induced reference measure achieve comparable quality to those generated under the Brownian reference. This suggests that the SBTS framework is stable with respect to the choice of reference dynamics and can be flexibly adapted beyond the standard Wiener setting.
\end{itemize}

The rest of the paper is organized as follows. Section~2 introduces the notation and problem setup. Section~3 presents the assumptions required for the subsequent analysis, together with the proposed SBTS framework and related preliminary constructions. In Section~4, we first outline the main roadmap of the convergence proof and then provide the detailed arguments. Section~5 presents numerical experiments designed to test the convergence behavior of the proposed method and to empirically examine its dependence on key parameters. In Section~6, we further investigate the robustness of the scheme by replacing the underlying Wiener reference measure with a path measure induced by a more general SDE. Finally, Section~7 concludes the paper.

\section{Preliminaries and notations}\label{sec2}
\subsection{Notations}
Throughout this paper, let $\Omega := C([0, T]; \mathbb{R}^d)$ be the space of $\mathbb{R}^d$-valued continuous functions defined on $[0, T]$, and let $\mathcal{B}(\Omega)$ denote its topological Borel field. Let $(\Omega, \mathcal{B}(\Omega))$ be the canonical space with the canonical process $X := (X_t)_{t\in [0, T]}$, where $X_t(\omega ) = \omega (t), \omega \in  \Omega$, and let \(\mathbb{F} = (\mathcal{F}_t)_{t \in [0, T]} \), where \( \mathcal{F}_t := \sigma\{X_s: s \in [0, t]\},  t \in [0, T] \).  In addition, we use $\mathcal{P}(E)$ to denote the space of all probability measures on the measurable space $E$. For $\mathbb{P} \in \mathcal{P}(\Omega)$,
$\mathbb{P}_t = X_t \# \mathbb{P} = \mathbb{P} \circ X_t^{-1}$ is the law of $X_t$. For a discrete time grid $\mathcal{T} = \{t_n, n = 1, \ldots, N\}$ with \(0 <t_1<\cdots<t_N\le T\), we set \( X_{t_1:t_n}:=(X_{t_1},\dots,X_{t_n}) \) for \(n = 1, \dots, N\). Let \(\mu\in\mathcal P((\mathbb R^d)^N)\) be a prescribed finite-dimensional distribution on the grid \(\mathcal T\).  We impose the finite-dimensional marginal constraint \(\mathbb P\circ X_{t_1:t_N}^{-1}=\mu\) on the path measure \(\mathbb P\in\mathcal P(\Omega)\). Under this constraint,
\(\mu\) is the distribution of \((X_{t_1},\ldots,X_{t_N})\). For \(n=1,..., N\), we denote by \(\mu_n\) the law of \(X_{t_1:t_n}\), and for \(n=1,\ldots,N-1\), by \(\mu_{n+1|1:n}\) the conditional distribution of \(X_{t_{n+1}}\) given \(X_{t_1:t_n} = x_{1:n} := (x_1, \ldots, x_n) \in (\mathbb{R}^d)^n \). When \(\mu\) admits a density with respect to the Lebesgue measure on \((\mathbb{R}^d)^N\), by abuse of notation,  we use the same symbol for its density and write \(\mu(x_{1:N})\). Define \(
\mu_n(x_{1:n}) := \int\mu(x_{1:N})\mathrm{d} x_{n+1} \cdots \mathrm{d} x_N \), we then have the conditional density \(
\mu_{n+1|1:n}(x_{n+1}|x_{1:n}) := \mu_{n+1}(x_{1:n+1})/\mu_{n}(x_{1:n}), \) for \(\mu_n\)-a.e. \(x_{1:n}\) such that \(\mu_n(x_{1:n})>0\). Similarly, let \(\mu^G\) be the distribution of the Brownian motion on the grid \( \mathcal{T}\). The corresponding marginal and conditional
distributions are denoted by \(\mu_n^G\) and \(\mu^G_{n+1|1:n}\), respectively. We use \(\mu^G(x_{1:N})\), \(\mu_n^G(x_{1:n})\) and \(\mu^G_{n+1|1:n}(x_{n+1}|x_{1:n})\) for their densities when they exist.

We equip \((\mathbb{R}^d)^n\) with the Euclidean norm and inner product
\[
|x_{1:n}|^2 := \sum_{i=1}^n |x_i|^2,
\qquad
\langle x_{1:n},y_{1:n}\rangle
:=
\sum_{i=1}^n \langle x_i,y_i\rangle_{\mathbb{R}^d},
\]
where \(|\cdot|\) and \(\langle\cdot,\cdot\rangle_{\mathbb{R}^d}\) denote the Euclidean norm and inner product in \(\mathbb{R}^d\), respectively.  When no confusion can arise, we use the same notation \(|\cdot|\) and \(\langle\cdot,\cdot\rangle\) for both \(\mathbb{R}^d\) and \((\mathbb{R}^d)^n\). For any matrix \( A \in \mathbb{R}^{nd \times nd} \), the Frobenius norm is defined as  
\(|A|_{\mathrm{F}} := \sqrt{\sum_{i,j=1}^{nd} A_{ij}^2}\) and the operator norm as \(|A| = |A|_{\text{op}} := \sup_{|v|=1} |Av|.\)  
Then \(|A| \leq |A|_{\mathrm{F}} \leq \sqrt{nd} |A|.\) On \( ((\mathbb{R}^d)^n, \mathcal{B}((\mathbb{R}^d)^n))\), define the space of probability measures with finite second moment by
\[
\mathcal{P}_2((\mathbb{R}^d)^n) := \left\{ \nu \in \mathcal{P}((\mathbb{R}^d)^n) \;\middle|\; \int_{(\mathbb{R}^d)^n} |y|^2  \nu(dy) < \infty \right\}.
\]
By \( \mathcal{L}(Y) \) we denote the probability distribution of the \( (\mathbb{R}^d)^n \)-valued random variable \( Y \). Then, for \( \nu_1, \nu_2  \in \mathcal{P}_2((\mathbb{R}^d)^n)\), the $L^2$-Wasserstein distance between multivariate distributions \(\nu_1\) and \(\nu_2\) is defined by
\begin{equation}
	\begin{aligned}
		\mathcal{W}_2(\nu_1, \nu_2) & := \inf_{\pi \in \mathcal{C}(\nu_1, \nu_2)} \left( \int_{(\mathbb{R}^d)^n \times (\mathbb{R}^d)^n} |y-\widetilde{y} |^2 \pi(dy, d\widetilde{y}) \right)^{1/2}\\
		&= \inf \left\{ \left(\mathbb{E}[|Y-\widetilde Y|^2]\right)^{1/2}, \quad \mathcal{L}(Y) = \nu_1, \quad \mathcal{L}(\widetilde{Y}) = \nu_2 \right\},
	\end{aligned}
\end{equation}
where \(\mathcal{C}(\nu_1, \nu_2)\) represents all the couplings of \(\nu_1\) and \(\nu_2\), i.e., \(\pi \in \mathcal{C}(\nu_1, \nu_2)\) if and only if \(\pi(\cdot, (\mathbb{R}^d)^n) = \nu_1(\cdot)\) and \(\pi((\mathbb{R}^d)^n, \cdot) = \nu_2(\cdot)\). Consequently, the \(L^2\)-Wasserstein distance can be bounded by the \(L^2\)-norm.

For any \(\phi \in C^2((\mathbb{R}^d)^n,\mathbb{R})\) and \(v_{1:n},x_{1:n}\in (\mathbb{R}^d)^n\), we define the directional derivative of
\(\phi\) at \(x_{1:n}\) in the direction \(v_{1:n}\) by
\begin{align}
	\nabla_{v_{1:n}} \phi(x_{1:n})
	= \lim_{\varepsilon\to 0}
	\frac{\phi(x_{1:n}+\varepsilon v_{1:n})-\phi(x_{1:n})}{\varepsilon}.
\end{align}
Since \(\phi\) is differentiable at \(x_{1:n}\), this derivative exists for
every \(v_{1:n}\in(\mathbb R^d)^n\), and
\(
\nabla_{v_{1:n}} \phi(x_{1:n})=\langle \nabla \phi(x_{1:n}),v_{1:n}\rangle .
\)
Here \(\nabla \phi(x_{1:n})\in(\mathbb R^d)^n\), and its norm is the Euclidean norm
induced by the inner product on \((\mathbb R^d)^n\). Thus, by the Cauchy--Schwarz inequality,
\begin{equation}
	\left|\nabla \phi(x_{1:n})\right|
	=
	\sup_{|v_{1:n}|=1}
	\left| \nabla_{v_{1:n}} \phi(x_{1:n})\right|.
\end{equation}
The Hessian \(\nabla^2\phi(x_{1:n})\) is equipped with the operator norm
	\begin{equation}
		\left|\nabla^2\phi(x_{1:n})\right|_{\mathrm{op}}=\sup_{|u_{1:n}|=|v_{1:n}|=1}\left|\nabla_{u_{1:n}}\nabla_{v_{1:n}}\phi(x_{1:n})
		\right|.
\end{equation}
We also write \(\nabla_{x_i}\phi(x_{1:n})\in\mathbb R^d \)
for the partial gradient of \(\phi\) with respect to the variable \(x_i\). For a vector-valued function \(\varphi\in C^2((\mathbb R^d)^n,\mathbb R^d),\) we denote by \(D_{x_i}\varphi(x_{1:n})\) its Jacobian matrix with respect to \(x_i\), namely
\[
D_{x_i}\varphi(x_{1:n})
= \left( \frac{\partial \varphi_p(x_{1:n})}{\partial x_{iq}}
\right)_{p,q=1}^d \in\mathbb R^{d\times d}.
\]
The second derivative of \(\varphi\) with respect to \(x_i\) and \(x_j\) is denoted by \(D_{x_i x_j}\varphi(x_{1:n})\). It can be viewed as a third-order tensor in \(\mathbb R^{d\times d\times d}\), or equivalently as a bilinear map 
\[ D_{x_i x_j}\varphi(x_{1:n}): \mathbb R^d\times\mathbb R^d \to \mathbb R^d.\]

\subsection{Schr\"{o}dinger Bridge problem for time series}
We now introduce the SBTS developed in \cite{hamdouche2026nonparametric}. We are given a target joint distribution \(\mu \in \mathcal{P}((\mathbb{R}^d)^N)\) corresponding to the law of a time series on \(\mathbb{R}^d\) observed at \(\mathcal{T}\).  Then, the SBTS is the path-space entropy projection
\begin{equation}\label{sbts}
	\mathbb{P}^* = {\arg\min}_{\mathbb{P}}\left\{ D(\mathbb{P} \| \mathbb{Q}): \mathbb{P} \in \mathcal{P}(\Omega), \mathbb{P}_0 = \delta_0, \mathbb{P} \circ X_{t_1:t_N}^{-1} = \mu \right\},
\end{equation}
where \( D(\mathbb{P} \| \mathbb{Q}) \) denotes the Kullback-Leibler divergence or the relative entropy between two probability measures \( \mathbb{P} \) and \( \mathbb{Q} \) which is defined as
\[
D(\mathbb{P} \| \mathbb{Q}) := 
\begin{cases} 
	\int \log\left(\frac{\mathrm{d}\mathbb{P}}{\mathrm{d}\mathbb{Q}}\right) \mathrm{d}\mathbb{P}, & \text{if } \mathbb{P} \ll \mathbb{Q}, \\
	\infty, & \text{otherwise}.
\end{cases}
\]
When \(N=1\) and \(t_1=T\), the constraint
\(\mathbb P\circ X_{t_1:t_N}^{-1}=\mu\) reduces to \(\mathbb P\circ X_T^{-1}=\mu\), and \eqref{sbts} becomes the classical Schr\"odinger bridge problem
with deterministic initial condition \(X_0=0\). For \(N\ge2\), the constraint in \eqref{sbts} requires the whole finite-dimensional distribution of \(X_{t_1:t_N}\) to match the target time series law \(\mu\) under \(\mathbb{P}\). Unlike the classical case, which only prescribes the terminal marginal distribution, the SBTS constraint encodes the temporal dependence among the observations at the prescribed time points.

In particular, when the reference measure \(\mathbb Q\) is chosen as the Wiener measure \(\mathbb W\), the canonical process \(X_t\) is a standard Brownian motion starting from \(0\) under \(\mathbb W\). For any \(\mathbb P\ll\mathbb W\) with finite relative entropy, Girsanov's theorem yields a \(\mathbb F\)-progressively measurable process \(\{\alpha_t\}_{t \in [0, T]}\) with \( \mathbb{E}_{\mathbb{P}}[\int_{0}^{T} |\alpha_t|^2 \mathrm{d} t] < \infty \), such that
\(
\frac{\mathrm{d}\mathbb{P}}{\mathrm{d}\mathbb{W}}|_{\mathcal{F}_T} := \exp\left\{ \int_0^T \langle \alpha_t, \mathrm{d}X_t\rangle - \frac{1}{2} \int_0^T |\alpha_t|^2 \mathrm{d} t \right\},
\)
and \( W_t := X_t - \int_0^t \alpha_s \mathrm{d}s \) is a \( \mathbb{P} \)-Brownian motion. In this case, we have \( D(\mathbb{P} \| \mathbb{W}) = \frac{1}{2} \mathbb{E}_{ \mathbb{P}}\left[\int_0^T |\alpha_t|^2 \mathrm{d}t\right]. \)
Then we can recast the SBTS problem as the following stochastic control problem:
\begin{equation}
	\begin{cases}
		\text{minimize over } \quad \alpha \in \mathcal{A}, & J(\alpha) = \frac{1}{2} \mathbb{E}_{\mathbb{P}} \left[ \int_0^T |\alpha_t|^2 \mathrm{d}t \right], \\
		\text{subject to } \quad \mathrm{d}X_t = \alpha_t \mathrm{d}t + \mathrm{d}W_t, & X_0 = 0, \quad X_{t_1:t_N} \overset{\mathbb{P}}{\sim} \mu,
	\end{cases}
\end{equation}
where $\mathcal{A}$ is the set of $\mathbb{R}^d$-valued $\mathbb{F}$-adapted processes s.t. 
$\mathbb{E}_{\mathbb{P}}[\int_0^T |\alpha_t|^2 \mathrm{d}t] < \infty$.

The following theorem provides the explicit dynamic representation of the SBTS problem as a piecewise SDE with an adapted log-gradient drift \cite{hamdouche2026nonparametric}.
\begin{thm}[\cite{hamdouche2026nonparametric}]\label{sbts_thm}
	    The SBTS problem \eqref{sbts} is solved by the probability measure \( \mathbb{P}^* := \frac{\mathrm{d}\mu}{\mathrm{d}\mu^G}(X_{t_1:t_N}) \mathbb{W} \) which is induced by the following SDE  
	\begin{equation}\label{sde1}
		\mathrm{d}X_t = \alpha(t, X_t; X_{t_1:\eta(t)})\mathrm{d}t + \mathrm{d}W_t, \quad X_0 = 0,
	\end{equation}
	with \(\eta(t) = \max\{t_n, t_n \leq t\}\), and for $ n = 0, \ldots, N-1 $, $t \in [t_n, t_{n+1}) $, where the drift is given by
	\begin{equation}
		\begin{aligned}\label{drift}
			&\alpha (t, x; x_{1:n}) = \frac{1}{t_{n+1} - t} 
			\frac{\mathbb{E}_{\mu} \left[ (X_{t_{n+1}} - x) F_n(t, x_n, x, X_{t_{n+1}}) \mid X_{t_1:t_n} = x_{1:n} \right]}
			{\mathbb{E}_{\mu} \left[ F_n(t, x_n, x, X_{t_{n+1}}) \mid X_{t_1:t_n} = x_{1:n} \right]}, \\
			&F_n(t, x_n, x, x_{n+1}) = \exp \left( -\frac{|x_{n+1} - x|^2}{2(t_{n+1} - t)} + \frac{|x_{n+1} - x_n|^2}{2(t_{n+1} - t_n)} \right),
		\end{aligned}
	\end{equation}
	and \(\mathbb{E}_\mu[\cdot]\) denotes expectation with respect to the target
	joint distribution \(\mu\).
\end{thm}
\begin{rem}
The representation in \eqref{drift} involves an expectation with respect to the target joint distribution $\mu$. Equivalently, it
can be rewritten as a Gaussian expectation involving the density ratio
between \(\mu\) and the Brownian reference law. Specifically, for
\(t\in[t_n,t_{n+1})\),
\begin{equation}\label{pre_regular}
	\begin{aligned}
		\alpha(t, x; x_{1:n})
		=& \frac{1}{t_{n+1} - t} 
		\frac{\mathbb{E}_{\mu} \left[ (X_{t_{n+1}} - x) F_n(t, x_n, x, X_{t_{n+1}}) \mid X_{t_1:t_n} = x_{1:n} \right]}
		{\mathbb{E}_{\mu} \left[ F_n(t, x_n, x, X_{t_{n+1}}) \mid X_{t_1:t_n} = x_{1:n} \right]}\\
		=& \nabla_x \log \mathbb{E}_{Y \sim N(0,I_d)} \left[\dfrac{ \mu_{n+1} \left(x_1,..., x_{n}, x + \sqrt{t_{n+1}-t}Y \right) }{\mu^G_{n+1} \left(x_1,..., x_{n}, x + \sqrt{t_{n+1}-t}Y\right)} \right]. 
	\end{aligned}
\end{equation}
This equivalent formulation expresses the expectation under the standard Gaussian measure rather than the target measure. Therefore, it is useful when the density of the target distribution, or an 	approximation of its density ratio with respect to the Gaussian reference law, is available.
\end{rem}
Theorem \ref{sbts_thm} provides a constructive description of the ideal SBTS diffusion. Nevertheless, in practical time series generation, the target law \(\mu\) is typically unknown and only finitely many samples are available. Consequently, the drift \(\alpha\) must be estimated from data and the resulting SDE must be discretized in time. This raises two central questions addressed in this work: how to construct a concrete data-driven numerical scheme for the SBTS diffusion, and how to quantify the errors introduced by drift estimation and temporal discretization, as well as their impact on the generated time-series distribution.

\section{Kernel estimation of the regularized drift and the temporal discretization}
For notational simplicity, our discussion is confined to the uniform grid $t_n = n, n = 1, ..., N$. The arguments extend directly to a general grid \(0<t_1<\cdots<t_N=T\) by replacing \(1\) with \(\Delta_n=t_{n+1}-t_n\).

We begin with the following assumptions throughout the paper:

\begin{ass}[Density and integrability]\label{asl}
	The target distribution \(\mu\) is absolutely continuous with respect to the Lebesgue measure on
	\((\mathbb R^d)^N\). Moreover,
	\[
	\int_{(\mathbb R^d)^N} |x_{1:N}|^2
	\mu(x_{1:N})\mathrm{d}x_{1:N}<\infty, \quad D(\mu\|\mu^G)<\infty.
	\]
\end{ass}

\begin{ass}[Regularity of density ratios]\label{asf}
	For each \(n=1,\ldots,N\), define the density ratio
	\[
	f_n(x_{1:n})
	:=
	\frac{\mu_n(x_{1:n})}{\mu_n^G(x_{1:n})},
	\qquad x_{1:n}\in(\mathbb R^d)^n .
	\]
	Assume that \(f_n\in C^2((\mathbb R^d)^n)\) and that there exist constants
	\(L,K_1>0\), independent of \(n\), such that for all
	\(x_{1:n},y_{1:n}\in(\mathbb R^d)^n\),
	\[
	|f_n(x_{1:n})-f_n(y_{1:n})|
	\le L \sum_{i=1}^{n}|x_i - y_i|,
	\]
	\[
	|\nabla f_n(x_{1:n})-\nabla f_n(y_{1:n})|
	\le L \sum_{i=1}^{n}|x_i - y_i|,
	\]
	and
	\[
	|f_n(x_{1:n})|
	\le K_1\bigl(1+\sum_{i=1}^{n}|x_i|\bigr).
	\]
\end{ass}

For later use, define the exponential weight
\[
\exp_n(x_{1:n})
:=
\exp\left(
\frac{|x_1|^2}{2}
+
\sum_{i=2}^{n}\frac{|x_i-x_{i-1}|^2}{2}
\right),
\qquad x_{1:n}\in(\mathbb R^d)^n .
\]

\begin{ass}[Regularity of weighted densities]\label{asg}
	For each \(n=1,\ldots,N\), define
	\[
	g_n(x_{1:n})
	:=
	\mu_n(x_{1:n})\bigl(\exp_n(x_{1:n})\bigr)^2 .
	\]
	Assume that \(g_n\) is Lipschitz continuous and has at most linear growth, i.e., there exist constants \(L,K_1>0\), independent of \(n\), such that for all
	\(x_{1:n},y_{1:n}\in(\mathbb R^d)^n\),
	\[
	|g_n(x_{1:n})-g_n(y_{1:n})|
	\le L \sum_{i=1}^{n}|x_i - y_i|,
	\]
	and
	\[
	|g_n(x_{1:n})|
	\le K_1\bigl(1+\sum_{i=1}^{n}|x_i|\bigr).
	\]
\end{ass}

\begin{rem}
		Assumptions \ref{asf} and \ref{asg} are stated directly for all
	\(n=1,\ldots,N\), mainly for notational convenience. In fact, under suitable regularity assumptions at \(n=N\), the Lipschitz and linear growth properties for \(f_n\), \(\nabla f_n\) and
	\(g_n\) \((n<N)\) follow from those. We take \(f_n\) as an example. Consider first \(n=N-1\). Since \(f_N=\mu_N/\mu_N^G\), we have
	\begin{align}
		f_{N-1}(x_{1:N-1})
		&=
		\frac{\mu_{N-1}(x_{1:N-1})}
		{\mu_{N-1}^G(x_{1:N-1})}
		\\
		&=
		\int f_N(x_{1:N})
		\frac{\mu_N^G(x_{1:N})}
		{\mu_{N-1}^G(x_{1:N-1})}
		\mathrm{d}x_N\\
        &=\mathbb{E}_{Y\sim\mathcal{N}(0,I_d)}
		\left[
		f_N(x_1,\ldots,x_{N-1},x_{N-1}+Y)
		\right].
	\end{align}
	If \(f_N\) is globally Lipschitz with constant \(L_N\), then
	\begin{align}
		&\left|
		f_{N-1}(x_{1:N-1})-f_{N-1}(y_{1:N-1})
		\right|\\
		\le&
		\mathbb{E}_{Y\sim\mathcal{N}(0,I_d)}\Big[
		\big|
		f_N(x_1,\ldots,x_{N-1},x_{N-1}+Y)
		-
		f_N(y_1,\ldots,y_{N-1},y_{N-1}+Y)
		\big|
		\Big]\\
		\le&
		L_N\left(
		\sum_{i=1}^{N-1}|x_i-y_i|
		+
		|x_{N-1}-y_{N-1}|
		\right)\\
		\le&
		2L_N\sum_{i=1}^{N-1}|x_i-y_i|.
	\end{align}
	Hence, \(f_{N-1}\) is also globally Lipschitz.
	Moreover, if
	\[
		|f_N(x_{1:N})|
		\le
		K_N\left(
		1+\sum_{i=1}^{N}|x_i|
		\right),
	\]
	then, 
    \begin{align}
		|f_{N-1}(x_{1:N-1})|
		&\le
		\mathbb{E}\left[
		\left|
		f_N(x_1,\ldots,x_{N-1},x_{N-1}+Y)
		\right|
		\right]\\
		&\le
		K_N\mathbb{E}_{Y\sim\mathcal{N}(0, I_d)}\left[
		1+\sum_{i=1}^{N-1}|x_i|
		+|x_{N-1}+Y|
		\right]\\
		&\le
		K_N\left(
		1+2\sum_{i=1}^{N-1}|x_i|
		+\mathbb{E}_{Y\sim\mathcal{N}(0, I_d)}[|Y|]
		\right)\\
		&\le
		K_{N-1}\left(
		1+\sum_{i=1}^{N-1}|x_i|
		\right)
	\end{align}
	for some constant \(K_{N-1}>0\). Repeating the same argument and taking a uniform upper bound over all levels gives the desired result.
\end{rem}

\begin{rem}
	Under Assumptions \ref{asl}-\ref{asf}, the gradient of \(f_n\) is uniformly bounded.
	Indeed, for any 
	$v_{1:n}=(v_1,\ldots,v_n)\in(\mathbb R^d)^n$, we have
	\begin{equation}
		|\nabla_{v_{1:n}} f_n(x_{1:n})|
		=
		\left|
		\lim_{\varepsilon\to 0}
		\frac{f_n(x_{1:n}+\varepsilon v_{1:n})-f_n(x_{1:n})}{\varepsilon}
		\right|  \notag\\
		\le
		L\sum_{i=1}^n |v_i|
		\le 
		\sqrt n L |v_{1:n}|.
	\end{equation}
	Hence
	\begin{align}
		|\nabla f_n(x_{1:n})|
		=
		\sup_{|v_{1:n}|=1}|\nabla_{v_{1:n}} f_n(x_{1:n})|
		\le \sqrt n L,
	\end{align}
	Moreover, by taking directions that only vary the \(i\)-th component, we obtain
	\begin{equation}
		|\nabla_{x_i} f_n(x_{1:n})|
		=
		\sup_{|v_i|=1}
		|\nabla_{(0,\ldots,0,v_i,0,\ldots,0)} f_n(x_{1:n})|
		\le L,
		\qquad i=1,\ldots,n .
	\end{equation}
	Similarly, applying Assumption \ref{asf}, whenever
	the Hessian exists,
	\begin{align}
		|\nabla^2 f_n(x_{1:n})|_{\mathrm{op}}
		\le \sqrt n L.
	\end{align}
	In particular, for each Hessian block,
	\begin{align}
		\left|\nabla_{x_j}\nabla_{x_i} f_n(x_{1:n})\right|_{\mathrm{op}}
		\le L,\qquad i,j=1,\ldots,n.
	\end{align}
	Here \(|\cdot|_{\mathrm{op}}\) denotes the operator norm induced by the
	Euclidean norm.
\end{rem}

It should be noted that, for the drift term in \eqref{pre_regular} above, guaranteeing the Lipschitz continuity typically requires a uniform positive lower bound on the density ratio \(\mu_n(x_{1:n})/ \mu^G_{n}(x_{1:n})\), which is restrictive but has been commonly used in the literature \cite{huang2025sfs, belinda2019theoretical}. To overcome the technical difficulty, we regularize \(\mu\) by mixing it with \(\mu^G\), similarly to \cite{huang2025sfs}.
Specifically, for \(\epsilon \in (0, 1)\), we introduce a modified target distribution 
\[\mu^{\epsilon} := (1 - \epsilon)\mu + \epsilon\mu^G\]
as the approximated target distribution. Thus the approximated SDE is defined as follows:
\begin{equation}\label{sde2}
	\mathrm{d} X_t^{\epsilon} = \alpha^{\epsilon}(t, X^{\epsilon}_t; X^{\epsilon}_{t_1:\eta(t)})\mathrm{d}t + \mathrm{d}W_t, \quad \eta(t) = \max\{t_n, t_n \leq t\},
\end{equation}
where the drift term takes the form:
\begin{equation}\label{alphaE}
	\begin{aligned}
		\alpha^{\epsilon}(t, x; x_{1:n}) 
		:=& \frac{1}{t_{n+1} - t} 
		\frac{\mathbb{E}_{\mu^\epsilon} \left[ (X_{t_{n+1}} - x) F_n(t, x_n, x, X_{t_{n+1}}) \mid X_{t_1:t_n} = x_{1:n} \right]}
		{\mathbb{E}_{\mu^\epsilon} \left[ F_n(t, x_n, x, X_{t_{n+1}}) \mid X_{t_1:t_n} = x_{1:n} \right]}\\
		=& \nabla_x \log \mathbb{E}_{Y \sim N(0,I_d)} \left[\dfrac{ \mu^\epsilon_{n+1} \left(x_1,..., x_{n}, x + \sqrt{t_{n+1}-t}Y \right) }{\mu^G_{n+1} \left(x_1,..., x_{n}, x + \sqrt{t_{n+1}-t}Y\right)} \right].
	\end{aligned}
\end{equation}
Note that 
\begin{equation}
	\begin{aligned}
		&\int F_n(t, x_n, x, x_{{n+1}}) \mu^G_{n+1}(x_{1:n+1}) \mathrm{d} x_{n+1}  \\
		=&\int \exp \left( -\frac{|x_{{n+1}} - x|^2}{2(t_{n+1}-t)} + \frac{|x_{{n+1}} - x_n|^2}{2} \right) \mu^G_{n+1}(x_{1:n+1}) \mathrm{d} x_{n+1}  \\
		=& (2 \pi)^{-\frac{d}{2}} \mu^G_{n}(x_{1:n}) \int \exp \left( -\frac{|x_{{n+1}} - x|^2}{2(t_{n+1}-t)} \right) \mathrm{d} x_{n+1} \\
		=& (2 \pi)^{-\frac{d}{2}} (2 \pi (t_{n+1}-t))^{\frac{d}{2}} \mu^G_{n}(x_{1:n}) \\
		=& C_G C_t^{-1} \mu^G_{n}(x_{1:n}).
	\end{aligned}
\end{equation}
where $C_t := (2 \pi (t_{n+1}-t))^{-\frac{d}{2}}$, and $C_G := (2 \pi)^{-\frac{d}{2}}$. Hence, \(\alpha^{\epsilon}(t, x; x_{1:n})\) can be written equivalently in the following form:
\begin{align}
	& \frac{1}{t_{n+1} - t} 
	\frac{\mathbb{E}_{\mu^\epsilon} \left[ (X_{t_{n+1}} - x) F_n(t, x_n, x, X_{t_{n+1}}) \mid X_{t_1:t_n} = x_{1:n} \right]}
	{\mathbb{E}_{\mu^\epsilon} \left[ F_n(t, x_n, x, X_{t_{n+1}}) \mid X_{t_1:t_n} = x_{1:n} \right]}\\
	=& \nabla_x \log \int F_n(t, x_n, x, x_{{n+1}}) \mu^{\epsilon}_{n+1|1:n} (x_{n+1}|x_{1:n}) \mathrm{d} x_{n+1}\\
	=& \nabla_x \log \int F_n(t, x_n, x, x_{{n+1}}) \mu^{\epsilon}_{n+1} (x_{1:n+1}) \mathrm{d} x_{n+1}\\
	=& \nabla_x \log \int F_n(t, x_n, x, x_{{n+1}}) \left[(1-\epsilon)\mu_{n+1} (x_{1:n+1}) + \epsilon \mu^G_{n+1} (x_{1:n+1}) \right] \mathrm{d} x_{n+1}\\
	=&\dfrac{(1-\epsilon) \nabla_x \int F_n(t, x_n, x, x_{{n+1}}) \mu_{n+1}(x_{1:n+1}) \mathrm{d} x_{n+1}}
	{(1-\epsilon)\int F_n(t, x_n, x, x_{{n+1}}) \mu_{n+1}(x_{1:n+1}) \mathrm{d} x_{n+1} + \epsilon \int F_n(t, x_n, x, x_{{n+1}}) \mu^G_{n+1}(x_{1:n+1}) \mathrm{d} x_{n+1}}\\
	=&\dfrac{(1-\epsilon) C_t \nabla_x \int F_n(t, x_n, x, x_{{n+1}}) \exp_n(x_{1:n}) \mu_{n+1}(x_{1:n+1}) \mathrm{d} x_{n+1}}
	{(1-\epsilon) C_t \int F_n(t, x_n, x, x_{{n+1}}) \exp_n(x_{1:n}) \mu_{n+1}(x_{1:n+1}) \mathrm{d} x_{n+1} + \epsilon C_G^{n+1}} \\
	=& \dfrac{(1-\epsilon) C_t d^{\epsilon}}{(1-\epsilon) C_t e^{\epsilon} + \epsilon C_G^{n+1}},
\end{align}
where 
\begin{align}
	d^{\epsilon} &:= \nabla_x \int F_n(t, x_n, x, x_{{n+1}}) \exp_n(x_{1:n}) \mu_{n+1}(x_{1:n+1}) \mathrm{d} x_{n+1},\label{de}\\
	e^{\epsilon} &:= \int F_n(t, x_n, x, x_{{n+1}}) \exp_n(x_{1:n}) \mu_{n+1}(x_{1:n+1}) \mathrm{d} x_{n+1}.\label{ee}
\end{align}

To estimate the drift, one can employ a kernel density estimation method. For technical reasons, we assume that we have two copies of i.i.d. data samples, denoted by \( \{X_{t_1:t_N}^i \}_{i=1}^M \) and  \(\{\tilde{X}_{t_1:t_N}^i \}_{i=1}^M \), of an \((\mathbb{R}^d)^N\)-valued random vector \(X_{t_1:t_N}\) with joint law \(\mu\). Define
\begin{align}
	d^{\epsilon,M} &:= \dfrac{1}{M} \sum\limits_{i=1}\limits^{M} \nabla_x F_n(t, X_{t_n}^{i}, x, X_{t_{n+1}}^{i}) \exp_n(X_{t_1:t_n}^i) \prod\limits_{j=1}^{n} K_H \left(x_j - X_{t_j}^i \right),\label{dm}\\
	e^{\epsilon,M} &:=  \dfrac{1}{M} \sum\limits_{i=1}\limits^{M} F_n(t, \tilde{X}_{t_n}^{i}, x, \tilde{X}_{t_{n+1}}^{i}) \exp_n(\tilde{X}_{t_1:t_n}^i) \prod\limits_{j=1}^{n} K_H \left(x_j - \tilde{X}_{t_j}^i \right).\label{em}
\end{align}
for \( t \in [t_n, t_{n+1}), x_{1:n} \in (\mathbb{R}^d)^n,  x \in \mathbb{R}^d \), \( n = 0, \ldots, N-1 \),  and \( K_H \) a kernel function, defined here as \( K_H(x) := H^{-d} K(x/H) \) with bandwidth \( H > 0 \). Then the approximate drift function is given by 
\begin{equation}\label{alphaM}
	\alpha^{\epsilon,M} (t, x; x_{1:n}) = \frac{(1-\epsilon) C_t d^{\epsilon,M}}{(1-\epsilon) C_t e^{\epsilon,M} + \epsilon C_G^{n+1}}.
\end{equation}
We then adopt the classical Euler-Maruyama scheme with \(\alpha^{\epsilon,M}\) to sample from the target distribution. More specifically, let $h = 1/m$ be a given stepsize with integer $m \geq 1$ and the grid points $t_n$ defined by $t_n=n, n = 0, 1,..., N$. The resulting numerical scheme is given by
\begin{equation}\label{sde3}
	X_{nm+l+1}^{\epsilon, N, M} = X_{nm+l}^{\epsilon, N, M} +\alpha^{\epsilon, M} (t_n+lh, X_{nm+l}^{\epsilon, N, M}; X_{m:nm}^{\epsilon, N, M}) h + \Delta W_{nm+l},
\end{equation}
for $n = 0, 1, ..., N-1, l = 0, 1,..., m-1,$ where  $X_{m:nm}^{\epsilon, N, M} = (X_{m}^{\epsilon, N, M}, ..., X_{nm}^{\epsilon, N, M})$, $X_0^{\epsilon, N, M}=0$, $\Delta W_{nm+l} = W_{t_n+(l+1)h} - W_{t_n+lh}$, $X_{nm+l}^{\epsilon, N, M}$ and $X_{m}^{\epsilon, N, M},...,X_{nm}^{\epsilon, N, M}$ are approximations to the exact solution $X_{t_n+lh}$ and $X_{t_1},..., X_{t_n}$ given by \eqref{sde1} , respectively.

It is worth noting that the proposed method does not require any pre-training. The algorithm is simple to implement, since it constructs samples step by step based on a deterministic approximation of the drift term. The pseudo-code of the SBTS algorithm is described in Algorithm 1.
\begin{algorithm}
	\caption{SBTS Simulation}
	\label{alg:sbts}
	\begin{algorithmic}[1]
		\Require two copies of data samples of time series \( \{X_{t_1:t_N}^i \}_{i=1}^M \), \(\{\tilde{X}_{t_1:t_N}^i \}_{i=1}^M \) and \(m\).
		\State Initialization: initial state \(x_0 = 0\);
		\For{\(i = 0, \ldots, N - 1\)}
		\State Initialize state \(y_0 = x_i\);
		\For{\(k = 0, \ldots, m - 1\)}
		\State Compute \(\alpha^{\epsilon,M} (t_i+k/m, y_k; x_{1:i})\) by kernel estimator \eqref{alphaM};
		\State Sample \(\varepsilon_k \sim \mathcal{N}(0, I_d)\) and compute
		\[
		y_{k+1} = y_k + \frac{1}{m} \alpha^{\epsilon, M}(t_i+k/m, y_k; x_{1:i}) + \frac{1}{\sqrt{m}} \varepsilon_k,
		\]
		\EndFor
		\State Set \(x_{i+1} = y_{m}\).
		\EndFor\\
		\Return \(x_1, \cdots, x_N\)
	\end{algorithmic}
\end{algorithm}

\section{Error estimates}
    In this section, we analyze the convergence of the Euler method introduced in \eqref{sde3}. Under Assumptions \ref{asl}-\ref{asg}, we derive an error estimate for the proposed algorithm in the \(L^2\)-Wasserstein distance. The analysis is divided into three parts. First, we prove the well-posedness of the regularized SBTS diffusion and establish Lipschitz and growth estimates for the regularized drift. Second, we quantify the error caused by the kernel approximation of the drift. This part is carried out by estimating separately the numerator and the denominator in the ratio representation of the drift, and by decomposing the kernel error into a bias term and a statistic term. Finally, we propagate the drift approximation error through the Euler scheme and obtain a recursive estimate for the strong error. The final Wasserstein bound follows from this strong error estimate. Figure~\ref{fig_sbts} summarizes the overall structure of the proposed algorithm and the main steps of the convergence analysis.
	
	\begin{figure}[htbp]
		\centering
		\includegraphics[width=\textwidth]{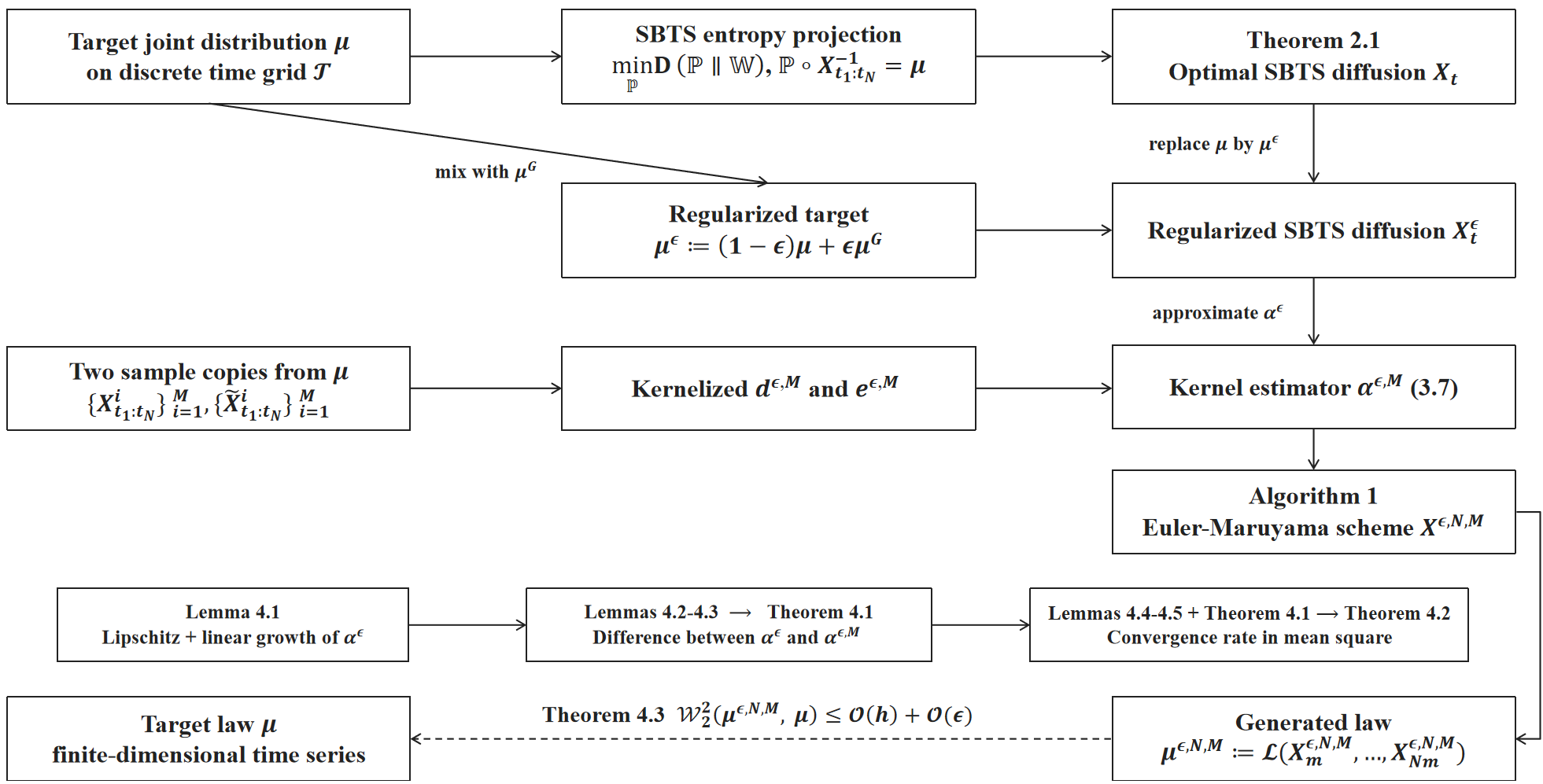}
		\caption{The proposed regularized SBTS algorithm and its convergence analysis.}
		\label{fig_sbts}
	\end{figure}

	Before proceeding with the analysis in this section, we give several elementary Gaussian integral identities and introduce some notation that will be used repeatedly throughout the paper. Specifically, we have
\begin{align}
	&(2 \pi)^{-\frac{d}{2}} \int \exp \left(-\frac{|y|^2}{2}\right) \mathrm{d}y = 1,\\
	&(2 \pi)^{-\frac{d}{2}} \int \exp \left(-\frac{|y|^2}{2}\right) |y| \mathrm{d}y = \sqrt{2} \frac{\Gamma\left(\tfrac{d+1}{2}\right)}{\Gamma\left(\tfrac{d}{2}\right)} := C_{h,1},\\
	&(2 \pi)^{-\frac{d}{2}} \int \exp \left(-\frac{|y|^2}{2}\right) |y|^2 \mathrm{d}y = d,\\
	&(2 \pi)^{-\frac{d}{2}} \int \exp \left(-\frac{|y|^2}{2}\right) |y|^3 \mathrm{d}y = 2^{\frac{3}{2}} \frac{\Gamma\left(\tfrac{d+3}{2}\right)}{\Gamma\left(\tfrac{d}{2}\right)}:= C_{h,2}.
\end{align}
Here, \(\Gamma(\cdot)\) denotes the Gamma function, defined by
\(\Gamma(s)=\int_0^\infty t^{s-1}e^{-t}\,\mathrm{d}t, \quad s>0. \)

\subsection{Well-posedness of the regularized SBTS diffusion}
To ensure that the regularized SBTS diffusion \eqref{sde2} admits a unique strong solution, we first establish the well-posedness of the regularized drift. The following lemma shows that the mixture structure \(\mu^\epsilon\) provides a uniform positive lower bound for the denominator in the logarithmic derivative, which allows us to prove spatial Lipschitz, temporal \(1/2\)-Hölder, and linear growth estimates. 

\begin{lem}\label{lemlip}
	Under Assumptions \ref{asl}-\ref{asf}, for each \(n=0,\ldots,N-1\), the regularized drift term \eqref{alphaE} is Lipschitz continuous in the spatial variables, \(1/2\)-Hölder continuous in time, and satisfies a linear growth condition at each time interval \([t_n, t_{n+1}]\).  That is, for any \(x, y \in \mathbb{R}^d, x_{1:n}, y_{1:n} \in (\mathbb{R}^d)^n\) and \(s, t \in [t_n, t_{n+1}]\), 
	\begin{equation}
		\begin{aligned}
			&\lvert \alpha^{\epsilon}(t,x; x_{1:n}) - \alpha^{\epsilon}(s,y; y_{1:n}) \rvert^2 \le
			\gamma^\epsilon \left( \lvert t-s \rvert+ \lvert x-y \rvert^2 + \lvert x_{1:n}-y_{1:n} \rvert^2\right) 
		\end{aligned}
	\end{equation}
	where $ \gamma^\epsilon := \left( 2N + 2{C_{h,1}^2}\right)	\left(\frac{L}{\epsilon}
	+ \left(\frac{L}{\epsilon}\right)^2	\right)^2$. As a consequence, there exists $C_0>0$, independent of $\epsilon$, $0<\epsilon<1$, such that
	\begin{equation}
		\begin{aligned}
			&\lvert \alpha^\epsilon(t, x; x_{1:n}) \rvert^2 \le C_0 \gamma^\epsilon ( 1 + \lvert x \rvert^2 + \lvert x_{1:n} \rvert^2 ).
		\end{aligned}
	\end{equation}
\end{lem}
\begin{proof}
	For \(t \in [t_n,t_{n+1}]\), the regularized drift can be written as
	\begin{equation}
		\alpha^{\epsilon}(t, x; x_{1:n}) 
		= \nabla_x \log \mathbb{E}_{Y \sim N(0,I_d)} \left[\dfrac{ \mu^\epsilon_{n+1} \left(x_1,..., x_{n}, x + \sqrt{t_{n+1}-t}Y \right) }{\mu^G_{n+1} \left(x_1,..., x_{n}, x + \sqrt{t_{n+1}-t}Y\right)} \right].
	\end{equation}
	Note that 
	\begin{equation}
		\mu^\epsilon(x_{1:N}) = (1-\epsilon)\mu(x_{1:N}) + \epsilon\mu^G(x_{1:N}).
	\end{equation}
	Hence, for each marginal density,
	\begin{equation}
		\mu_{n+1}^\epsilon(x_{1:n+1}) = (1-\epsilon)\mu_{n+1}(x_{1:n+1}) + \epsilon \mu^G_{n+1}(x_{1:n+1}).					
	\end{equation}
	By construction,
	\begin{equation}
		\dfrac{\mu_{n+1}^\epsilon(x_{1:n+1})}{\mu^G_{n+1}(x_{1:n+1})}
		=
		(1-\epsilon)\dfrac{\mu_{n+1}(x_{1:n+1})}{\mu^G_{n+1}(x_{1:n+1})} + \epsilon
		\ge \epsilon.
	\end{equation}
	Define
	\begin{align}
		d(t,x;x_{1:n}) &:= \nabla_x \mathbb{E}_{Y \sim N(0,I_d)} \left[\dfrac{ \mu^\epsilon_{n+1} \left(x_1,..., x_{n}, x + \sqrt{t_{n+1}-t}Y \right) }{\mu^G_{n+1} \left(x_1,..., x_{n}, x + \sqrt{t_{n+1}-t}Y\right)} \right],\\
		e(t,x;x_{1:n}) &:= \mathbb{E}_{Y \sim N(0,I_d)} \left[\dfrac{ \mu^\epsilon_{n+1} \left(x_1,..., x_{n}, x + \sqrt{t_{n+1}-t}Y \right) }{\mu^G_{n+1} \left(x_1,..., x_{n}, x + \sqrt{t_{n+1}-t}Y\right)} \right].
	\end{align}
	Then we have
	\begin{equation}\label{lip1_x}
		\begin{aligned}
			&\left|\frac{d(t, x; x_{1:n})}{e(t, x; x_{1:n})}-\frac{d(t,y; y_{1:n})}{e(t,y; y_{1:n})}\right|\\
			\le& \left|\frac{d(t, x; x_{1:n})-d(t,y; y_{1:n})}{e(t, x; x_{1:n})}\right|
			+ \left|\frac{d(t,y; y_{1:n})}{e(t,y; y_{1:n})}\right| \left|\frac{e(t,y; y_{1:n})-e(t, x; x_{1:n})}{e(t, x; x_{1:n})}\right| \\
			\le& \frac{L(|x-y|+\sum_{i=1}^{n}|x_i-y_i|)}{\epsilon}
			+ \frac{L^2(|x-y|+\sum_{i=1}^{n}|x_i-y_i|)}{\epsilon^2} \\
			\le& \left(\frac{L}{\epsilon}+\left(\frac{L}{\epsilon}\right)^2\right)(|x-y|+\sum_{i=1}^{n}|x_i-y_i|). 
		\end{aligned}
	\end{equation}
	Next, we estimate the dependence on the time variable. For fixed \(y \in \mathbb{R}^d\) and \(y_{1:n} \in (\mathbb{R}^d)^n\), using
	the elementary inequality
	\[
	\left|\sqrt{t_{n+1}-t}-\sqrt{t_{n+1}-s}\right|
	\le \sqrt{|t-s|},
	\quad s,t\in[t_n,t_{n+1}],
	\]
	and the definition of \(C_{h,1}\), we obtain
	\begin{equation}\label{lip1_t}
		\begin{aligned}
			&\left|\frac{d(t, y;y_{1:n})}{e(t, y;y_{1:n})}-\frac{d(s,y;y_{1:n})}{e(s,y;y_{1:n})}\right|\\
			\le& \left|\frac{d(t, y;y_{1:n})-d(s,y;y_{1:n})}{e(t, y;y_{1:n})}\right|
			+ \left|\frac{d(s,y;y_{1:n})}{e(s,y;y_{1:n})}\right| \left|\frac{e(t, y;y_{1:n})-e(s,y;y_{1:n})}{e(t, y;y_{1:n})}\right| \\
			\le& \frac{L|\sqrt{t_{n+1}-t}-\sqrt{t_{n+1}-s}|\mathbb{E}[|Y]}{\epsilon}
			+ \frac{L^2|\sqrt{t_{n+1}-t}-\sqrt{t_{n+1}-s}|\mathbb{E}[|Y|]}{\epsilon^2} \\
			\le& C_{h,1}\left(\frac{L}{\epsilon}+\left(\frac{L}{\epsilon}\right)^2\right)\sqrt{|t-s|}.
		\end{aligned}
	\end{equation}
	Combining \eqref{lip1_x} and \eqref{lip1_t}, we get
	\begin{equation}
		\begin{aligned}
			&|\alpha^\epsilon(t, x; x_{1:n})-\alpha^\epsilon(s,y; y_{1:n})|^2\\
			\le& 2(n+1)\left(\frac{L}{\epsilon}+\left(\frac{L}{\epsilon}\right)^2\right)^2\left(|x-y|^2+|x_{1:n}-y_{1:n}|^2\right)
			+ 2C_{h,1}^2\left(\frac{L}{\epsilon}+\left(\frac{L}{\epsilon}\right)^2\right)^2|t-s| \\
			\le& \gamma^\epsilon(|t-s|+|x-y|^2+|x_{1:n}-y_{1:n}|^2),
		\end{aligned}
	\end{equation}
	where \(\gamma^\epsilon := \left(2N+2C_{h,1}^2\right)
	\left(\frac{L}{\epsilon}+\left(\frac{L}{\epsilon}\right)^2\right)^2.\)
	Taking the fixed point \((t_n,0;0_{1:n})\), we have
	\begin{equation}
		\begin{aligned}
			|\alpha^\epsilon(t, x; x_{1:n})|^2
			&\le { 2 |\alpha^\epsilon(t, x; x_{1:n})-\alpha^\epsilon(t_n,0; 0_{1:n})|^2}
			+ 2|\alpha^\epsilon(t_n,0; 0_{1:n})|^2 \\
			&\le 2 \gamma^\epsilon(|t-t_n|+|x|^2+|x_{1:n}|^2) + 2|\alpha^\epsilon(t_n,0; 0_{1:n})|^2.
		\end{aligned}
	\end{equation}
	On the other hand, since \(e(t_n,0;0_{1:n})\ge \epsilon\) and \(|d(t_n,0;0_{1:n})|\le L\), we have
	\[
	|\alpha^\epsilon(t_n,0;0_{1:n})|^2
	=
	\left|
	\frac{d(t_n,0;0_{1:n})}
	{e(t_n,0;0_{1:n})}
	\right|^2\le
	\left(\frac{L}{\epsilon}\right)^2
	\le
	\left(\frac{L}{\epsilon}
	+\left(\frac{L}{\epsilon}\right)^2\right)^2.
	\]
	Therefore, there exists a constant \(C_0>0\), independent of \(\epsilon\), such that
		\begin{equation}
			|\alpha^\epsilon(t, x; x_{1:n})|^2
			\le C_{0} \gamma^\epsilon(1+|x|^2+|x_{1:n}|^2),
	\end{equation}
	This completes the proof.
\end{proof}
\begin{rem}
	The drift coefficient in \eqref{sde2} is defined piecewise in time. 
	Nevertheless, the SDE is well-posed on the whole interval \([0,T]\). 
	On each subinterval \([t_n,t_{n+1}]\), the past values 
	\(X_{t_1:t_n}\) have already been determined by the solution constructed on the previous 
	intervals, and the drift satisfies the global Lipschitz and linear growth 
	conditions with respect to the current state variable \(X_t\) and the historical variables \(X_{t_1:t_n}\). Therefore, the standard 
	existence and uniqueness theorem for SDEs applies successively on each 
	subinterval. Since the number of subintervals is finite, the local solutions 
	can be concatenated to obtain a unique strong solution on \([0,T]\). This means that the regularized SBTS diffusion \eqref{sde2} has a unique strong solution.
\end{rem}


\subsection{Error estimate for the kernel drift approximation} 
In this subsection, we estimate the error between the regularized drift \(\alpha^{\epsilon}(t,x;x_{1:n})\) and its kernel approximation \(\alpha^{\epsilon, M}(t,x;x_{1:n})\) in mean square for any fixed \(t \in [t_n, t_{n+1})\), \(x \in \mathbb{R}^d\) and \(x_{1:n} \in (\mathbb{R}^d)^n\). Unless otherwise specified, all the analyses in this subsection are conducted on the interval \(t \in [t_n, t_{n+1})\). The main difficulty comes from the ratio structure and kernel estimation of the drift.  We therefore estimate the numerator and the denominator separately. Define
\begin{align}
	d^h &:= \mathbb{E}_{\mu}\left[
	\nabla_x F_n (t, X_{t_n}, x, X_{t_{n+1}})\exp_n \left(X_{t_1:t_n}\right) \prod\limits_{j=1}^nK_H \left(x_j - X_{t_j} \right)\right],\label{dh}\\
	e^h &:=  \mathbb{E}_{\mu}\left[
	F_n (t, X_{t_n}, x, X_{t_{n+1}})\exp_n \left(X_{t_1:t_n}\right) \prod\limits_{j=1}^nK_H \left(x_j - X_{t_j} \right) \right],\label{eh}
\end{align}
and the drift under \(d^h\), \(e^h\) is given by
\begin{equation}\label{alphah}
	\alpha^h(t,x;x_{1:n}) := \frac{(1-\epsilon) C_t d^h}{(1-\epsilon) C_t e^h + \epsilon C_G^{n+1}}.
\end{equation}
Since \(d^{\epsilon,M}\) and \(e^{\epsilon,M}\) are constructed from two independent copies of the data samples, it follows that
\begin{align}
	&\mathbb{E}\left[\left|\alpha^\epsilon(t,x;x_{1:n})-\alpha^{\epsilon,M}(t,x;x_{1:n})	\right|^2\right]\\
	\le& 2 \mathbb{E}\left[\left|\alpha^\epsilon(t,x;x_{1:n})-\alpha^h(t,x;x_{1:n})	\right|^2\right] + 2 \mathbb{E}\left[\left|\alpha^h(t,x;x_{1:n})-\alpha^{\epsilon,M}(t,x;x_{1:n})	\right|^2\right]\\
	\le&4 \frac{(1-\epsilon)^2 C_t^2}{\epsilon^2 C_G^{2(n+1)}} \mathbb{E}\left[ \left| d^\epsilon - d^h \right|^2 \right] + 4 \frac{(1-\epsilon)^4 C_t^4}{\epsilon^4 C_G^{4(n+1)}} \mathbb{E}\left[ \left| d^h \right|^2 \right]\mathbb{E}\left[ \left| e^h - e^\epsilon \right|^2 \right]\\
	&+ 4 \frac{(1-\epsilon)^2 C_t^2}{\epsilon^2 C_G^{2(n+1)}} \mathbb{E}\left[ \left| d^h - d^{\epsilon,M} \right|^2 \right]+4 \frac{(1-\epsilon)^4 C_t^4}{\epsilon^4 C_G^{4(n+1)}} \mathbb{E}\left[ \left| d^{\epsilon,M}\right|^2 \right]\mathbb{E}\left[ \left| e^{\epsilon,M} - e^h \right|^2 \right].
\end{align}
We next state the assumptions on the kernel function and give a basic estimate for \(e^{\epsilon, M}\), \(d^{\epsilon, M}\) and \(d^h\).
\begin{ass}[Kernel conditions]\label{kernel}
	The kernel \( K : \mathbb{R}^d \to \mathbb{R} \) is a nonnegative function with the normalization and symmetry properties
	\[
	\int K(v) \mathrm{d} v = 1, \qquad  K(v) = K(-v),
	\]
	Moreover, \(K\) satisfies the following integrability conditions,
	\[
	\int  K^2(v) \mathrm{d}v < \infty, \qquad \int K^2(v) |v| \mathrm{d}v < \infty,
	\]
	and
	\[
	\int K(v) |v| \mathrm{d}v < \infty, \qquad \int K(v) |v|^2 \mathrm{d}v < \infty.
	\]
\end{ass}
\begin{lem}\label{bound}
	Assume that Assumptions \ref{asl}--\ref{asg} and \ref{kernel} hold. Let \(0<H\le 1\). Then, for any \(t\in[t_n,t_{n+1})\), \(x\in\mathbb R^d\), and \(x_{1:n}\in(\mathbb R^d)^n\), there exists a constant \(K_2>0\), independent of \(H\), such that
	\begin{align}
		\mathbb{E}_{\mu}\left[\left| e^{\epsilon, M} \right|^2 \right] &\le K_2 (t_{n+1}-t)^{\frac{d}{2}} H^{-nd} (1+ |x|+\sum_{i=1}^{n}|x_i|), 	\\
		\mathbb{E}_{\mu}\left[\left| d^{\epsilon, M} \right|^2 \right] &\le K_2 (t_{n+1}-t)^{\frac{d}{2}-1} H^{-nd} (1+ |x|+\sum_{i=1}^{n}|x_i|), 	\\
		\left|d^{h} \right|^2 &\le K_2 (t_{n+1}-t)^{d-1} (1+ |x|^2+|x_{1:n}|^2),
	\end{align}
	where
	\begin{align*}
		K_2 &:= \max \{C_K^n K_1  \pi^{\frac{d}{2}} (1+C_{h, 1}), C_K^n K_1 (2 \pi)^{\frac{d}{2}} (\frac{1}{2})^{\frac d2 + 1}(d+C_{h, 2}), (n+2) \bar{C}_K^{n} K_1^2 (2 \pi)^{-(n+1)d} (d+C_{h, 1})^2
		\}, 
		\\
		C_K^n &= n \int 
		K^2(v) |v| \mathrm{d}v \left(\int  
		K^2(v) \mathrm{d}v\right)^{n-1} +\left(\int  
		K^2(v) \mathrm{d}v\right)^n,  \bar{C}_K^{n} = \left(n \int K(v) |v| \mathrm{d}v + 1\right)^2.
	\end{align*}
\end{lem}
\begin{proof}
	We first estimate \(e^{\epsilon,M}\). 
	Since \(e^{\epsilon,M}\) is the empirical average of independent
	samples, by Jensen's inequality and Assumption \ref{asg}, we obtain
	\begin{align}
		&\mathbb{E}_{\mu}\left[\left| e^{\epsilon, M} \right|^2 \right] \\
		\le& \mathbb{E}_{\mu}\left[ \left|
		F_n (t, X_{t_n}, x, X_{t_{n+1}})\exp_n \left(X_{t_1:t_n}\right) \prod\limits_{j=1}^n K_H \left(x_j - X_{t_j} \right)\right|^2 \right]\\
		=& \int  F_n^2 (t, y_n, x, y_{n+1})\exp^2_n \left(y_{1:n}\right) \prod\limits_{j=1}^n 
		K_H^2(x_j-y_j) \mu(y_{1:N}) \mathrm{d}y_{1:N}\\
		=& \int  \exp \left(-\frac{|y_{n+1}-x|^2}{t_{n+1}-t} \right)
		\prod\limits_{j=1}^n 
		K_H^2(x_j-y_j) g_{n+1}(y_{1:n+1}) \mathrm{d}y_{1:n+1}\\
		=& \left(\frac{t_{n+1}-t}{2}\right)^{\frac d2} \frac{1}{H^{nd}} \int \exp\left(-\frac{|y|^2}{2}\right)  \prod\limits_{j=1}^n 
		K^2(v_j)
		g_{n+1}\left(x_1+v_1H,...,x_n+v_nH,  x+y\sqrt{\frac{t_{n+1}-t}{2}} \right)  \mathrm{d}v_{1:n} \mathrm{d}y\\
		\le& K_1 \left(\frac{t_{n+1}-t}{2}\right)^{\frac d2}\frac{1}{H^{nd}}\int \exp\left(-\frac{|y|^2}{2}\right)\prod\limits_{j=1}^n 
		K^2(v_j)\left(1+\sum_{i=1}^{n}\left(|x_i|+|v_i|H\right)+|x|+|y|\right)\mathrm{d}v_{1:n} \mathrm{d}y.
	\end{align}
	Based on the Assumption \ref{kernel}, we have
		\begin{align}
			& \int \prod\limits_{j=1}^n 
			K^2(v_j)\left(1+\sum_{i=1}^{n}\left(|x_i|+|v_i|H\right)+|x|+|y|\right)\mathrm{d}v_{1:n} \\
			=& H \int \prod\limits_{j=1}^n 
			K^2(v_j) \sum\limits_{i=1}^n |v_i| \mathrm{d}v_{1:n} + \int \prod\limits_{j=1}^n 
			K^2(v_j) \mathrm{d}v_{1:n} \left(1+\sum_{i=1}^{n}|x_i|+|x|+|y|\right) \label{eq1}\\
			=& H \sum\limits_{i=1}^n \prod\limits_{j=1, j\neq i}^n \int 
			K^2(v_j)\mathrm{d}v_j \int K^2(v_i) |v_i| \mathrm{d}v_i + \prod\limits_{j=1}^n \int  
			K^2(v_j) \mathrm{d}v_j \left(1+\sum_{i=1}^{n}|x_i|+|x|+|y|\right) \\
			\le& C_K^n \left(1+\sum_{i=1}^{n}|x_i|+|x|+|y|\right), 
		\end{align}
where \(C_K^n = n \int K^2(v) |v| \mathrm{d}v \left(\int K^2(v) \mathrm{d}v\right)^{n-1} +\left(\int 
K^2(v) \mathrm{d}v\right)^n\). Hence, it is easy to acquire that
\begin{align}
	\mathbb{E}_{\mu}\left[\left| e^{\epsilon, M} \right|^2 \right] \le& K_1 \left(\frac{t_{n+1}-t}{2}\right)^{\frac d2}\frac{1}{H^{nd}}\int \exp\left(-\frac{|y|^2}{2}\right) C_K^n \left(1+\sum_{i=1}^{n}|x_i|+|x|+|y|\right) \mathrm{d} y\\
	\le& K_1 \left(\frac{t_{n+1}-t}{2}\right)^{\frac d2}\frac{1}{H^{nd}} C_K^n (2\pi)^{\frac d2}\left((1+\sum_{i=1}^{n}|x_i|+|x|) + C_{h,1}\right)\\
	\le& C_K^n K_1 \pi^{\frac{d}{2}} (t_{n+1}-t)^{\frac d2}\frac{1}{H^{nd}} \Bigl((1+\sum_{i=1}^{n}|x_i|+|x|) + C_{h,1}\Bigr)\\
	\le& K_{h, 1} (t_{n+1}-t)^{\frac d2 } \frac{1}{H^{nd}} (1+\sum_{i=1}^{n}|x_i|+|x|),
\end{align}
where $K_{h, 1} = C_K^n K_1  \pi^{\frac{d}{2}} (1+C_{h, 1})$. Next, we estimate \(d^{\epsilon,M}\). Similar to the estimation for \(\mathbb{E}_{\mu}\left[\left| e^{\epsilon, M} \right|^2 \right]\), applying Assumption \ref{asg}, it can be obtained that
\begin{align}
	&\mathbb{E}_{\mu}\left[\left| d^{\epsilon, M} \right|^2 \right]\\
	\le& \mathbb{E}_{\mu}\left[ \left|\nabla_x 
	F_n (t, X_{t_n}, x, X_{t_{n+1}})\exp_n \left(X_{t_1:t_n}\right) \prod\limits_{j=1}^n K_H \left(x_j - X_{t_j} \right)\right|^2 \right]\\
	=& \int \left|\frac{y_{n+1}-x}{t_{n+1}-t}\right|^2 F^2_n (t, y_n, x, y_{n+1})\exp^2_n \left(y_{1:n}\right) \prod\limits_{j=1}^n 
	K_H^2(x_j-y_j) \mu(y_{1:N}) \mathrm{d}y_{1:N}\\
	=& \int \left|\frac{y_{n+1}-x}{t_{n+1}-t}\right|^2 \exp \left(-\frac{|y_{n+1}-x|^2}{t_{n+1}-t} \right)
	\prod\limits_{j=1}^n 
	K_H^2(x_j-y_j) g_{n+1}(y_{1:n+1}) \mathrm{d}y_{1:n+1}\\
	=& \frac{1}{4} \left(\frac{t_{n+1}-t}{2}\right)^{\frac d2 -1}\frac{1}{H^{nd}}\int |y|^2 \exp\left(-\frac{|y|^2}{2}\right) \prod\limits_{j=1}^n 
	K^2(v_j)
	g_{n+1}\left( x_1+v_1H,...,x_n+v_nH,  x+y\sqrt{\frac{t_{n+1}-t}{2}} \right)  \mathrm{d}v_{1:n} \mathrm{d}y\\
	\le& \frac{K_1}{4} \left(\frac{t_{n+1}-t}{2}\right)^{\frac d2 -1}\frac{1}{H^{nd}}\int |y|^2 \exp\left(-\frac{|y|^2}{2}\right)\prod\limits_{j=1}^n 
	K^2(v_j)(1+\sum_{i=1}^{n}(|x_i|+|v_i|H)+|x|+|y|)\mathrm{d}v_{1:n} \mathrm{d}y,
\end{align}
Combining \eqref{eq1}, we can arrive at
\begin{align}
	\mathbb{E}_{\mu}\left[\left| d^{\epsilon, M} \right|^2 \right] \le& \frac{K_1}{4} \left(\frac{t_{n+1}-t}{2}\right)^{\frac d2 -1}\frac{1}{H^{nd}}\int |y|^2 \exp\left(-\frac{|y|^2}{2}\right) C_K^n \left(1+\sum_{i=1}^{n}|x_i|+|x|+|y|\right) \mathrm{d} y\\
	\le& \frac{K_1}{4} \left(\frac{t_{n+1}-t}{2}\right)^{\frac d2 -1}\frac{1}{H^{nd}} C_K^n (2\pi)^{\frac d2}\left((1+\sum_{i=1}^{n}|x_i|+|x|)d + C_{h,2}\right)\\
	\le& C_K^n K_1 (2 \pi)^{\frac{d}{2}} (\frac{1}{2})^{\frac d2 + 1}(t_{n+1}-t)^{\frac d2 - 1} \frac{1}{H^{nd}} \left((1+\sum_{i=1}^{n}|x_i|+|x|)d + C_{h,2}\right)\\
	\le& K_{h, 2} (t_{n+1}-t)^{\frac d2 -1} \frac{1}{H^{nd}} (1+\sum_{i=1}^{n}|x_i|+|x|),
\end{align}
where $K_{h, 2} = C_K^n K_1 (2 \pi)^{\frac{d}{2}} (\frac{1}{2})^{\frac d2 + 1}(d+C_{h, 2})$. It remains to estimate \(d^h\). Using Assumption \ref{asf}, we have
\begin{align}
	&\left| d^h \right|^2\\
	\le& \left|\mathbb{E}_{\mu}\left[ 
	\nabla_x 
	F_n (t, X_{t_n}, x, X_{t_{n+1}})\exp_n \left(X_{t_1:t_n}\right) \prod\limits_{j=1}^n K_H \left(x_j - X_{t_j} \right) \right] \right|^2 \\
	=& \left| \int \frac{y_{n+1}-x}{t_{n+1}-t} F_n (t, y_n, x, y_{n+1})\exp_n \left(y_{1:n}\right) \prod\limits_{j=1}^n K_H \left(x_j - y_j \right) \mu(y_{1:N}) \mathrm{d}y_{1:N} \right|^2\\
	=&  \left|(2\pi)^{-\frac{n+1}{2}d} \int \frac{y_{n+1}-x}{t_{n+1}-t} \exp \left(-\frac{|y_{n+1}-x|^2}{2(t_{n+1}-t)} \right)
	\prod\limits_{j=1}^n K_H \left(x_j - y_j \right) \dfrac{\mu_{n+1}(y_{1:n+1})}{ \mu^G_{n+1}(y_{1:n+1})} \mathrm{d}y_{1:n+1} \right|^2\\
	=& \left| (2\pi)^{-\frac{n+1}{2}d} (t_{n+1}-t )^{\frac{d}{2} - \frac{1}{2}}\int y \exp\left(-\frac{|y|^2}{2}\right)
	\prod\limits_{j=1}^n K \left(v_j \right)
	f_{n+1}\left(x_1+v_1H,...,x_n+v_nH,  x+y\sqrt{t_{n+1}-t}  \right) \mathrm{d}v_{1:n}  \mathrm{d}y\right|^2\\
	\le&\left| K_1 (2\pi)^{-\frac{n+1}{2}d} (t_{n+1}-t )^{\frac{d}{2} - \frac{1}{2}} \int |y|  \exp\left(-\frac{|y|^2}{2}\right)\prod\limits_{j=1}^n K \left(v_j \right)(1+\sum_{i=1}^{n}(|x_i|+|v_i|H)+|x|+|y|)\mathrm{d}v_{1:n} \mathrm{d}y \right|^2.
\end{align}
It follows from Assumption \ref{kernel} that
\begin{align}
	& \int \prod\limits_{j=1}^n K \left(v_j \right)(1+\sum_{i=1}^{n}(|x_i|+|v_i|H)+|x|+|y|)\mathrm{d}v_{1:n} \\
	=&\int H \int \prod\limits_{j=1}^n 
	K(v_j) \sum\limits_{i=1}^n |v_i| \mathrm{d}v_{1:n} + \int \prod\limits_{j=1}^n 
	K(v_j) \mathrm{d}v_{1:n} (1+\sum_{i=1}^{n}|x_i|+|x|+|y|) \\
	=&H \sum\limits_{i=1}^n \prod\limits_{j=1, j\neq i}^n \int 
	K(v_j)\mathrm{d}v_j \int K(v_i) |v_i| \mathrm{d}v_i + \prod\limits_{j=1}^n \int  
	K(v_j) \mathrm{d}v_j (1+\sum_{i=1}^{n}|x_i|+|x|+|y|) \\
	\le& \left(nH \int 
	K(v) |v| \mathrm{d}v +1\right)\left(1+\sum_{i=1}^{n}|x_i|+|x|+|y|\right),
\end{align}
and then
\begin{align}
	&\left|d^h \right|^2\\
	\le&  \left|K_1 (2\pi)^{-\frac{n+1}{2}d} (t_{n+1}-t )^{\frac{d}{2} - \frac{1}{2}} \int  |y|  \exp\left(-\frac{|y|^2}{2}\right) \left(nH \int K(v) |v| \mathrm{d}v +1\right)\left(1+\sum_{i=1}^{n}|x_i|+|x|+|y|\right) \mathrm{d} y\right|^2\\
	\le& \left| K_1 (2\pi)^{-\frac{n+1}{2}d} (t_{n+1}-t )^{\frac{d}{2} - \frac{1}{2}} \left(nH \int 
	K(v) |v| \mathrm{d}v +1\right)(2\pi)^{\frac d2}\left((1+\sum_{i=1}^{n}|x_i|+|x|)C_{h,1}+d \right)\right|^2\\
	\le& \bar{C}_K^{n} K_1^2 (2 \pi)^{-nd} (t_{n+1}-t)^{d - 1} \left((1+\sum_{i=1}^{n}|x_i|+|x|)C_{h,1} + d\right)^2\\
	\le& K_{h, 3} (t_{n+1}-t)^{d - 1} (1+|x|^2+|x_{1:n}|^2),
\end{align}
	where $K_{h, 3} = (n+2) \bar{C}_K^{n} K_1^2 (2 \pi)^{-nd} (d+C_{h, 1})^2$, $\bar{C}_K^n = \left(n \int 
	K(v) |v| \mathrm{d}v +1\right)^2$. Define $K_2 := \max \{C_K^n K_1  \pi^{\frac{d}{2}} (d+C_{h, 1}), C_K^n K_1 (2 \pi)^{\frac{d}{2}} (\frac{1}{2})^{\frac d2 + 1}(d+C_{h, 2}), (n+2) \bar{C}_K^{n} K_1^2 (2 \pi)^{-(n+1)d} (d+C_{h, 1})^2
	\}$, we conclude the proof.
\end{proof}
We are now ready to quantify the approximation errors arising from the
kernel approximation and the empirical sampling. More precisely, we derive mean square error estimate for  \(d^\epsilon-d^h\),  \(d^h-d^{\epsilon,M}\), together with the
corresponding estimates for  \(e^\epsilon-e^h\) and \(e^h-e^{\epsilon,M}\).
\begin{lem}\label{de_error}
	Suppose that Assumptions \ref{asl}-\ref{asg} and \ref{kernel} hold. Then, for any \(t\in[t_n,t_{n+1})\), there exists a constant \(K_2>0\),
	independent of \(H\) and \(M\), such that
	\begin{align}
		\left| d^\epsilon - d^h \right|^2
		\le& n^2 H^4 \mathcal{K}_2^2 L^2(2\pi)^{-nd} (t_{n+1}-t)^{{d-1}} C_{h, 1}^2,\\
		\mathbb{E}\left[ \left| d^h - d^{\epsilon, M} \right|^2 \right]
		\le& \frac{1}{M} K_2 (t_{n+1}-t)^{\frac{d}{2}-1} H^{-nd} (1+|x|+\sum_{i=1}^{n}|x_i|),\\
		\left| e^\epsilon - e^h \right|^2 
		\le& n^2 H^4 \mathcal{K}_2^2 L^2(2\pi)^{-nd} (t_{n+1}-t)^{d},\\
		\mathbb{E}\left[ \left| e^h - e^{\epsilon, M} \right|^2 \right]
		\le& \frac{1}{M} K_2 (t_{n+1}-t)^{\frac{d}{2}} H^{-nd} (1+|x|+\sum_{i=1}^{n}|x_i|),
	\end{align}
	where \(K_2\) is the constant defined in Lemma \ref{bound}.
\end{lem}
\begin{proof}
	We first estimate the kernel approximation error of the numerator term. By the definition \eqref{de} and \eqref{dh} , we have
	\begin{align}
		\left| d^\epsilon - d^h \right|^2 =& \Bigg| \mathbb{E}_{\mu}\left[ \nabla_x
		F_n (t, X_{t_n}, x, X_{t_{n+1}})\exp_n(X_{t_1:t_n}) \prod\limits_{j=1}^n K_H \left(x_j - X_{t_j} \right)\right] \\
		&- \int \nabla_x F_n (t, x_n, x, x_{n+1}) \exp_n(x_{1:n}) \mu_{n+1}(x_{1:n+1}) \mathrm{d} x_{n+1} \Bigg|^2\\
		=& \Bigg| \int \nabla_x
		F_n (t, y_{n}, x, y_{{n+1}})\exp_n(y_{1:n}) \prod\limits_{j=1}^n K_H \left(x_j - y_{j} \right) \mu(y_{1:N}) \mathrm{d} y_{1:N} \\
		&- \int \nabla_x F_n (t, x_n, x, x_{n+1}) \exp_n(x_{1:n}) \mu_{n+1}(x_{1:n+1}) \mathrm{d} x_{n+1}
		\Bigg|^2,
	\end{align}
	Define 
	\begin{align}
		P(x_{1:n}) &= \int \nabla_x
		F_n (t, x_{n}, x, x_{{n+1}})\exp_n(x_{1:n}) \mu_{n+1}(x_{1:n+1}) \mathrm{d} x_{n+1}\\
		&= \int \frac{x_{n+1}-x}{t_{n+1}-t} \exp \left( -\dfrac{|x_{{n+1}} - x|^2}{2(t_{n+1}-t)}\right) \exp_{n+1}(x_{1:n+1}) \mu_{n+1}(x_{1:n+1}) \mathrm{d} x_{n+1} \\
		&= \int \frac{x_{n+1}-x}{t_{n+1}-t} \exp \left( -\dfrac{|x_{{n+1}} - x|^2}{2(t_{n+1}-t)}\right)(2\pi)^{-\frac{n+1}{2}d}\dfrac{\mu_{n+1} (x_{1:n+1})}{\mu^G_{n+1} (x_{1:n+1})}\mathrm{d}x_{n+1}\\
		&= (2\pi)^{-\frac{n+1}{2}d}	(t_{n+1}-t)^{\frac{d}{2}} \int \frac{y}{\sqrt{t_{n+1}-t}} \exp\left(-\frac{|y|^2}{2}\right) f_{n+1}(x_1,..., x_{n}, x + y\sqrt{t_{n+1}-t}) \mathrm{d}y.
	\end{align}
	Then, for \(i,j=1,\ldots,n\), Assumption \ref{asf} yields
	\begin{align}
		|D_{x_i x_j}^2 P (x_{1:n})| &\le L(2\pi)^{-\frac{n+1}{2}d} (t_{n+1}-t)^{\frac{d-1}{2}}\int |y| \exp \left(-\frac{|y|^2}{2}\right)\mathrm{d}y\\
		&\le L(2\pi)^{-\frac{nd}{2}} (t_{n+1}-t)^{\frac{d-1}{2}} C_{h, 1}.
	\end{align}
	 By fixing \(x_{1:n}\) and using the change of variables \(y_{1:n} = x_{1:n} + v_{1:n}H \), we obtain
	\begin{align}
		\left| d^\epsilon - d^h \right|^2 \le& \left| \int P(y_{1:n}) \prod\limits_{j=1}^n K_H \left(x_j - y_{j} \right) \mathrm{d} y_{1:n} - P(x_{1:n})\right|^2\\
		\le& \left| \int
		P(x_{1:n}+v_{1:n}H) \prod\limits_{j=1}^n K \left(v_j \right) \mathrm{d} v_{1:n} - P(x_{1:n})\right|^2.
	\end{align}
	By Taylor expansion of \(P(x_{1:n}+v_{1:n}H)\) at \(x_{1:n}\), 
	\begin{equation}
		\begin{aligned}
			& P(x_{1:n}+v_{1:n}H) = P(x_{1:n}) + H \sum\limits_{i=1}^n D_{x_i} P(x_{1:n})v_i + \frac{H^2}{2} \sum\limits_{i=1}^n\sum\limits_{j=1}^n D_{x_i x_j}^2 P(x_{1:n})(v_i, v_j) + o(H^2 |v_{1:n}|^2),\\
			& |D_{x_i x_j}^2 P(x_{1:n})(v_i, v_j)| \leq |D_{x_i x_j}^2 P(x_{1:n})| |v_i||v_j|,
		\end{aligned}
	\end{equation}
	together with the moment conditions \( \int K(v) \mathrm{d} v = 1, \int v K(v) \mathrm{d} v = 0, \mathcal{K}_2 :=  \int |v|^2 K(v) \mathrm{d} v < \infty\), yields
	\begin{align}
		\left| d^\epsilon - d^h \right|^2 
		\le& \left| \int \left( \frac{H^2}{2} \sum\limits_{i=1}^n\sum\limits_{j=1}^n D_{x_i x_j}^2 P(x_{1:n})(v_i, v_j) + o(H^2 \sum_{i=1}^n |v_i|^2) \right) \prod\limits_{j=1}^n K \left(v_j \right) \mathrm{d} v_{1:n} \right|^2\\
		\le& H^4 \left| \int \left( \sum\limits_{i=1}^n\sum\limits_{j=1}^n | D_{x_i x_j}^2 P(x_{1:n})(v_i, v_j) | \right) \left(\prod\limits_{j=1}^n K \left(v_j \right)\right) \mathrm{d} v_{1:n} \right|^2\\
		\le& H^4 \mathcal{K}_2^2 \sum\limits_{i=1}^n \sum\limits_{j=1}^n |D_{x_i x_j}^2 P(x_{1:n})|^2 \le  n^2 H^4 \mathcal{K}_2^2 L^2(2\pi)^{-nd} (t_{n+1}-t)^{{d-1}} C_{h, 1}^2.
	\end{align}
	We next estimate the statistical error caused by the Monte Carlo
	approximation. According to the definition \eqref{dm} and \eqref{dh} , we have
 \begin{align}
 	\mathbb{E}\left[ \left| d^h - d^{\epsilon, M} \right|^2 \right]
 	=& \mathbb{E}\Bigg[ \Bigg|
 	\mathbb{E}_{\mu}\left[ \nabla_x
 	F_n (t, X_{t_n}, x, X_{t_{n+1}})\exp(X_{t_1:t_n}) \prod\limits_{j=1}^n K_H \left(x_j - X_{t_j} \right)\right]\\
 	&-\dfrac{1}{M} \sum\limits_{i=1}\limits^{M}
 	\nabla_x F_n (t, X_{t_n}^{i}, x, X_{t_{n+1}}^{i}) \exp(X_{t_1:t_n}^i) \prod\limits_{j=1}^n K_H \left(x_j - X_{t_j}^i \right) \Bigg|^2\Bigg]\\
 	=& \frac{1}{M^2} \sum\limits_{i=1}\limits^{M} \mathbb{E}\Bigg[ \Bigg| \nabla_x F_n (t, X_{t_n}^{i}, x, X_{t_{n+1}}^{i}) \exp(X_{t_1:t_n}^i) \prod\limits_{j=1}^n K_H \left(x_j - X_{t_j}^i \right) \\
 	&- \mathbb{E}_{\mu}\left[ \nabla_x
 	F_n (t, X_{t_n}, x, X_{t_{n+1}})\exp(X_{t_1:t_n}) \prod\limits_{j=1}^n K_H \left(x_j - X_{t_j} \right)\right] \Bigg|^2\Bigg],
 \end{align}
 by using the inequality \( \mathbb{E} \left[ \left|X - \mathbb{E}[X]\right|^2 \right] \le \mathbb{E} \left[X^2\right] \), we obtain
 \begin{align}
 	\mathbb{E}\left[ \left| d^h - d^{\epsilon, M} \right|^2 \right]
 	\le& \frac{1}{M} \mathbb{E}_{\mu}\left[ \left|\nabla_x
 	F_n (t, X_{t_n}, x, X_{t_{n+1}})\exp(X_{t_1:t_n}) \prod\limits_{j=1}^n K_H \left(x_j - X_{t_j} \right) \right|^2\right]\\
 	\le& \frac{1}{M} K_2 (t_{n+1}-t)^{\frac{d}{2}-1} H^{-nd} (1+|x|+\sum_{i=1}^{n}|x_i|).
  \end{align}
 Following the same argument, according to the definitions in \eqref{ee}, \eqref{eh} and \eqref{em}, we can also have
 \begin{align}
 	\left| e^\epsilon - e^h \right|^2 
 	&\le n^2 H^4 \mathcal{K}_2^2 L^2(2\pi)^{-nd} (t_{n+1}-t)^{d},\\
 	\mathbb{E}\left[ \left| e^h - e^{\epsilon, M} \right|^2 \right]
 	&\le \frac{1}{M} K_2 (t_{n+1}-t)^{\frac{d}{2}} H^{-nd} (1+|x|+\sum_{i=1}^{n}|x_i|).
 \end{align}
 This completes the proof.
\end{proof}
We now combine the Lemma \ref{bound} and \ref{de_error} to quantify the overall error between \(\alpha^\epsilon\) and \(\alpha^{\epsilon,M}\).
\begin{thm}\label{drift_error}
	Under Assumptions \ref{asl}-\ref{asg} and \ref{kernel}, for any 
	\(t\in[t_n,t_{n+1})\), there exist constants \(K_{31},K_{32}>0\),
	independent of \(H\) and \(M\), such that
	\begin{align}
		\mathbb{E}\left[\left|\alpha^{\epsilon}(t,x;x_{1:n})	- \alpha^{\epsilon, M}(t,x;x_{1:n})\right|^2 \right] 
		\le& \left[ K_{31}(t_{n+1}-t)^{-1} H^4 + K_{32} (t_{n+1}-t)^{-d-1} \frac{1}{MH^{2nd}} \right]\\
		& \times \left(1+|x|^2+\sum_{i=1}^{n}|x_i|^2\right), \qquad \forall t \in [t_n, t_{n+1}), 
	\end{align}
	where \(K_{31} = \frac{4}{\epsilon^4} N^2  \mathcal{K}_2^2 L^2 \left(C_{h,1}^2 +  K_2 (2\pi)^{Nd}\right), K_{32} = \frac{4}{\epsilon^4} (N+2) \left( K_2 (2\pi)^{Nd} + K_2^2 (2\pi)^{2Nd} \right) \).
\end{thm}
\begin{proof}
	According to Lemma \ref{bound} and \ref{de_error}, we have the following sequence of inequalities:
	\begin{align}
		&\mathbb{E}\left[\left|\alpha^\epsilon(t,x;x_{1:n})-\alpha^{\epsilon,M}(t,x;x_{1:n})	\right|^2\right]\\
		\le& 2 \mathbb{E}\left[\left|\alpha^\epsilon(t,x;x_{1:n})-\alpha^h(t,x;x_{1:n})	\right|^2\right] + 2 \mathbb{E}\left[\left|\alpha^h(t,x;x_{1:n})-\alpha^{\epsilon,M}(t,x;x_{1:n})	\right|^2\right]\\
		\le&4 \frac{(1-\epsilon)^2 C_t^2}{\epsilon^2 C_G^{2(n+1)}} \mathbb{E}\left[ \left| d^\epsilon - d^h \right|^2 \right] + 4 \frac{(1-\epsilon)^4 C_t^4}{\epsilon^4 C_G^{4(n+1)}} \mathbb{E}\left[ \left| d^h \right|^2 \right]\mathbb{E}\left[ \left| e^h - e^\epsilon \right|^2 \right]\\
		&+ 4 \frac{(1-\epsilon)^2 C_t^2}{\epsilon^2 C_G^{2(n+1)}} \mathbb{E}\left[ \left| d^h - d^{\epsilon,M} \right|^2 \right]+4 \frac{(1-\epsilon)^4 C_t^4}{\epsilon^4 C_G^{4(n+1)}} \mathbb{E}\left[ \left| d^{\epsilon,M}\right|^2 \right]\mathbb{E}\left[ \left| e^{\epsilon,M} - e^h \right|^2 \right]\\
		\le& 4 \frac{(1-\epsilon)^2 C_t^2}{\epsilon^2 C_G^{2(n+1)}} n^2 H^4 \mathcal{K}_2^2 L^2(2\pi)^{-nd} (t_{n+1}-t)^{{d-1}} C_{h, 1}^2 \\
		&+ 4 \frac{(1-\epsilon)^4 C_t^4}{\epsilon^4 C_G^{4(n+1)}} K_2 (t_{n+1}-t)^{d-1} (1+ |x|^2+\sum_{i=1}^{n}|x_i|^2) \times n^2 H^4 \mathcal{K}_2^2 L^2(2\pi)^{-nd} (t_{n+1}-t)^{d} \\
		&+ 4 \frac{(1-\epsilon)^2 C_t^2}{\epsilon^2 C_G^{2(n+1)}} \frac{1}{M} K_2 (t_{n+1}-t)^{\frac{d}{2}-1} H^{-nd} (1+|x|+\sum_{i=1}^{n}|x_i|)\\
		&+ 4 \frac{(1-\epsilon)^4 C_t^4}{\epsilon^4 C_G^{4(n+1)}} K_2 (t_{n+1}-t)^{\frac{d}{2}-1} H^{-nd} (1+ |x|+\sum_{i=1}^{n}|x_i|) \times \frac{1}{M} K_2 (t_{n+1}-t)^{\frac{d}{2}} H^{-nd} (1+|x|+\sum_{i=1}^{n}|x_i|)\\
		\le& \frac{4}{\epsilon^4} n^2  \mathcal{K}_2^2 L^2 \left(C_{h,1}^2 +  K_2 (2\pi)^{nd}\right) (t_{n+1}-t)^{-1} H^4 (1+ |x|^2+\sum_{i=1}^{n}|x_i|^2)\\ 
		& + \frac{4}{\epsilon^4} (n+2) \left( K_2 (2\pi)^{nd} + K_2^2 (2\pi)^{2nd} \right) (t_{n+1}-t)^{-d-1} \frac{1}{MH^{2nd}} (1+ |x|^2+\sum_{i=1}^{n}|x_i|^2)\\
		\le& \left[ K_{31}(t_{n+1}-t)^{-1} H^4 + K_{32} (t_{n+1}-t)^{-d-1} \frac{1}{MH^{2nd}} \right] (1+ |x|^2+\sum_{i=1}^{n}|x_i|^2),
	\end{align}
where \(K_{31} = \frac{4}{\epsilon^4} N^2  \mathcal{K}_2^2 L^2 \left(C_{h,1}^2 +  K_2 (2\pi)^{Nd}\right), K_{32} = \frac{4}{\epsilon^4} (N+2) \left( K_2 (2\pi)^{Nd} + K_2^2 (2\pi)^{2Nd} \right) \).
\end{proof}
The above results complete the error analysis of the kernel approximation for the regularized drift. Specifically, the error is separated into a deterministic bias term and a statistical Monte Carlo term, which are shown in Theorem \ref{drift_error} to be of orders \(H^4\) and \((MH^{2nd})^{-1}\), respectively, up to the singular factor depending on \(t_{n+1}-t\). This is essential to propagate the drift approximation error through the Euler discretization scheme.

\subsection{Convergence analysis of the Euler discretization} 
Due to the path-dependent structure of the drift coefficient, we first establish second moment bounds and temporal increment estimates for the exact solution \(X^\epsilon\), where the estimates are formulated through accumulated sums over the observation time grid. We then derive the uniform boundedness in second moment estimates for the numerical solution \(X^{\epsilon,N,M}\). Finally, we obtain the mean-square convergence rate of the proposed Euler discretization scheme.
\begin{lem}\label{exact_solution_bound}
	Assume that Assumption \ref{asl}-\ref{asf} hold. Let \(\epsilon>0\) satisfy
	\(C_0\gamma^\epsilon>d\), 
	where \(C_0\) and \(\gamma^\epsilon\) are defined in Lemma \ref{lemlip}. Define \[M_n:=1+\sum_{i=0}^{n-1} \sup_{t\in[t_i,t_{i+1}]} \mathbb{E}\left[|X_t^\epsilon|^2\right].\]
	Then
	\[M_n < \frac{C^{n+1}}{C-1},\]
	where
	\(C =4(1+C_0\gamma^\epsilon) \exp(3C_0\gamma^\epsilon)\). Moreover, for any \(t_n<s<t\leq t_{n+1}\), the exact solution satisfies
	\[ \mathbb{E} \left[ |X_t^\epsilon-X_s^\epsilon|^2\right]
	\leq \Theta_n(\epsilon,t-s), \]
	where \(\Theta_n(\epsilon,t-s)
	:= 2C_0\gamma^\epsilon(t-s)^2M_{n+1}
	+ 2d(t-s)\). 
\end{lem}
 
\begin{proof}
	When $t \in [t_n, t_{n+1}]$, applying the elementary inequality, Cauchy–Schwarz inequality and Lemma \ref{lemlip}, one can show that
	\begin{align}
		\mathbb{E}\left[|X_t^\epsilon|^2\right] &\leq 3 \mathbb{E}\left[|X_{t_n}^\epsilon|^2\right] + 3\mathbb{E}\left[|\int_{t_n}^{t} \alpha^{\epsilon} (s, X_s^{\epsilon}; X_{t_1:t_n}^{\epsilon}) \mathrm{d}s|^2\right] + 3\mathbb{E}\left[|\int_{t_n}^{t} \mathrm{d} W_s |^2\right]\\
		&\leq 3 \mathbb{E}\left[|X_{t_n}^\epsilon|^2\right] + 3 C_0 \gamma^\epsilon (t-t_n) \int_{t_n}^{t} (1 + \mathbb{E}\left[| X_s^{\epsilon} |^2\right] + \sum_{i=1}^{n} \mathbb{E}\left[| X_{t_i}^{\epsilon} |^2\right] )  \mathrm{d}s + 3d(t-t_n).
	\end{align}
	By the definition of $M_n$, we have 
	\begin{align}
		\mathbb{E}\left[|X_t^\epsilon|^2 \right] &\leq 3 \mathbb{E}\left[ |X_{t_n}^\epsilon|^2\right] + 3 C_0 \gamma^\epsilon M_n + 3 C_0 \gamma^\epsilon \int_{t_n}^{t} (\mathbb{E} \left| X_s^\epsilon\right|^2)\mathrm{d} s + 3d.
	\end{align}
	Applying the Gronwall inequality yields that
	\begin{align}
		\sup\limits_{t \in [t_n, t_{n+1}]} \mathbb{E}\left[|X_t^\epsilon|^2 \right] \leq \left(3 \mathbb{E}\left[ |X_{t_n}^\epsilon|^2\right] + 3 C_0 \gamma^\epsilon M_n + 3d \right) \exp( 3 C_0 \gamma^\epsilon).
	\end{align}
	Since \(d < C_0 \gamma^\epsilon\), we can arrive at
	\begin{align}
		M_{n+1} = M_n + 
		\sup\limits_{t \in [t_n, t_{n+1}]} \mathbb{E}\left[|X_t^\epsilon|^2 \right] 
		&\leq M_n + \left( 3 \mathbb{E}\left[ |X_{t_n}^\epsilon|^2\right] + 3 C_0 \gamma^\epsilon M_n + 3d \right) \exp( 3 C_0 \gamma^\epsilon)\\
		&\leq \left( 4M_n + 3 C_0 \gamma^\epsilon M_n + 3d \right) \exp( 3 C_0 \gamma^\epsilon)\\
		&\leq 4 \left( 1 + C_0 \gamma^\epsilon \right) \exp( 3 C_0 \gamma^\epsilon)(1+M_n)\\
		&\leq C(1+M_n),
	\end{align}
	where $C=4 \left( 1 + C_0 \gamma^\epsilon \right) \exp( 3 C_0 \gamma^\epsilon)$.
	Consequently,
	\begin{align}
		M_n = 1 + \sum\limits_{i=0}^{n-1} \sup\limits_{t \in [t_i, t_{i+1}]} \mathbb{E}\left[|X_t^\epsilon|^2\right] \leq C^{n-1}\left(\frac{C}{C-1} + M_1\right) < \frac{C^{n+1}}{C-1},
	\end{align}
	for $M_1 =1+ \sup\limits_{t \in [t_0, t_{1}]} \mathbb{E}\left[|X_t^\epsilon|^2\right] \leq 1 + (2d+2C_0\gamma^{\epsilon})\exp (2 C_0\gamma^\epsilon)<C$.
	Similarly, for
	\(t_n<s<t\leq t_{n+1}\), we have
	\begin{align}
		\mathbb{E}[|X_t^\epsilon - X_s^\epsilon|^2] &\leq 2\mathbb{E}\left[\left|\int_s^t \alpha^{\epsilon} (r, X_r^{\epsilon}; X_{t_1:t_n}^{\epsilon}) \mathrm{d}r\right|^2\right] + 2\mathbb{E}\left[\left|\int_s^t dW_r\right|^2\right]\\
		&\leq 2C_0  \gamma^\epsilon (t-s)\mathbb{E}\left[\int_s^t( 1 + | X_r^{\epsilon} |^2 + \sum_{i=1}^{n}| X_{t_i}^{\epsilon} |^2) \mathrm{d}r\right] + 2d(t-s) \nonumber \\
		&\leq 2C_0  \gamma^\epsilon (t-s)^2 M_{n+1} + 2d(t-s)\\
		&\leq \Theta_n(\epsilon, t-s).
	\end{align}
\end{proof}

\begin{lem}\label{num_solution_bound}
	Assume that Assumptions \ref{asl}-\ref{asg} and \ref{kernel} hold. Let \(\epsilon>0\) be such that \(C_0\gamma^\epsilon>d\), where \\\(C_0\) and
	\(\gamma^\epsilon\) are defined as in Lemma \ref{lemlip}. Suppose further that the bandwidth \(H\) and the sample size \(M\) are chosen such that \(H\sim\mathcal O(h^{1/2})\) and  \(M\sim\mathcal O(h^{-(n+1)d-2})\) on each intervel \([t_n, t_{n+1}), n=0,...,N-1\). Define 
	$$M_n^N :=  1+ \sum\limits_{i=0}^{n-1} \max\limits_{0\leq l \leq m} \mathbb{E}\left[ |X_{im+l}^{\epsilon, N, M}|^2 \right].$$
	Then
	\begin{equation}
		M_n^N \leq 2\Lambda^{n},
	\end{equation}
	where \(\Lambda = 3 \exp \left(1+4C_0 \gamma^\epsilon+4K_3h\right) \), $K_3 = K_{31} + K_{32}$, $K_{31}$ and $K_{32}$ are the constants defined in Theorem \ref{drift_error}.
\end{lem}
\begin{proof}
	By Theorem \ref{drift_error}, together with the choices
	\(M\sim\mathcal O(h^{-(n+1)d-2})\) and \(H\sim\mathcal O(h^{1/2})\),
	for \(0\leq l\leq m-1\), we have
	\begin{align}
		&\mathbb{E}\left[|\alpha^{\epsilon}(t_n+lh, X_{nm+l}^{\epsilon, N, M}; X_{m:nm}^{\epsilon, N, M}) - \alpha^{\epsilon, M}(t_n+lh, X_{nm+l}^{\epsilon, N, M}; X_{m:nm}^{\epsilon, N, M})|^2 \right]\\
		\le& \left[ K_{31}(t_{n+1}-t)^{-1} H^4 + K_{32} (t_{n+1}-t)^{-d-1} \frac{1}{MH^{2nd}} \right]  \left(1 + \mathbb{E}\left[|X_{nm+l}^{\epsilon, N, M}|^2 \right] + \sum_{i=1}^{n}\mathbb{E}\left[|X_{im}^{\epsilon, N, M}|^2 \right]\right)\\
		\le& (K_{31} + K_{32}) h \left(1 + \mathbb{E}\left[|X_{nm+l}^{\epsilon, N, M}|^2 \right] + \sum_{i=1}^{n}\mathbb{E}\left[|X_{im}^{\epsilon, N, M}|^2 \right]\right).
	\end{align}
	Introducing \(K_3=K_{31}+K_{32}\), the above estimate yields
	\begin{align}
		\mathbb{E}\left[|X_{nm+l+1}^{\epsilon, N, M}|^2\right] 
		\leq& (1+h)\mathbb{E}\left[|X_{nm+l}^{\epsilon, N, M}|^2\right] + \left(1+\frac{1}{h}\right)\mathbb{E}\left[|\alpha^{\epsilon, M}(t_n+lh, X_{nm+l}^{\epsilon, N, M}; X_{m:nm}^{\epsilon, N, M})|^2\right] h^2 + dh \\
		\leq& (1+h)\mathbb{E}\left[|X_{nm+l}^{\epsilon, N, M}|^2 \right]+ 4h \mathbb{E}\left[|\alpha^{\epsilon} (t_n+lh, X_{nm+l}^{\epsilon, N, M}; X_{m:nm}^{\epsilon, N, M})|^2\right]  + dh\\
		&+ 4h \mathbb{E}\left[|\alpha^{\epsilon}(t_n+lh, X_{nm+l}^{\epsilon, N, M}; X_{m:nm}^{\epsilon, N, M}) - \alpha^{\epsilon, M}(t_n+lh, X_{nm+l}^{\epsilon, N, M}; X_{m:nm}^{\epsilon, N, M})|^2 \right]\\
		\leq& (1+h)\mathbb{E}\left[|X_{nm+l}^{\epsilon, N, M}|^2\right] + 4h C_0 \gamma^\epsilon \left(1 + \mathbb{E}\left[|X_{nm+l}^{\epsilon, N, M}|^2 \right] + \sum_{i=1}^{n}\mathbb{E}\left[|X_{im}^{\epsilon, N, M}|^2 \right]\right) + dh \\
		&+ 4 K_3 h^2 \left(1 + \mathbb{E}\left[|X_{nm+l}^{\epsilon, N, M}|^2 \right] + \sum_{i=1}^{n}\mathbb{E}\left[|X_{im}^{\epsilon, N, M}|^2 \right]\right) \\
		\leq& (1+(1+4C_0 \gamma^\epsilon+4K_3h)h)\mathbb{E}\left[|X_{nm+l}^{\epsilon, N, M}|^2\right] \\
		&+ (4h C_0 \gamma^\epsilon+4 K_3 h^2)\left(1+ \sum_{i=1}^n  \mathbb{E}\left[|X_{im}^{\epsilon, N, M}|^2\right] \right) + dh.
	\end{align}
	According to Gronwall inequality, one can arrive at
	\begin{align}
		& \max_{0 \leq l \leq m} \mathbb{E}\left[|X_{nm+l}^{\epsilon, N, M}|^2\right] \leq \exp(1+4C_0 \gamma^\epsilon+4K_3h) \times \\
		& \qquad \left[\mathbb{E}\left[|X_{nm}^{\epsilon, N, M}|^2\right] + \left( (4C_0 \gamma^\epsilon+4K_3h)\left(1+ \sum_{i=1}^n  \mathbb{E}\left[|X_{im}^{\epsilon, N, M}|^2\right] \right) + d \right)/(1+4C_0 \gamma^\epsilon+4K_3h)  \right].
	\end{align}
	Let \(M_{n}^N = 1+ \sum\limits_{i=0}^{n-1} \max\limits_{0\leq l \leq m} \mathbb{E}\left[ |X_{im+l}^{\epsilon, N, M}|^2 \right]\), then
	\begin{align}
		M_{n+1}^N = M_n^N + \max_{0 \leq l \leq m} \mathbb{E}\left[|X_{nm+l}^{\epsilon, N, M}|^2\right] &\leq M_n^N  + \exp(1+4C_0 \gamma^\epsilon+4K_3h)\left[\mathbb{E}\left[|X_{nm}^{\epsilon, N, M}|^2\right] + M_n^N+1\right]\\
		&\leq \exp(1+4C_0 \gamma^\epsilon+4K_3h)\left[3M_n^N+1\right]\\
		&\leq \Lambda \left( M_n^N + 1\right),
	\end{align}
	where \(\Lambda = 3 \exp \left(1+4C_0 \gamma^\epsilon+4K_3h\right) \). Iterating the above recursive inequality gives
	\begin{align}
		M_n^N = 1+ \sum\limits_{i=0}^{n-1} \max\limits_{0\leq l \leq m} \mathbb{E}\left[ |X_{im+l}^{\epsilon, N, M}|^2 \right] \leq \Lambda^{n-1}\left(\frac{\Lambda}{\Lambda-1} + M_1\right) < 2 \Lambda^n,
	\end{align}
	for $M_1 = 1+ \max\limits_{0\leq l \leq m} \mathbb{E}\left[ |X_{l}^{\epsilon, N, M}|^2 \right] \leq  2 \exp \left(1+4C_0 \gamma^\epsilon+4K_3 h \right) +1 \leq  \Lambda $.
\end{proof}
\begin{rem}
	The use of Theorem~\ref{drift_error} at the random numerical states
	\(\bigl( X_{nm+l}^{\epsilon,N,M}, X_{m:nm}^{\epsilon,N,M} \bigr)\)
	requires the samples used in the current drift approximation to be independent of these states. 
	Thus, in the theoretical analysis, for each \(n=0,\ldots,N-1\) and \(l=0,\ldots,m-1\), the two sample sets used to approximate
	\[
	\alpha^{\epsilon,M}
	\bigl(t_n+lh, X_{nm+l}^{\epsilon,N,M};
	X_{m:nm}^{\epsilon,N,M}\bigr)
	\]
	are taken independently of the previous samples and of the Brownian increments up to time \(t_n+lh\). 
	This allows us to apply Theorem~\ref{drift_error} conditionally on
	\(\bigl( X_{nm+l}^{\epsilon,N,M}, X_{m:nm}^{\epsilon,N,M} \bigr)\).
	In practical computations, however, we reuse the same two sample sets
	\(\{X_{t_{1:N}}^{i}\}_{i=1}^{M}\) and
	\(\{\widetilde X_{t_{1:N}}^{i}\}_{i=1}^{M}\)
	throughout the Euler iteration. 
	This sample reuse is not covered by the present proof, but it does not affect the convergence behavior observed in our numerical experiments.
\end{rem}

Note that Lemmas \ref{exact_solution_bound} and \ref{num_solution_bound} provide the second-moment estimates for the exact solution \(X^\epsilon\) of the regularized SBTS diffusion \eqref{sde2} and the numerical solution \(X^{\epsilon,N,M}\) generated by \eqref{sde3}, respectively. Before stating the main error estimate, we recall that the prescribed observation grid is given by
\[
t_n=n,\qquad n=1,\ldots,N.
\]
For the Euler discretization, each interval \([t_n,t_{n+1}]\) is divided into \(m\) uniform subintervals with stepsize \(h=1/m\). We denote the resulting subgrid points by
\[
t_n+lh,\qquad l=0,1,\ldots,m-1.
\]
Accordingly, \(X_{nm+l}^{\epsilon,N,M}\) denotes the Euler approximation at the subgrid point \(t_n+lh\), while \(X_{nm}^{\epsilon,N,M}\) denotes the approximation at the observation time \(t_n\). 
\begin{thm}\label{discretization_error}
	Assume all the assumptions in Lemma \ref{exact_solution_bound} and \ref{num_solution_bound} are true. For \(i=1,\ldots,N\), define
	\(\Delta X_{im}^{\epsilon} := X_{t_i}^{\epsilon} - X_{im}^{\epsilon,N,M}\). Then the accumulated mean square error at the prescribed observation time grid satisfies
	\begin{align}
		\sum\limits_{i=1}^N \mathbb{E}[|\Delta X_{im}^{\epsilon}|^2] \leq \left(h + \Theta (\epsilon, h) + 12K_3 \Lambda^{N}h\right) 3^N \exp (N(1+6\gamma^\epsilon)),
	\end{align}
	where $\Theta(\epsilon, h)$ is defined in Lemma \ref{exact_solution_bound} by taking \(n=N\).
\end{thm}

\begin{proof}
	We take difference between the following two sets of equations
	\begin{align}
		X_{t_{n}+(l+1)h}^{\epsilon} &= X_{t_n+lh}^{\epsilon} + \int_{t_n+lh}^{t_{n}+(l+1)h} \alpha^\epsilon(t, X_t^\epsilon; X_{t_1:t_n}^\epsilon) \mathrm{d}t + \int_{t_n+lh}^{t_{n}+(l+1)h} \mathrm{d}W_s, \\
		X_{nm+l+1}^{\epsilon, N, M} &= X_{nm+l}^{\epsilon, N, M} + \alpha^{\epsilon, M}(t_n+lh, X_{nm+l}^{\epsilon, N, M}; X_{m:nm}^{\epsilon, N, M}) h + \Delta W_{nm+l}.
	\end{align}
	Let $\Delta X_{nm+l}^{\epsilon} =  X_{t_{n}+lh}^{\epsilon} - X_{nm+l}^{\epsilon, N, M}$, using Young's inequality, Jensen's inequality, and \(h\le 1\), we obtain
	\begin{align}
		\mathbb{E}[|\Delta X_{nm+l+1}^{\epsilon}|^2] \leq& (1+h)\mathbb{E}[|\Delta X_{nm+l}^{\epsilon}|^2]  +  h\left(1+\frac{1}{h}\right)\mathbb{E}\Bigg[\int_{t_n+lh}^{t_{n}+(l+1)h}|\alpha^\epsilon(t, X_t^\epsilon; X_{t_1:t_n}^\epsilon) - \alpha^\epsilon(t_n+lh, X_{t_n+lh}^\epsilon; X_{t_1:t_n}^\epsilon)  \\
		& + \alpha^\epsilon(t_n+lh, X_{t_n+lh}^\epsilon; X_{t_1:t_n}^\epsilon) - \alpha^\epsilon(t_n+lh, X_{nm+l}^{\epsilon, N, M}; X_{m:nm}^{\epsilon, N, M})\\
		& + \alpha^\epsilon(t_n+lh, X_{nm+l}^{\epsilon, N, M}; X_{m:nm}^{\epsilon, N, M}) - \alpha^{\epsilon, M}(t_n+lh, X_{nm+l}^{\epsilon, N, M}; X_{m:nm}^{\epsilon, N, M})|^2 \mathrm{d} t\Bigg]\\
		\leq& (1+h)\mathbb{E}[|\Delta X_{nm+l}^{\epsilon}|^2]  +  6\mathbb{E}\Bigg[\int_{t_n+lh}^{t_{n}+(l+1)h} |\alpha^\epsilon(t, X_t^\epsilon; X_{t_1:t_n}^\epsilon) - \alpha^\epsilon(t_n+lh, X_{t_n+lh}^\epsilon; X_{t_1:t_n}^\epsilon) |^2 \mathrm{d} t\\
		& + h|\alpha^\epsilon(t_n+lh, X_{t_n+lh}^\epsilon; X_{t_1:t_n}^\epsilon) - \alpha^\epsilon(t_n+lh, X_{nm+l}^{\epsilon, N, M}; X_{m:nm}^{\epsilon, N, M})|^2\\
		& + h|\alpha^\epsilon(t_n+lh, X_{nm+l}^{\epsilon, N, M}; X_{m:nm}^{\epsilon, N, M}) - \alpha^{\epsilon, M}(t_n+lh, X_{nm+l}^{\epsilon, N, M}; X_{m:nm}^{\epsilon, N, M})|^2 \Bigg]\\
		\leq& (1+h)\mathbb{E}[|\Delta X_{nm+l}^{\epsilon}|^2]  +  6\mathbb{E}\Bigg[\int_{t_n+lh}^{t_{n}+(l+1)h} \gamma^\epsilon \left(|t-(t_n+lh)|+|X_t^\epsilon-X_{t_n+lh}^\epsilon|^2\right) \mathrm{d} t\\
		& + \gamma^\epsilon h \left(|\Delta X_{nm+l}^\epsilon|^2 +\sum_{i=1}^{n}|\Delta X_{im}^\epsilon|^2\right) + K_3 h^2 \left(1+ |X_{nm+l}^{\epsilon, N, M}|^2+ \sum_{i=1}^{n}|X_{im}^{\epsilon, N, M}|^2 \right) \Bigg]\\
		\leq& (1+h)\mathbb{E}[|\Delta X_{nm+l}^{\epsilon}|^2] + 6h\left(\gamma^\epsilon\left(h + \Theta(\epsilon, h) \right) + \gamma^\epsilon\mathbb{E}[|\Delta X_{nm+l}^{\epsilon}|^2] + \gamma^\epsilon E_n + K_3 M_{n+1}^N h\right)\\
		\leq& \big(1+(1+6\gamma^\epsilon)h\big)\mathbb{E}[|\Delta X_{nm+l}^{\epsilon}|^2] + 6h \left(\gamma^\epsilon h + \gamma^\epsilon \Theta(\epsilon, h) + \gamma^\epsilon E_n + 2K_3 \Lambda^{n+1} h \right),
	\end{align}			
	where $E_n := \sum\limits_{i=1}^{n} \mathbb{E}\left[ |\Delta X_{im}^\epsilon|^2\right]$. By Gronwall inequality, we have
	\begin{align}
		\max_{0 \leq l \leq m} \mathbb{E}[|\Delta X_{nm+l}^{\epsilon}|^2] &\leq \exp(1+6\gamma^\epsilon) \left( \mathbb{E}[|\Delta X_{nm}^{\epsilon}|^2] + \left(6\gamma^\epsilon h + 6\gamma^\epsilon \Theta(\epsilon, h)+6\gamma^\epsilon E_n + 12K_3  \Lambda^{n+1} h \right)/(1+6\gamma^\epsilon) \right)\\
		&\leq \exp(1+6\gamma^\epsilon) \left( \mathbb{E}[|\Delta X_{nm}^{\epsilon}|^2] +  E_n + h + \Theta(\epsilon, h)+ 12K_3 \Lambda^{N} h \right). 
	\end{align}
	Iterating this recursion from \(E_0=0\), we obtain 
	\begin{align}
		E_{n+1} = E_n + \mathbb{E}[|\Delta X_{(n+1)m}^{\epsilon}|^2] &\leq E_n + \exp(1+6\gamma^\epsilon) \left( \mathbb{E}[|\Delta X_{nm}^{\epsilon}|^2] +  E_n + h + \Theta(\epsilon, h)+ 12K_3 \Lambda^{N}h \right)\\
		&\leq \exp(1+6\gamma^\epsilon) \left( 3E_n + h + \Theta(\epsilon, h)+ 12K_3 \Lambda^{N} h \right)\\
		&\leq \left(h + \Theta (\epsilon, h) + 12K_3 \Lambda^{N} h \right) \exp (1+6\gamma^\epsilon) \sum\limits_{i=0}^{n} \left(3 \exp (1+6\gamma^\epsilon)\right)^i \\
		&\leq \left(h + \Theta (\epsilon, h) + 12K_3 \Lambda^{N} h \right) 3^{n+1} \exp ((n+1)(1+6\gamma^\epsilon)).
	\end{align}		
	Taking \(n=N-1\) gives
	\begin{align}
		E_N = \sum\limits_{i=1}^N \mathbb{E}[|\Delta X_{im}^{\epsilon}|^2] \leq \left(h + \Theta (\epsilon, h) + 12K_3 \Lambda^{N} h \right) 3^N \exp (N(1+6\gamma^\epsilon)).
	\end{align}
	The proof is complete.
\end{proof}

\subsection{Convergence analysis of the numerical time series distribution}
The following theorem gives the final error estimate between the joint law generated by the numerical scheme and the original target time series law
\(\mu\).

\begin{thm}\label{wasserstein_error}
	Assume that Assumptions \ref{asl}--\ref{asg} and \ref{kernel} hold. Suppose
	also that \(\mu,\mu^G\in\mathcal P_2((\mathbb R^d)^N)\). Let
	\[
	\mu^{\epsilon,N,M}
	:=
	\mathcal L\bigl(X_m^{\epsilon,N,M},X_{2m}^{\epsilon,N,M},
	\ldots,X_{Nm}^{\epsilon,N,M}\bigr)
	\]
	be the joint law generated by the Euler approximation with the
	regularized kernel estimated drift. If, on each interval
	\([t_n,t_{n+1})\), the bandwidth and the sample size are chosen as
	\(
	H\sim \mathcal O(h^{1/2}), 	M\sim \mathcal O\bigl(h^{-(n+1)d-2}\bigr).
	\)
	Then
	\[
	\mathcal{W}_2^2\bigl(\mu^{\epsilon,N,M},\mu\bigr)
	\sim \mathcal{O}(h) + \mathcal{O}(\epsilon) .
	\]
\end{thm}

\begin{proof}
	We divide the proof into two steps.
	
	\medskip
	\noindent
	\textit{Step 1: Error caused by the regularization of the target law.}
	Let \(Z\sim\mu\), \(Y\sim\mu^G\), and let \(\theta\) be a Bernoulli random
	variable independent of \(Z\) and \(Y\), such that \(	\mathbb P(\theta=1)=1-\epsilon, 
	\mathbb P(\theta=0)=\epsilon.\)
	Define
	\[
	\widetilde Z:=\theta Z+(1-\theta)Y .
	\]
	Then
	\[
	\mathcal L(\widetilde Z)
	=
	(1-\epsilon)\mu+\epsilon\mu^G
	=
	\mu^\epsilon .
	\]
	Thus \((Z,\widetilde Z)\) is a coupling of \(\mu\) and \(\mu^\epsilon\).
	By the definition of the \(L^2\)-Wasserstein distance,
	\[
	\begin{aligned}
		\mathcal W_2^2(\mu^\epsilon,\mu)
		\le&
		\mathbb E\bigl[|\widetilde Z-Z|^2\bigr] = \mathbb{E}[|(\theta Z + (1 - \theta)Y) - Z|^2] \\
		\leq& \mathbb{E}[|(\theta Z + (1 - \theta)X) - Z|  \theta = 1] \mathbb{P}(\theta = 1) \\
		&+ \mathbb{E}[|(\theta Z + (1 - \theta)X) - Z|  \theta = 0] \mathbb{P}(\theta = 0) \\
		=&
		\epsilon\mathbb E\bigl[|Y-Z|^2\bigr].
	\end{aligned}
	\]
	Since \(\mu,\mu^G\in\mathcal P_2((\mathbb R^d)^N)\), there exists a constant \(C_3>0\), independent of
	\(\epsilon\), such that
	\[
	\mathcal{W}_2^2(\mu^\epsilon,\mu)
	\le C_3\epsilon .
	\]
	
	\medskip
	\noindent
	\textit{Step 2: Error caused by the kernel estimation and Euler discretization.}
	Since
	\(\mathcal  L(X^\epsilon_{t_1},\ldots,X^\epsilon_{t_N}) = \mu^\epsilon,\) the pair
	\(
	\left(
	(X_m^{\epsilon,N,M},\ldots,X_{Nm}^{\epsilon,N,M}),
	(X^\epsilon_{t_1},\ldots,X^\epsilon_{t_N})
	\right)
	\)
	is a coupling of \(\mu^{\epsilon,N,M}\) and \(\mu^\epsilon\). Therefore, 
	\[
	\begin{aligned}
		\mathcal W_2^2\bigl(\mu^{\epsilon,N,M},\mu^\epsilon\bigr) \le \mathbb E\left[ \bigl|X_{m:Nm}^{\epsilon,N,M}-X^\epsilon_{t_1:t_N}\bigr|^2 \right] = 
		\sum_{n=1}^N \mathbb E\left[ \bigl|X_{nm}^{\epsilon,N,M}-X^\epsilon_{t_n}\bigr|^2 \right].
	\end{aligned}
	\]
	By the strong error estimate established in Theorem \ref{discretization_error}, under the choices
	\(
	H\sim \mathcal O(h^{1/2}), 
	M\sim \mathcal O\bigl(h^{-(n+1)d-2}\bigr),
	\)
	there exists a constant \(C_4>0\), independent of \(h,H,M\), such that
	\[
	\sum_{i=1}^N \mathbb E\left[ \bigl|X_{nm}^{\epsilon,N,M}-X^\epsilon_{t_n}\bigr|^2 \right] \le \left(h + \Theta (\epsilon, h) + 12K_3 \Lambda^{N}h\right) 3^N \exp (N(1+6\gamma^\epsilon)) \le C_4 h.
	\]
	Consequently,
	\[
	\mathcal W_2^2\bigl(\mu^{\epsilon,N,M},\mu^\epsilon\bigr)
	\le C_4 h.
	\]
	\medskip
	\noindent
	Combining the two errors. The proof is complete.
\end{proof}


\section{Numerical experiments}

\begin{example}
We consider the following GARCH Model: 
\begin{align}
    \begin{cases}
        X_{t_{n+1}} = \sigma_{t_{n+1}} \epsilon_{n+1} \\ 
        \sigma^2_{t_{n+1}} = \alpha_0 + \alpha_1 X^2_{t_n}+ \alpha_2 X^2_{t_{n-1}}, n=1,...N.
    \end{cases}
\end{align}
In this experiment, we take only one intermediate observation time, so that each path consists of three time points, that is $ 0=t_0 < t_1 < t_2 = T$ with $T=1.0$. To investigate the convergence behavior of the numerical sampler, we vary the number of Euler steps used on each observation interval \([t_n,t_{n+1}]\). Specifically, we take $m \in \{2,4,8,16,24\}.$ For each value of $m$, the sample size and kernel bandwidth are chosen as $M = \left\lfloor 0.1 m^{(1+1)\times 1 + 2} \right\rfloor = \left\lfloor 0.1 m^4 \right\rfloor$, and $H = 0.6 m^{-1/2}$. For the model parameters, we let $\epsilon_{t_n} \sim \mathcal{N}(0, 0.1)$, $\alpha_0=5.0, \alpha_1=0.4, \alpha_2=0.1$. 

In the left figure in Figure \ref{fig: GARCH_n1d1}, we study the convergence behavior of the designed algorithm under different regularization parameter $\epsilon$. We note that a direct empirical evaluation of the strong $L^2$ error is generally difficult, since it requires the exact solution of the underlying continuous-time SDE as a reference. For the data-generating models considered here, such a closed-form solution is typically unavailable. We therefore evaluate the generated samples through distributional metrics instead of path-wise errors. Specifically, we compute the empirical Wasserstein-2 distance
\[
\mathcal{W}_2\left(
    \frac{1}{M}\sum_{i=1}^M \delta_{(X^i_{t_1},X^i_{t_2})},
    \frac{1}{M}\sum_{i=1}^M \delta_{(\widetilde X^i_{t_1},\widetilde X^i_{t_2})}
\right),
\]
where $(X^i_{t_1},X^i_{t_2})$ denotes samples from the reference data-generating model and $(\widetilde X^i_{t_1},\widetilde X^i_{t_2})$ denotes samples generated by the numerical method. This provides a direct comparison of the joint law of the time series, rather than only comparing the marginal distributions at each time point.

This metric is also naturally related to strong error estimates. Indeed, for any coupling $(Y_1,Y_2)$ with marginal laws $\mu_1$ and $\mu_2$, one has
\[
    \mathcal{W}_2(\mu_1,\mu_2)
    \leq
    \left(\mathbb E |Y_1-Y_2|^2\right)^{1/2}.
\]
Therefore, whenever a path-wise strong $L^2$ comparison is available, it provides an upper bound for the corresponding Wasserstein-2 distance. In this sense, the Wasserstein-2 error is a weaker but natural distributional diagnostic for assessing the convergence of the generated time-series law. In our experiment, we compute numerically the $\mathbb{W}_2$ loss by using total $M^{\text{Gen}}=6000$ for the data generated by our scheme and the true labeled data.

In Figure~\ref{fig: GARCH_n1d1}, we investigate the convergence behavior of the proposed method for several choices of the regularization parameter,
$\epsilon \in \{0.2,0.15,0.1,0.05\}.$
For each value of $\epsilon$, the experiment is repeated 30 times, and the reported $\mathcal{W}_2$ error is computed as the average over these independent runs. When the sample size is chosen according to the scaling suggested by our analysis,
$M \sim \mathcal{O}(N^4)$
the observed decay is faster than the theoretical reference rate and is close to first order over the tested range of $N$. This behavior is not inconsistent with the theory, since the theoretical estimate provides a worse-case convergence guarantee rather than a sharp asymptotic characterization. The numerical result therefore suggests that the practical convergence rate can be better than the conservative rate guaranteed by the analysis.

The right panel of Figure~\ref{fig: GARCH_n1d1} shows the dependence of the empirical $\mathcal{W}_2$ error on the regularization parameter. As $\epsilon$ decreases, the empirical law generated by the regularized scheme moves closer to the empirical law of the original labeled data. This is consistent with the interpretation that the Brownian mixture regularization introduces an approximation error which vanishes as $\epsilon \to 0$.
\begin{figure}[!htb]
\center
    \includegraphics[width=1.0\textwidth]{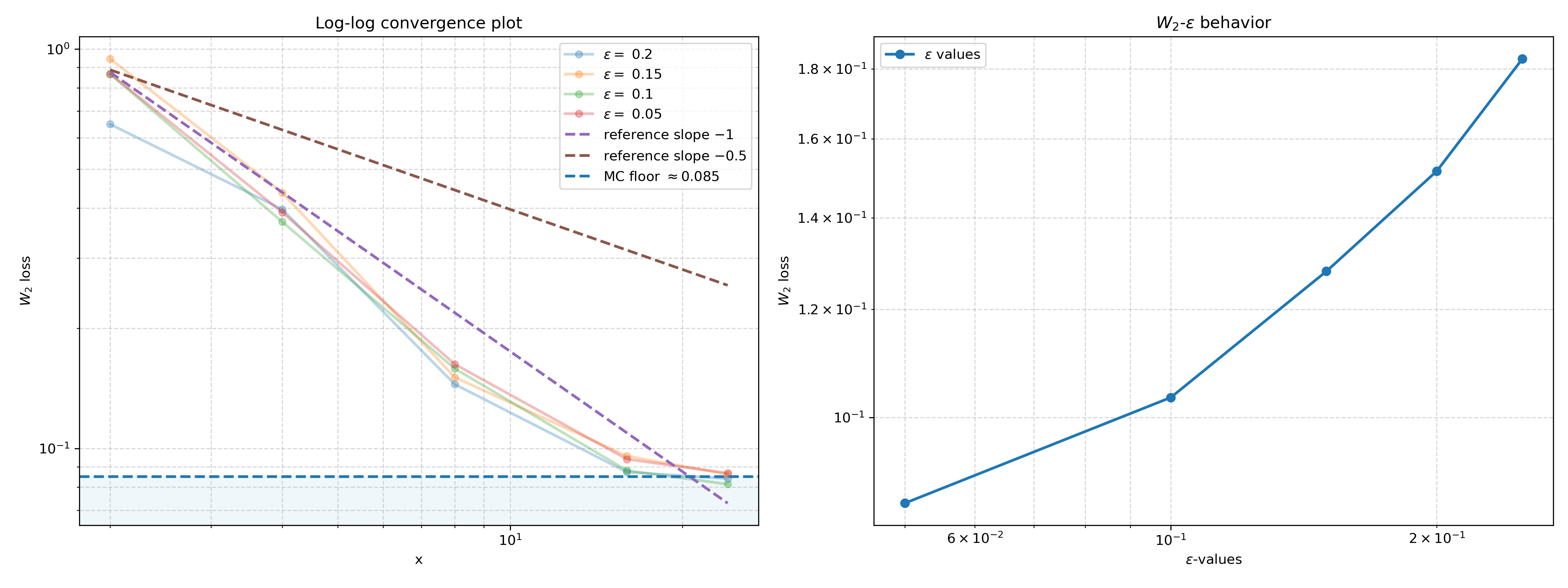}
    \caption{Convergence study rate study under different regularization parameter with given data sample size $M \sim \mathcal{O}(m^4)$. Left: First order decay of the algorithm for $m=2,4,8,16,24$. Right: Convergence with respect to regularization parameter $\epsilon$.}
    \label{fig: GARCH_n1d1}
\end{figure}

To investigate the effect of the kernel bandwidth $H$ on the convergence behavior, we compare several choices of $H$, namely $H=0.6m^{-1}$, $H=0.6m^{-1/2}$, and fixed bandwidths $H\in\{0.6,0.9,1.2\}$. The experiment is performed with regularization parameter $\epsilon=0.05$ and sample-size scaling $M\sim \mathcal{O}(m^4)$, where $m\in\{2,4,8,16,24\}$. The left panel of Figure~\ref{fig: hComparison_GARCH_n1d1} shows that, when the kernel bandwidth decreases as $m$ increases, the empirical $\mathcal{W}_2$ error rapidly approaches the Monte Carlo sampling floor. Moreover, over the tested range of $m$, there is no significant difference between the choices $H=0.6m^{-1/2}$ and $H=0.6m^{-1}$ in terms of the observed decay rate. In contrast, when the bandwidth is kept fixed at $H=0.6,0.9,$ or $1.2$, the computed error eventually reaches a plateau, suggesting that further refinement of the time discretization or increase of the ensemble size does not lead to a visible improvement in accuracy. This indicates that a fixed bandwidth introduces a persistent kernel approximation error. On the right panel of the figure, the computed $\mathcal{W}_2$ error decreases with decreasing size of $H=1.2,0.9,0.6, 0.122,  0.025$ for $M=\lfloor 0.1 (\frac{1}{24})^5  \rfloor$ according to the experiment setup. Although in practical implementations one may choose $H$ sufficiently small, a fixed choice of $H$ remains arbitrary and may limit the achievable accuracy. The numerical results therefore suggest that the kernel bandwidth should decrease jointly with the time step, which is consistent with the scaling requirement in our theoretical analysis.

\begin{figure}[!htb]
\center
    \includegraphics[width=1.0\textwidth]{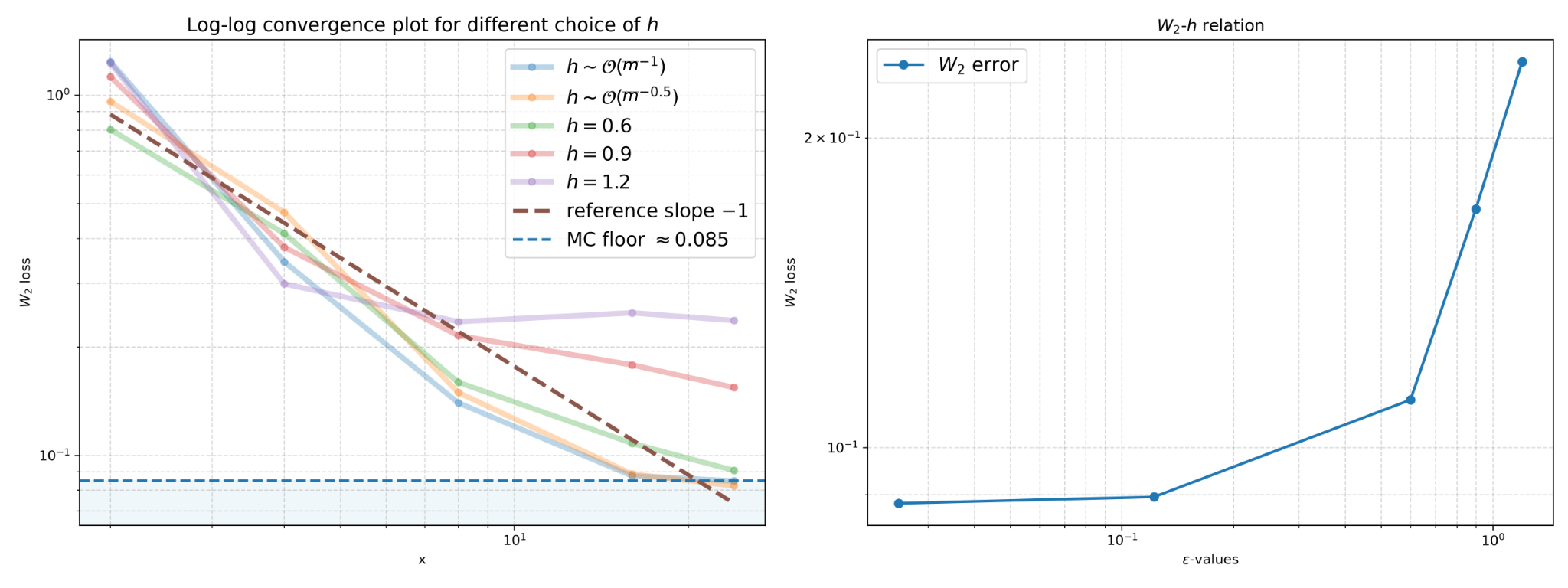}
    \caption{Comparison between the $\mathbb{W}_2$ error decay and achieved accuracy for different kernel bandwidth $H$ with $M \sim \mathcal{O}(m^5)$. Left: Error comparison for $H=0.6N^{-1/2}$, $H=0.6m^{-1}$, and fixed bandwidths $H\in\{0.6,0.9,1.2\}$. Right: The achieved accuracy for $M=33178$, $H=0.025,0.122,0.6,0.9,1.2$.}
    \label{fig: hComparison_GARCH_n1d1}
\end{figure}

To further investigate the role of the sample-size scaling in the convergence behavior, we repeat the experiment with
$M \sim \mathcal{O}(m^2),$ using $m \in \{2,4,8,16,32,64\},$
and keeping all other experimental parameters unchanged. Under this reduced sample-size scaling, the empirical $W_2$ error decays at a rate close to $0.5$. Compared with the faster decay observed under the scaling $M \sim \mathcal{O}(m^4)$, this result illustrates the impact of the ensemble-size error on the overall convergence behavior. The right panel of Figure~\ref{fig: GARCH_n1d1M2} further displays the effect of the regularization parameter. As $\epsilon$ decreases, the empirical $\mathcal{W}_2$ error exhibits an overall decreasing trend. This agrees with the expected behavior of the regularized scheme, since the approximation error introduced by the Brownian mixture decreases as $\epsilon$ is reduced.
\begin{figure}[!htb]
\center
    \includegraphics[width=1.0\textwidth]{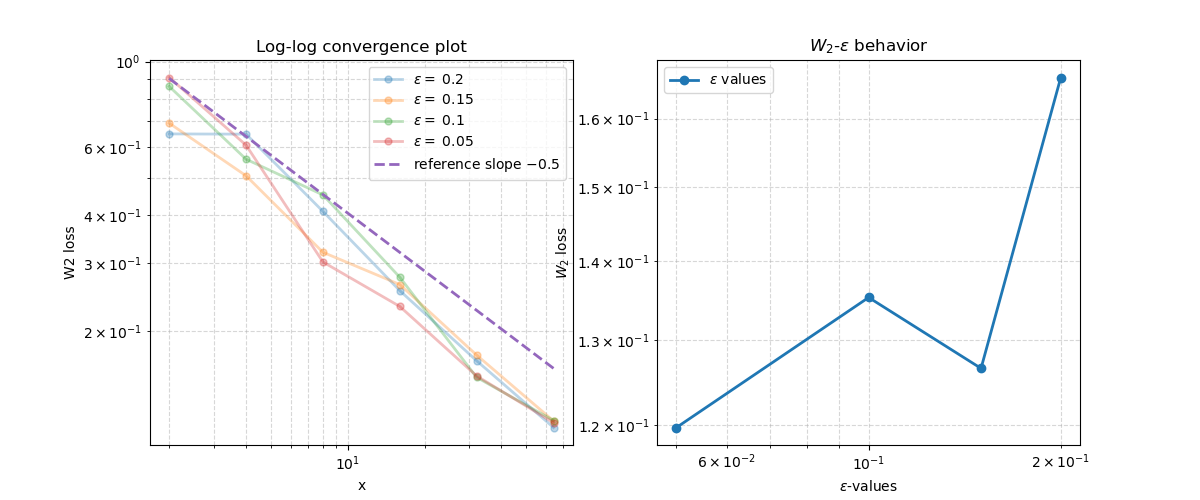}
    \caption{Convergence study rate study under different regularization parameter with given data sample size $M \sim \mathcal{O}(m^2)$.}
    \label{fig: GARCH_n1d1M2}
\end{figure}
\end{example}
\begin{example}

For this example, we retain the one-dimensional GARCH setup from Example~1, but modify the parameters to $\alpha_0=2.5, \alpha_1=\alpha_2=1.0.$
This choice increases the influence of the previous observations on the current conditional volatility. We also increase the number of observation points by adding one additional intermediate time point, thereby producing a longer time series. Accordingly, in this experiment we choose the sample-size scaling
$M \sim \mathcal{O}(m^5),$ and take
$ m \in \{2,4,6,8,12\}.$ The corresponding sample sizes are $M \in \{10,307,2333,9830,74650\}.$ As in Example~1, for each value of the regularization parameter $\epsilon$, the experiment is repeated 30 times, and the reported $\mathcal{W}_2$ error is computed as the average over these independent runs.

The left panel of Figure~\ref{fig: 1d2pts_GARCH} shows the convergence behavior under different choices of $\epsilon$. Due to the high-order sample-size scaling $M \sim \mathcal{O}(m^5)$, the empirical $\mathcal{W}_2$ error decreases rapidly and soon approaches the Monte Carlo sampling floor, which is indicated by the horizontal blue dashed line. Before reaching this floor, the log-log plot shows a decay that is faster than the reference rate $-0.5$. The right panel of Figure~\ref{fig: 1d2pts_GARCH} illustrates the dependence of the empirical $\mathcal{W}_2$ error on the regularization parameter. As $\epsilon$ decreases, the observed $\mathcal{W}_2$ loss exhibits an overall decreasing trend, which is consistent with the expectation that the approximation error introduced by the Brownian mixture regularization vanishes as $\epsilon \to 0$.

\begin{figure}[!htb]
\center
    \includegraphics[width=1.0\textwidth]{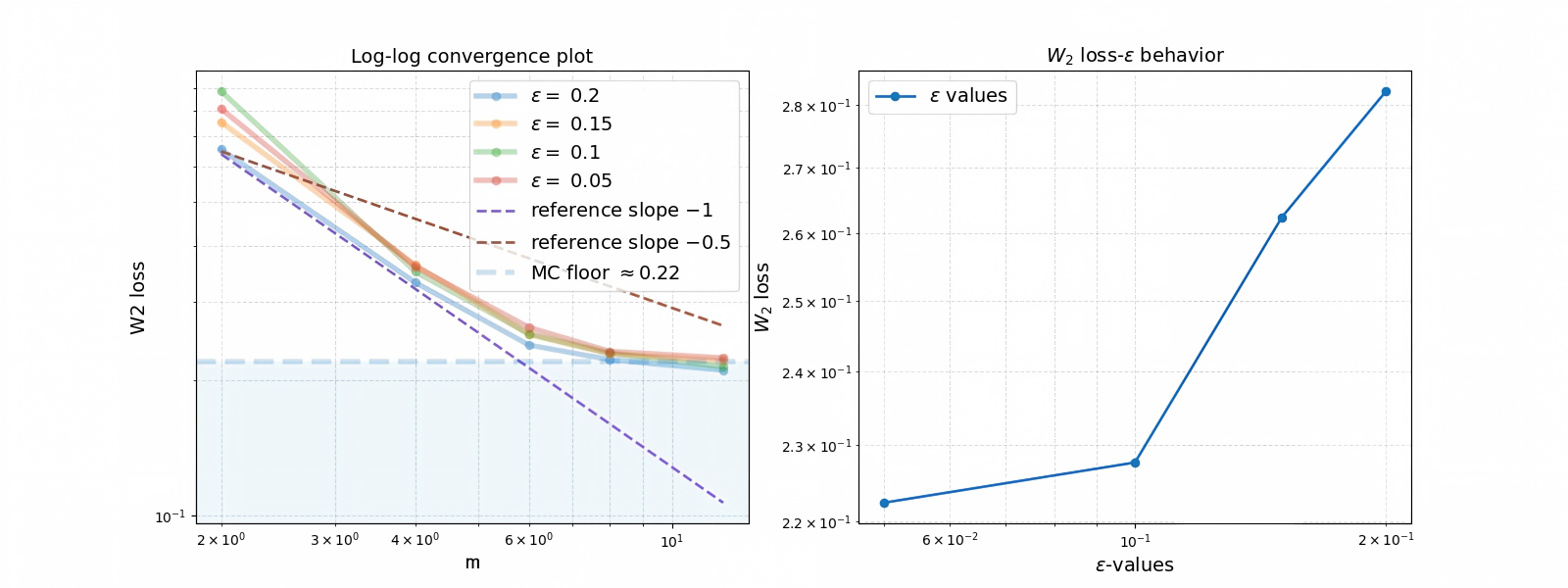}
    \caption{Convergence rate study for a longer timeseries data generation ($N=3$) under different regularization parameter with given data sample size $M \sim \mathcal{O}(m^5)$. Left: First order decay of the algorithm for $m=2,4,6, 8,12$. Right: Convergence with respect to regularization parameter $\epsilon$.}
    \label{fig: 1d2pts_GARCH}
\end{figure}
    
\end{example}

\begin{example}
For this example, we consider the following two-dimensional GARCH model. 
\begin{align}
    \begin{cases}
        X_{t_{n+1}} &=D_{n+1} L \xi_{n+1} \\ 
        D_{n+1} & = \text{diag}(\sigma_{n+1}^1, \sigma_{n+1}^2)
    \end{cases}
\end{align}
where $\xi_{n+1} \sim \mathcal{N}(0,I_2)$ and
\begin{align}
    \begin{pmatrix}
        (\sigma^{1}_{t_{n+1}})^2 \\
        (\sigma^{2}_{t_{n+1}})^2
    \end{pmatrix}=
    \alpha_0+ 
    A\begin{pmatrix}
        (X^{1}_{t_n})^2 \\
        (X^{2}_{t_n})^2
    \end{pmatrix}
    +B 
    \begin{pmatrix}
        (X^{1}_{t_{n-1}})^2 \\
        (X^{2}_{t_{n-1}})^2
    \end{pmatrix}.
\end{align}
In this example, we take
\begin{align}
    L=\begin{pmatrix}
        1, \ \ \ \ 0 \\
        \rho, \sqrt{1-\rho^2}
    \end{pmatrix}, \quad
A=\begin{pmatrix}
        0.35, 0.05 \\
        0.05, 0.35
    \end{pmatrix}, \quad
B=\begin{pmatrix}
        0.08, 0.02 \\
        0.02, 0.08
    \end{pmatrix}   
\end{align}
with $\rho=0.5$. For this example, we consider a two-dimensional time series observed at three time points, starting from the deterministic initial condition $X_{t_0}=0$. Since the state dimension is $d=2$, the scaling prescribed by our analysis gives $M \sim \mathcal{O}(m^6).$
When computing the $W_2$ loss, we flatten the 2-D timeseries to a 1-D vector while acknowledging that the joint distributional structure is still preserved. 

In Figure \ref{fig: 2d_GARCH_n1d2}, acknowledging that we need to take $M\sim \mathcal{O}(m^6)$ according to our analysis, we compute the $\mathcal{W}_2$ loss under various $\epsilon$ for only $m=2,4,6,8,10$ due to limited computational capacity. In the figure on the left, it was noted that the computed $\mathcal{W}_2$ error quickly reaches the Monte Carlo floor when increasing $m$ from $2$ to $8$. The slope obtained is again closer to 1 rather than 0.5. Interestingly, the figure for $\mathcal{W}_2$ loss also potentially presents an exponential decay: 
\begin{align}
    \mathcal{W}_2(m) \approx \mathcal{W}_{2,\text{floor}} + C e^{-c m} 
\end{align}
rather than $ \mathcal{W}_2(N) \approx \mathcal{W}_{2,\text{floor}} + C m^{-p}$
where $\mathcal{W}_{2,\text{floor}}$ is the Monte Carlo floor. Nevertheless, since the range of discretization levels is limited and the error floor is estimated empirically, we do not claim exponential convergence.  Instead, the experiment indicates that the theoretical estimate is conservative for this example, and that the practical convergence behavior may be better than the guaranteed rate.

On the left of Figure \ref{fig: 2d_GARCH_n1d2}, we again compare the empirical distribution formulated by the $M$ samples generated by our scheme under various regularization parameters $\epsilon$ and the distribution formulated by the true labeled data. And on the right panel of Figure \ref{fig: 2d_GARCH_n1d2}, it is in general observed the $\mathcal{W}_2$ loss decreases with decreasing $\epsilon$ as expected.

\begin{figure}[!htb]
\center
    \includegraphics[width=1.0\textwidth]{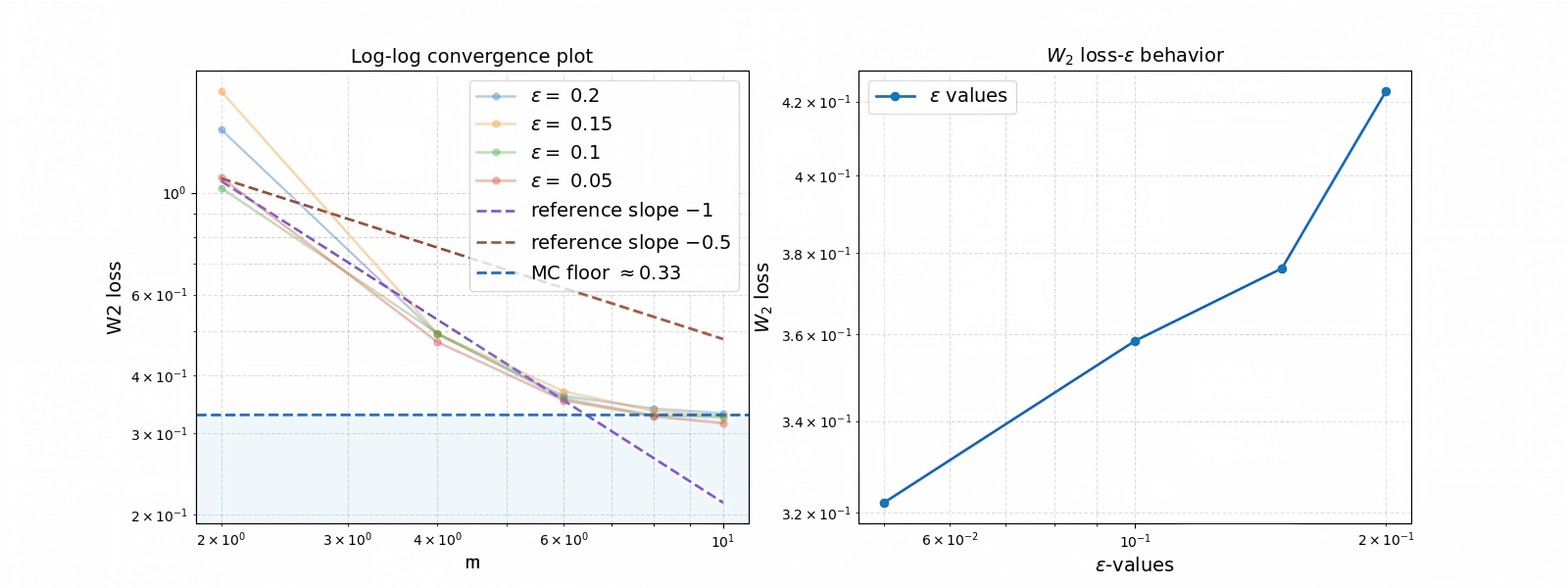}
    \caption{Convergence rate study for a 2-D GARCH example under different regularization parameter with given data sample size $M \sim \mathcal{O}(m^6)$. Left: First order decay of the algorithm for $m=2,4,6,8,10$. Right: Convergence with respect to regularization parameter $\epsilon$.}
    \label{fig: 2d_GARCH_n1d2}
\end{figure}
To further examine the relationship between $m$ and $M$ and their joint impact on the convergence rate, we now reduce the sample-size scaling to $M\sim \mathcal{O}(m^2)$. In the left panel of Figure~\ref{fig: fixedM_2d_GARCH_n1d2}, we take $m=2,4,8,16,32,64,128$, with the corresponding sample sizes $M=9,34,137,546,2184,8738$. The log-log plot shows that, under the scaling $M\sim \mathcal{O}(m^2)$, the empirical $\mathcal{W}_2$ error decays at a rate slower than one half. This suggests that, as the dimension $d$ increases, a stronger growth of the sample size $M$ with respect to $m$ is required in order to maintain a comparable convergence rate.

In the right panel of Figure~\ref{fig: fixedM_2d_GARCH_n1d2}, we further examine the behavior of the empirical $\mathcal{W}_2$ error when the ensemble size $M$ of the true labeled data is fixed. For each fixed $M$, the computed $\mathcal{W}_2$ loss first decreases as $N$ increases, but then begins to increase once $N$ becomes sufficiently large. This phenomenon is observed consistently across several choices of $M$. At the same time, for a fixed value of $N$, the empirical $W_2$ error generally decreases as $M$ increases. Importantly, for fixed $M$, the error does not simply reach a plateau after $N$ passes a certain threshold; instead, it increases as the time step is further refined, with $h\sim \mathcal{O}(1/N)$. This indicates that refining the temporal discretization without simultaneously increasing the ensemble size may deteriorate the approximation quality due to insufficient sample coverage. Therefore, the experiment highlights that, for the numerical scheme to converge, the sample size $M$ must scale appropriately with $m$. From this perspective, our theoretical requirement $M\sim \mathcal{O}(m^{(n+1)d+2})$ provides a conservative but practically meaningful condition for controlling the interaction between time discretization and ensemble-size errors.
\begin{figure}[!htb]
\center
    \includegraphics[width=1.0\textwidth]{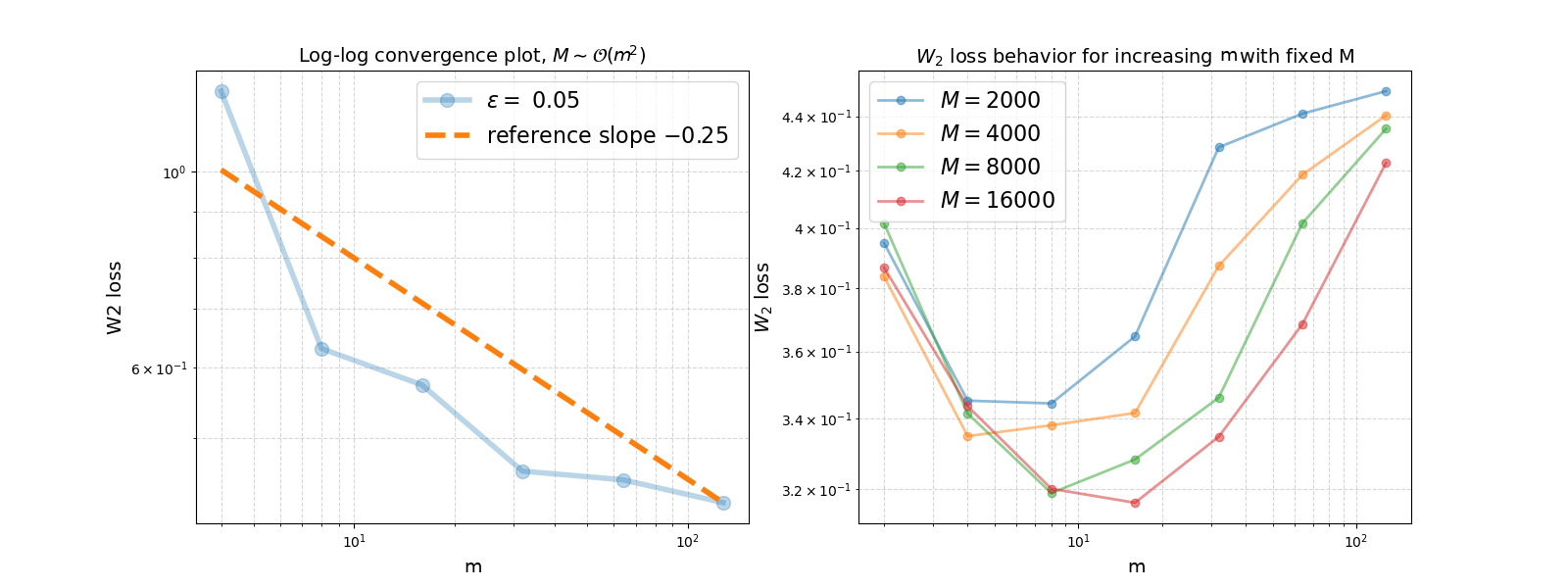}
    \caption{Convergence rate study for a 2-D GARCH example under different regularization parameter with given data sample size $M \sim \mathcal{O}(m^2)$. Left: Study of error decay of the algorithm for $m=2,4,8,16,32,64,128$. Right: Study of $\mathcal{W}_2$ error change with respect to $m$ given different ensemble data size.}
    \label{fig: fixedM_2d_GARCH_n1d2}
\end{figure}
\end{example}

\section{Further Extension}
In the preceding section, Theorem~2.1 was established under the assumption that the reference path measure is the Wiener measure on the underlying path space. In this section, we extend the result to a more general setting in which the reference measure $\mathbb{Q}$ is induced by the stochastic process
\begin{align}\label{new_reference_SDE}
\mathrm{d} X_t = b(t,X_t)\mathrm{d}t + \sigma_t\mathrm{d}W_t,
\qquad X_0=0,
\end{align}
where $\sigma_t\in\mathbb{R}^{d\times d}$ is positive deterministic in time and $b(t,x)\in\mathbb R^d$ is continuously differentiable in $t$ and twice continuously differentiable in $x$, with bounded derivatives. We establish an analogue of Theorem~2.1 for this more general reference process and numerically compare the resulting schemes under different choices of $b$ and $\sigma$.

Following the terminology of \cite{ys4}, two reference processes of particular interest are the variance-exploding (VE) and variance-preserving (VP) SDEs. These SDEs exhibit different variance behavior in their marginal distributions and constitute two standard reference dynamics in diffusion-based generative modeling. Our numerical experiments compare the performance of the proposed time-series generation algorithm under different choices of the reference path measure. The results provide further evidence of the robustness of the proposed framework with respect to the choice of the underlying reference dynamics.

We now consider the following new problem: Minimize the loss
\begin{align}
		\inf_{\alpha \in \mathcal{A}} J(\alpha):=  \frac{1}{2}\E^{\bbP}\bsbracket{\int^T_0 |\alpha_s|_{a_s^{-1}}^2 \mathrm{d} s }, \ \ \ a_t=\sigma_t\sigma_t^T
\end{align}
    under the constraint: 
\begin{align}
	\mathrm{d} X_t = \big( b(t,X_t)+ \alpha_t \big) \mathrm{d} t + \sigma_t \mathrm{d} W_t , \ \ X_0 =0, \ \ (X_{t_1},X_{t_2},...,X_{t_N}) \stackrel{\bbP} \sim \mu.
\end{align}
In the setup above, $\{W_t\}_{0\leq t\leq T}$ is a standard Brownian motion under $\mathbb{P}$, and $\mathcal{A}$ denotes the collection of progressively measurable controls $\alpha$ for which the controlled SDE is well posed,
\begin{align}
    \E^{\mathbb{P}}\left[\int_0^T |\alpha_t|_{a_t^{-1}}^2\,\mathrm{d}t\right]<\infty,
\end{align}
and the associated Girsanov change of measure is well defined.

Recall that $\mathbb{Q}$ is the path measure induced by the new SDE \eqref{new_reference_SDE} and abusing notation, we denote the joint density of $(X_{t_1},X_{t_2},...,X_{t_N})$ by $\mu_{\textbf{X}}(x_1, x_2, ..., x_N)$ and the transition probability density between $(t,x)$ and $(s,y)$ as $q(s,y|t,x)$. Assume that $\mu << \mu_{\textbf{X}}$ and that $\frac{\mu}{\mu_{\textbf{X}}} >0, \ \mu_{\textbf{X}}-a.e$ and that $D_{KL}(\mu|\mu_\mathbf{X}) < \infty$. We now prove a theorem similar to Theorem \ref{sbts_thm}. 
\begin{thm}
    The control defined by 
    \begin{align}
        \alpha(t,x; x_{1:n}):=a_t  \nabla_x \ln \E^{\mathbb{Q}}[\frac{\mu}{\mu_{\textbf{X}}}(X_{t_1},...,X_{t_N})| X_t=x, X_{t_1:t_n}=x_{1:n}]
    \end{align}
 solves the Schr{\"o}dinger bridge problem, where $a_t:= \sigma_t \sigma^T_t$, $\lambda I_d \preceq a_t \preceq \Lambda I_d$ for some $0 < \lambda < \Lambda$ and $\mu_{\textbf{X}}$ is the joint density of $(X_{t_1},...,X_{t_N})$ with $\lbrace X_{t_n} \rbrace^N_{n=1}$ being solution to: 
 \begin{align}\label{general_reference_meas}
      \mathrm{d} X_t = b(t,X_t)\mathrm{d}t + \sigma_t \mathrm{d}W_t, \ \ \ X_0=0
 \end{align}
 evaluated at $(t_1, t_2, ..., t_N)$. 
Moreover, the control can be expressed in terms of the data distribution $\mu$ in the following way: 
\begin{align}\label{final_form}
   \alpha(t,x; x_{1:n})&= a_t \frac{ \nabla \E^\mu[\frac{q(t_{n+1},X_{t_{n+1}} |t,x)}{q(t_{n+1},X_{t_{n+1}} |t_{n},x_n)} | X_{t_1 : t_n}=x_{1:n}]}{\E^\mu[\frac{q(t_{n+1},X_{t_{n+1}} |t,x)}{q(t_{n+1},X_{t_{n+1}} |t_{n},x_n)} | X_{t_1 : t_n}=x_{1:n}]}.
\end{align}
\end{thm}
\begin{proof}
 Note that the function $\frac{\mu}{\mu_{\mathbf{X}}}(x_{1},...,x_{N}) \geq 0$, and that $\E^{\mathbb{Q}}[\frac{\mu}{\mu_{\mathbf{X}}}(X_{t_1},...,X_{t_N})]=1$ by construction. We can define a probability measure $\mathbb{P}^* << \mathbb{Q}$ and the related Radon-Nikodym process as follows: for $t \in [t_n, t_{n+1}), X_t =x$
\begin{align}
    Z_t&:=\E^{\mathbb{Q}}[ \frac{\mathrm{d} \mathbb{P}^*}{\mathrm{d} \mathbb{Q}} \big| \mcal{F}_t]
    =\E^{\mathbb{Q}}[\frac{\mu}{\mu_{\textbf{X}}}(X_{t_1},...,X_{t_N})| \mcal{F}_t] \nonumber \\ 
    &=\E^{\mathbb{Q}}[\frac{\mu}{\mu_{\textbf{X}}}(X_{t_1},...,X_{t_N})| X_t=x, X_{t_1 : t_n}=x_{1:n} ] \\ 
    &=\E^{\mathbb{Q}} \Big[\frac{\mu}{\mu_{\textbf{X}}} \big(x_{1},...,x_{n}, X_{t_{n+1}},...,X_{t_{N}} \big)| X_t=x\Big] \\
    & :=h_n(t,x ; x_{1:n}), \qquad t \in [t_n, t_{n+1}), \ \ n=0,...,N-1. 
\end{align}
Then,  when $t=t_{n+1}$ for any $n=0,...,N-2$, we have by construction: 
\begin{align}
    h_n(t^{-}_{n+1},x; x_{1:n})=h_{n+1}(t_{n+1},x; x_{1:n}, x). 
\end{align}
Thus, noting that with the terminal function $h_{N-1}(t_N,x; x_{1:N-1})=\frac{\mu}{\mu _{\mathbf{X}}}(x_{1},...,x_{N-1},x)$, we have a system of equations by the Martingale property of $h_n(t,x;x_{1:n})$ on the interval $[t_n, t_{n+1})$, $n=0,...,N-1$:
\begin{equation}
\left\{
\begin{aligned}
&\Bigg( \frac{\partial}{\partial t} + b\cdot\nabla + \frac{1}{2}\sum_{l,j}
(\sigma_t\sigma_t^T)_{l,j} \frac{\partial^2}{\partial x^l\partial x^j} \Bigg)
h_n(t,x;x_{1:n})=0, \qquad \forall\, t\in[t_n,t_{n+1}), \\[0.6em]
&h_n(t_{n+1}^-,x;x_{1:n})=h_{n+1}(t_{n+1},x;x_{1:n},x), \qquad 0 \leq n\leq N-2,\\[0.6em]
&h_n(t_{n+1}^-,x;x_{1:n})=\frac{\mu}{\mu^{\mathbf X}} (x_{1:N-1},x), \qquad \text{if }n=N-1.
\end{aligned}
\right.
\end{equation}
As such, we have for $t \in [t_{n}, t_{n+1}), n=0,...,N-1$: 
\begin{align}
    dZ_t &= \nabla h_n(t, X_t; X_{t_1:{t_n}})^T \sigma_t  \mathrm{d}W^{\mathbb{Q}}_t \\
    &=Z_t \left( \sigma^T_t \nabla \ln h_n(t, X_t; X_{t_1:{t_n}}) \right)^T \mathrm{d} W^{\mathbb{Q}}_t \ \ \ \ \ Z_t=h_n(t, X_t; X_{t_1:t_n})\\
    &=Z_t (\sigma_t^{-1} \alpha_t^*)^T \mathrm{d} W^{\mathbb{Q}}_t
\end{align}
where we defined $\alpha^*:=\sigma_t \sigma_t^T \nabla \log  h_n(t, X_t; X_{t_1:t_n})$. Let $a_t:=\sigma_t \sigma_t^T$, hence the density process admits the stochastic exponential representation:
\begin{align}
    \frac{\mathrm{d} \mathbb{P}^*}{\mathrm{d} \mathbb{Q}} \Big|_{\mathcal{F}_T}= \exp \Big( \int^T_0 (\sigma_t^{-1}\alpha_t^*)^T \mathrm{d} W^{\mathbb{Q}}_t - \frac{1}{2}\int^T_0 |\alpha_t^*|_{a_t^{-1}}^2 \mathrm{d}t \Big).
\end{align}
Under the change of measure we have
\begin{align}
    \mathrm{d}X_t = \left( b(t,X_t)+ \alpha^*_t\right) \mathrm{d}t + \sigma_t  \mathrm{d}W_t^{\mathbb{P}^*}.
\end{align}
To check that $(X_{t_1},X_{t_2},...,X_{t_N})$ follow the desired distribution:
\begin{align}
    \E^{\mathbb{P}^*}[f(X_{t_1},X_{t_2},...,X_{t_N})]&= \E^{\mathbb{Q}}[f(X_{t_1},X_{t_2},...,X_{t_N})\frac{\mu}{\mu_{\mathbf{X}}}(X_{t_1},X_{t_2},...,X_{t_N}) ] \\ 
    &= \int f(x_{1},x_{2},...,x_{N})\mu(x_1, ..., x_N) \mathrm{d}x_1...\mathrm{d}x_N.
\end{align}
So far, we have shown that the measure $\mathbb{P}^*$ induced by the controlled SDE with control $\alpha^*$ satisfies the prescribed joint distribution $\mu$. To establish optimality over all feasible controls, we note that: 
\begin{align}
   D_{KL}(\mu||\mu_{\mathbf{X}}) &= D_{KL}(\mathbb{P}^*||\mathbb{Q})\\ 
    &= \E^{\mathbb{P}^*}[\ln \frac{\mathrm{d} \mathbb{P}^*}{\mathrm{d} \mathbb{Q}}]=\E^{\mathbb{P}^*}\Big[\int^T_0 (\sigma_t^{-1}\alpha^*)^T \mathrm{d}W^{\mathbb{P}^*}_t +\int^T_0|\alpha^*|^2_{a_t^{-1}} \mathrm{d}t - \frac{1}{2}\int^T_0 |\alpha_t^*|_{a_t^{-1}}^2 \mathrm{d}t \Big] \\ 
    & = \frac{1}{2}\E^{\mathbb{P}^*}\Big[\int^T_0 |\alpha_t^*|_{a_t^{-1}}^2 \mathrm{d}t \Big]. \label{finite_check}
\end{align}

We now show the optimality: let $\alpha$ be a feasible control such that $(X_{t_1}, ..., X_{t_N}) \sim \mu$ and let $\mathbb{P}^\alpha$ be the induced path measure. Then we have
\begin{align}
    1&=\E^{\mathbb{Q}}[\frac{\mu}{\mu_{\mathbf{X}}}(X_{t_1},...,X_{t_{N-1}},X_{t_N})] \\ 
    &=\E^{\mathbb{P}^\alpha}\Big[ \exp \big( \ln \frac{\mu}{\mu_{\mathbf{X}}}(X_{t_1},...,X_{t_{N-1}},X_{t_N}) - \int^T_0 (\sigma_t^{-1} \alpha_t)^T dW^{\mathbb{P}^\alpha}_t  - \frac{1}{2}\int^T_0 |\alpha_t|^2_{a_t^{-1}} \mathrm{d}t \big) \Big] \\
    & \geq \exp \Big(\E^{\mathbb{P}^\alpha}[\ln \frac{\mu}{\mu_{\mathbf{X}}}(X_{t_1},...,X_{t_{N-1}},X_{t_N}) - \int^T_0 (\sigma_t^{-1} \alpha_t)^T dW^{\mathbb{P}^\alpha}_t  - \frac{1}{2}\int^T_0 |\alpha_t|_{a_t^{-1}}^2 \mathrm{d}t \big)] \Big) \\
    & = \exp \Big(\E^{\mathbb{P}^\alpha}[\ln \frac{\mu}{\mu_{\mathbf{X}}}(X_{t_1},...,X_{t_{N-1}},X_{t_N})  - \frac{1}{2}\int^T_0 |\alpha_t|_{a_t^{-1}}^2 \mathrm{d}t \big)] \Big)
\end{align}
which gives us 
\begin{align}
    \frac{1}{2}\E^{\mathbb{P}^\alpha}[\int^T_0 |\alpha_t|_{a_t^{-1}}^2dt ] \geq \E^{\mathbb{P}^\alpha}[\ln \frac{\mu}{\mu_{\mathbf{X}}}(X_{t_1},...,X_{t_{N-1}},X_{t_N})] =  D_{KL}(\mu||\mu_{\mathbf{X}}). 
\end{align}
This then shows that the control $u^*$ is optimal. This concludes the first part of the theorem. 

For the second part of the theorem, since the reference process is Markov and admits transition densities $q(s,y\mid t,x)$, its joint density satisfies
\begin{align}\label{transition_sde}
    \mu_{\mathbf{X}}(x_{1},...,x_{n}, ..., x_{N-1}, x_{N})=\prod^{N-1}_{j=0} q(t_{j+1},x_{j+1}|t_{j},x_{j}), \qquad x_0=0. 
\end{align}
Note that for $t \in [t_n, t_{n+1})$: 
\begin{align}
    h_n(t,x ; x_{1:n})&=\E^{\mathbb{Q}} \Big[\frac{\mu}{\mu_{\mathbf{X}}} \big(X_{t_1},...,X_{t_n}, X_{t_{n+1}},...,X_{t_{N}} \big)| X_t=x, X_{t_1: t_n}=x_{1:n}\Big] \\ 
    &=\int \frac{\mu}{\mu_{\mathbf{X}}} (x_{1},...,x_{n}, ... x_{N}) q(t_{n+1},x_{n+1} |t,x) \prod_{j=n+1}^{N-1} q(t_{j+1},x_{j+1} |t_{j},x_j) \mathrm{d}x_{n+1}...\mathrm{d}x_{N} \\
    &= \int \frac{\mu}{\mu_{\mathbf{X}}} (x_{1},...,x_{n}, ..., x_{N-1}, x_{N}) \frac{q(t_{n+1},x_{n+1} |t,x)}{q(t_{n+1},x_{n+1} |t_{n},x_n)}\prod_{j=n}^{N-1} q(t_{j+1},x_{j+1} |t_{j},x_j)\mathrm{d}x_{n+1}...\mathrm{d}x_{N} \\
    &= C \int \frac{q(t_{n+1},x_{n+1} |t,x)}{q(t_{n+1},x_{n+1} |t_{n},x_n)}   \mu(x_{1}, ..., x_{N-1}, x_{N})\mathrm{d}x_{n+1}...\mathrm{d}x_{N} \label{change_measure_1}\\
    &= C \int \frac{1}{\int \mu(x_1,...,x_N)\mathrm{d}x_{n+1}...\mathrm{d}x_{N}} \frac{q(t_{n+1},x_{n+1} |t,x)}{q(t_{n+1},x_{n+1} |t_{n},x_n)}  \mu(x_{1}, ..., x_{N-1}, x_{N}) \mathrm{d}x_{n+1}...\mathrm{d}x_{N} \label{change_measure_2} \\ 
    & = C \E^\mu[\frac{q(t_{n+1},X_{t_{n+1}} |t,x)}{q(t_{n+1},X_{t_{n+1}} |t_{n},x_n)} | X_{t_1: t_n}=x_{1:n}]
\end{align}
where in the lines above, the generic constant $C$ could change line by line and it depends on  ${\lbrace t_l,x_l \rbrace_{l=0,...,n}}$ but does not depend on $(t,x)$. 
As such, due to the cancellation of constant $C$, we have that
\begin{align}
    \alpha &=a_t \nabla \ln h_n(t,x; x_{1:n}) \\
    &= a_t\frac{ \nabla \E^\mu[\frac{q(t_{n+1},X_{t_{n+1}} |t,x)}{q(t_{n+1},X_{t_{n+1}} |t_{n},x_n)} |X_{t_1:t_n}=x_{1:n}]}{\E^\mu[\frac{q(t_{n+1},X_{t_{n+1}} |t,x)}{q(t_{n+1},X_{t_{n+1}} |t_{n},x_n)} |X_{t_1:t_n}=x_{1:n}]}. \label{final_form}
\end{align}
\end{proof}
Based on this result, we present a few choices for the drift term and the diffusion term. These choices correspond to the `Variance Exploding' and `Variance Preserving' cases in \cite{ys4}.

\begin{example}\label{eg_ve}
\textbf{Variance Exploding (VE).}
    
We first consider the SDE of the following form 
\begin{align}
    \mathrm{d}X_t = \sqrt{\frac{\mathrm{d}\beta_t}{\mathrm{d}t}} \mathrm{d}W_t, \qquad \beta_t \in C^1([0,T]), \beta'_t >0. 
\end{align}
The transition density is: 
\begin{align}
    q(t,x_t|s,x_s)  =\mcal{N}\Big(x_t| x_s, (\beta_t - \beta_s) I_d \Big). 
\end{align}
In this case, according to \eqref{final_form} we have 
\begin{align}
    \alpha^*(t,x;x_{1:n})=\frac{\beta'_t \E^\mu[(X_{t_{n+1}}-x) F(t_n,x_n,t,x,t_{n+1},X_{t_{n+1}}) |X_{t_1 : t_n}=x_{1:n} ]  }{( \beta_{t_{n+1}}-\beta_t)\E^\mu[ F(t_n,x_n,t,x,t_{n+1},X_{t_{n+1}}) |X_{t_1:t_n}=x_{1:n} ]}
\end{align}
where we define the unnormalized weight $F(t_n,x_n,t,x,t_{n+1},X_{t_{n+1}})= \exp(-\frac{|X_{t_{n+1}}-x|^2}{2(\beta_{t_{n+1}} -\beta_t)} + \frac{|X_{t_{n+1}}-x_{t_{n}}|^2}{2(\beta_{t_{n+1}}-\beta_{t_{n}})})$.
It is easily observed that when $\beta_t:=t$, the reference process reduces to standard Brownian motion, we recover the case in \cite{hamdouche2026nonparametric}. 
\end{example}

\begin{example} \label{eg_vp}
\textbf{Variance Preserving (VP)}

Consider the SDE of the form: 
\begin{align}
    \mathrm{d}X_t &= -\frac{1}{2} \beta_t X_t \mathrm{d}t + \sqrt{\beta_t} \mathrm{d}W_t \ \ 
\end{align}
where it is assumed that $\beta \in C^1([0,T]), \ \beta_t >0$. 
This process is conventionally referred to as the variance-preserving
SDE in the generative-modeling literature.
This is a Gaussian process whose transition density is 
\begin{align}
    q(t,x_t|s,x_s)= \mcal{N}\Big(x_t| x_s \gamma(s,t), \left(1-\gamma^2(s,t)\right)I_d\Big)
\end{align} 
where $\gamma(s,t):= e^{-\frac{1}{2}\int^{t}_{s} \beta_r dr}$.
In this case,  define
\begin{align}
F(t_n,x_n,t,x,t_{n+1},x_{n+1})
   &=\exp \Big( -\frac{|x_{n+1}-x \gamma(t,t_{n+1})|^2}{2(1- \gamma^2(t,t_{n+1}))} +\frac{|x_{n+1}-x_n  \gamma(t_n,t_{n+1})|^2}{2(1-\gamma^2(t_n,t_{n+1}))}\Big). 
\end{align}
It then follows from \eqref{final_form} that, for $t\in[t_n,t_{n+1})$, the optimal control is 
\begin{align}
    \alpha^*(t,x; x_{1:n})&=\frac{\beta_t \gamma(t,t_{n+1}) \E^\mu[(X_{t_{n+1}} - x \gamma(t,t_{n+1}))  F(t_n,x_n,t,x,t_{n+1},X_{t_{n+1}})|X_{t_1 : t_n}=x_{1:n} ]}{\big(1-\gamma^2(t,t_{n+1})\big) \E^\mu[ F(t_n,x_n,t,x,t_{n+1},X_{t_{n+1}})|X_{t_1 : t_n}=x_{1:n} ]}.
\end{align}
\end{example}

\begin{example}\textbf{Sub-Variance Preserving (sub-VP)}
Consider the reference SDE
\begin{equation}
\mathrm{d}X_t=-\frac12\beta_t X_t\,\mathrm{d}t+\sqrt{\beta_t\left(1-e^{-2A_t}\right)}\,\mathrm{d}W_t,\qquad A_t:=\int_0^t\beta_r\,\mathrm{d}r .
\end{equation}
Defining
\begin{equation}
\gamma(s,t):=\exp\left(-\frac12\int_s^t\beta_r\,\mathrm{d}r\right),
\end{equation}
then the transition law takes the form:
\begin{equation}
q(t,x_t| s, x_s) = \mathcal{N}\left(x_t | \gamma(s,t)x_s,\lambda(s,t)I_d\right),
\end{equation}
where
\begin{equation}
\lambda(s,t):=\int_s^t\gamma(r,t)^2\beta_r\left(1-e^{-2A_r}\right)\,\mathrm{d}r=\int_s^t\exp\left(-\int_r^t\beta_u\,\mathrm{d}u\right)\beta_r\left(1-e^{-2A_r}\right)\,\mathrm{d}r.
\end{equation}
For $t\in[t_n,t_{n+1})$, we have
\begin{equation}
F(t_n,x_n,t,x,t_{n+1},x_{t_{n+1}})=\exp\left(-\frac{|x_{n+1}-\gamma(t,t_{n+1})x|^2}{2\lambda(t,t_{n+1})}+\frac{|x_{n+1}-\gamma(t_n,t_{n+1})x_n|^2}{2\lambda(t_n,t_{n+1})}\right).
\end{equation}

In this case,
\begin{equation}
a_t=\sigma_t\sigma_t^\top=\beta_t\left(1-e^{-2A_t}\right)I_d.
\end{equation}
Since
\begin{equation}
\nabla_x\log q(t_{n+1},x_{n+1}\mid t,x)=\frac{\gamma(t,t_{n+1})}{\lambda(t,t_{n+1})}\left(x_{n+1}-\gamma(t,t_{n+1})x\right),
\end{equation}
we obtain the explicit expression
\begin{equation}
    \alpha^*(t,x;x_{1:n})=\frac{\beta_t\left(1-e^{-2A_t}\right)\gamma(t,t_{n+1})\mathbb E^\mu\left[\left(X_{t_{n+1}}-\gamma(t,t_{n+1})x\right)F(t_n,x_n,t,x,t_{n+1},X_{t_{n+1}})\mid X_{t_1:t_n}=x_{1:n}\right]}{\lambda(t,t_{n+1})\mathbb E^\mu\left[F(t_n,x_n,t,x,t_{n+1},X_{t_{n+1}})\mid X_{t_1:t_n}=x_{1:n}\right]}.
\end{equation}
\end{example}
To numerically approximate the conditional expectations, we again employ a kernel-based estimator. We tested several choices of kernels, including the standard Gaussian kernel and the quartic kernel $K(x) \propto (1-|x|^2)^2 \mathbf{1}_{|x|\leq 1}$. Across our numerical experiments, the quartic kernel achieved comparable or better accuracy than the Gaussian kernel. Moreover, because the quartic kernel has compact support, samples sufficiently far from the conditioning data can be discarded. When this property is implemented through an appropriate masking procedure, it can substantially reduce the computational cost.

Our experiments also indicate that the VE and VP reference processes consistently outperform the sub-VP reference process. We therefore focus the remainder of the numerical investigation on the VE and VP formulations. Based on the preceding discussion, the corresponding kernel approximations of the optimal drifts are given, for $t\in[t_n,t_{n+1})$, by
\begin{align}
 \hat{\alpha}(t,x;x_{1:n})  &= \frac{\beta'_t \sum^M_{i=1} (X^i_{t_{n+1}}-x) F(t_n,x_n,t,x,t_{n+1},X^i_{t_{n+1}}) \prod^n_{j=1} K_h(x_j-X^i_{t_j}) }{( \beta_{t_{n+1}}-\beta_t) \sum^M_{i=1} F(t_n,x_n,t,x,t_{n+1},X^i_{t_{n+1}}) \prod^n_{j=1} K_h(x_j-X^i_{t_j}) }, && \textbf{(VE)} \label{ve_app} \\ 
\hat{\alpha}(t,x; x_{1:n})  &= \frac{\beta_t \gamma(t,t_{n+1})\sum^M_{i=1} (X^i_{t_{n+1}}-x\gamma(t,t_{n+1})) F(t_n,x_n,t,x,t_{n+1},X^i_{t_{n+1}}) \prod^n_{j=1} K_h(x_j-X^i_{t_j})  }{( 1-\gamma^2(t,t_{n+1})) \sum^M_{i=1} F(t_n,x_n,t,x,t_{n+1},X^i_{t_{n+1}}) \prod^n_{j=1} K_h(x_j-X^i_{t_j})}. && \textbf{(VP)} \label{vp_app}
\end{align}

\subsection{Further numerical experiments.}

In this section, we compare the performance of the proposed algorithm under the VE and VP reference dynamics. We do not employ Gaussian mixture regularization in these experiments. For different choices of the diffusion profile $\sigma_t$, the corresponding regularization mechanism would need to be derived separately. Since the purpose of this experiment is to investigate how the choice of reference path measure affects the generative performance of the algorithm, rather than to examine its convergence properties, we restrict attention to the unregularized formulation. All experiments reported below are conducted on a MacBook Air equipped with an Apple M4 chip featuring four performance cores and six efficiency cores.

\begin{example}
In this example, we generate sample paths of fractional Brownian motion $B^H$ using the following parameter settings:
\begin{enumerate}
    \item The Hurst index is set to $H=0.2$, the terminal time is $T=1.0$, and the interval $[0,T]$ is discretized using $N=60$ time steps. The total number of training samples is $M=1500$.

    \item For the time-series generation procedure, the kernel bandwidth is set to $0.05$. On each interval $[t_n,t_{n+1}]$, we use $100$ numerical integration steps. For the VP reference process, the additional parameter is set to $\tau=2.0$.

    \item The computational times required to generate the sample paths are approximately $229$ seconds for the VE formulation and $347$ seconds for the VP formulation.
\end{enumerate}
We first report the estimated Hurst indices in Table~\ref{tab:fBM_index}, where ``H Data'' denotes the Hurst index estimated from the reference fBM sample paths, while ``H VE'' and ``H VP'' denote the corresponding estimates obtained from the paths generated using the VE and VP approaches, respectively.
\begin{table}[h]
\centering
\caption{Example: fractional BM Hurst index approximation comparison.}
\label{tab:fBM_index}
\begin{tabular}{|c|c|c|c|c|}  
\hline                       
Estimator &True Hurst & H Data & H VE & H VP \\ \hline
Mean &0.2   & 0.2019 & 0.2011 & 0.2026  \\ \hline
Std &--   & 0.025 & 0.024 & 0.025 \\ \hline
\end{tabular}
\end{table}
The estimated Hurst indices obtained from both generated data sets are close to the theoretical value $H=0.2$ and are comparable to those obtained from the reference data.

We also compare the temporal correlation structures of the three data sets. As shown in Figure~\ref{fig:fbm_fig}, the three empirical correlation matrices are visually very similar, indicating that both the VE and VP generated samples reproduce the temporal dependence structure of the reference fBM data well.
\begin{figure}[htbp]
    \centering
    \begin{subfigure}[b]{0.33\textwidth}
    \includegraphics[width=1.2\textwidth]{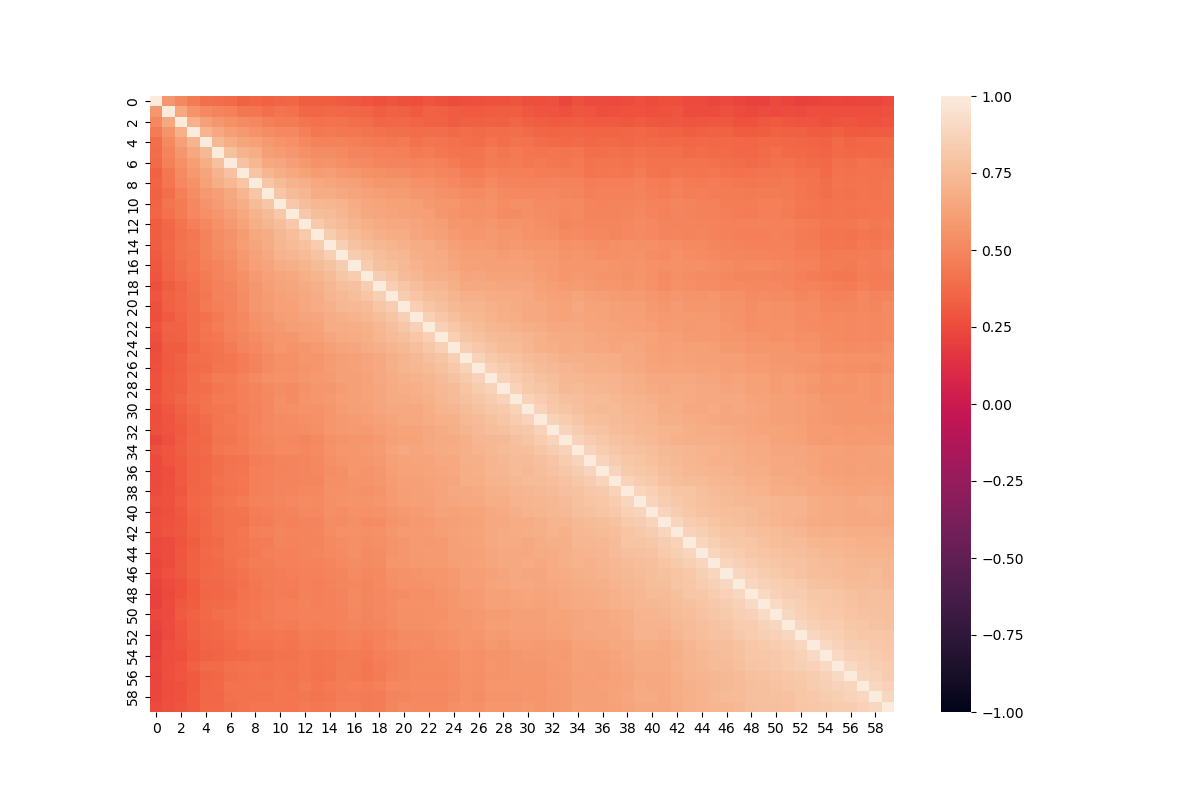}
    \caption{Estimated via Data}
    \label{fig:fbm_sub0}
    \end{subfigure}
    \hfill
    \hspace{-0.5cm}
    \begin{subfigure}[b]{0.33\textwidth}
        \includegraphics[width=1.2\textwidth]{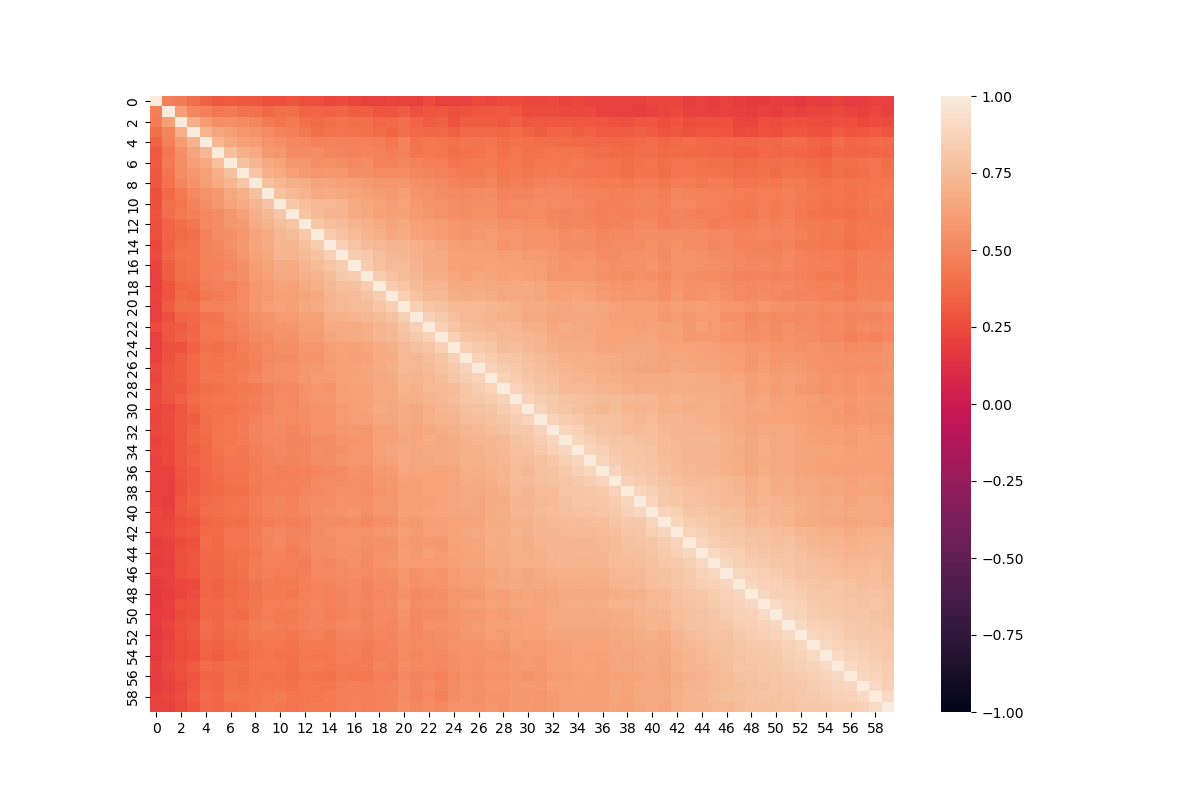}
        \caption{Estimated via VE SBTS}
        \label{fig:fbm_sub2}
    \end{subfigure}
    \hfill
    \hspace{-0.5cm}
    \begin{subfigure}[b]{0.33\textwidth}
        \includegraphics[width=1.2\textwidth]{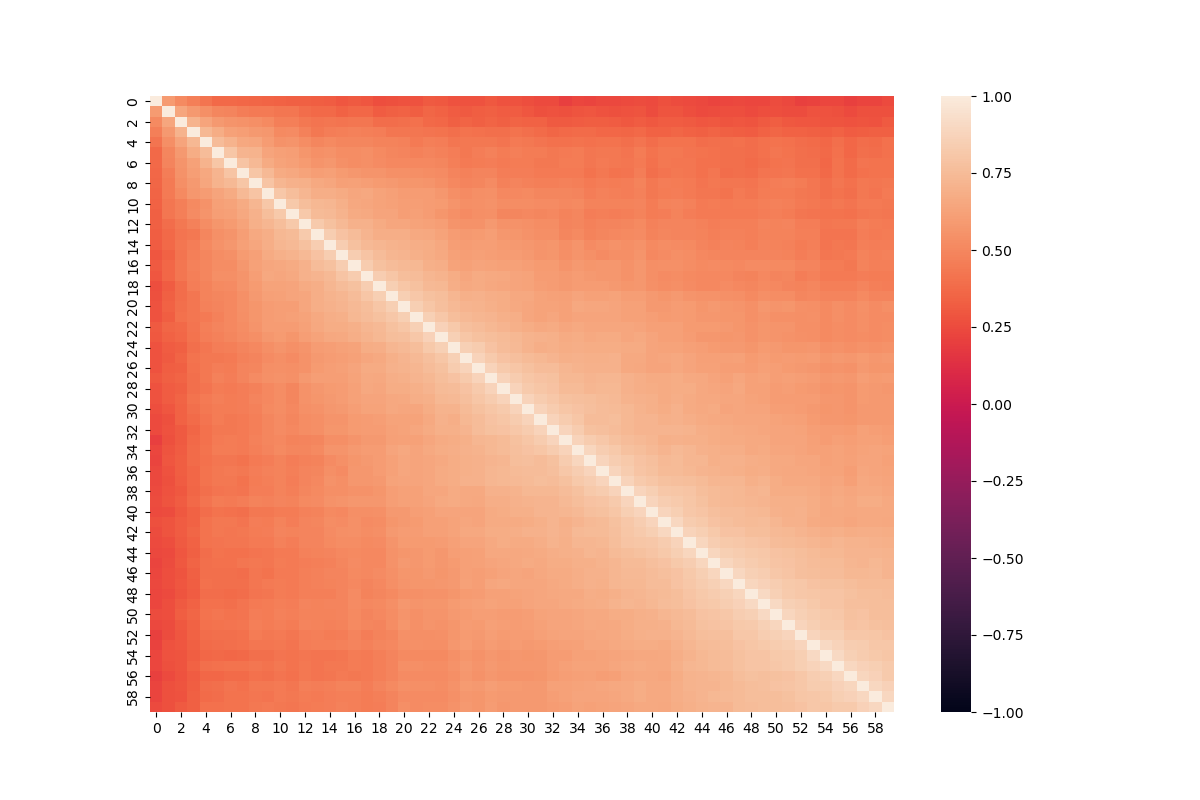}
        \caption{Estimated via VP SBTS}
        \label{fig:fbm_sub3}
    \end{subfigure}
    \caption{Comparison of the temporal correlation structures of fractional Brownian motion time series. Figure~\ref{fig:fbm_sub0} shows the correlation matrix computed from the reference data, while Figures~\ref{fig:fbm_sub2} and \ref{fig:fbm_sub3} show the corresponding correlation matrices for the samples generated using the VE and VP approaches, respectively.}
    \label{fig:fbm_fig}
\end{figure}

For further comparison, we compute the discrete quadratic variation of each fBM time series, $\text{QV}=\sum^{N-1}_{n=0} |X_{t_{n+1}}-X_{t_{n}}|^2$ and compare the corresponding empirical density estimates in Figure~\ref{fig:QV_fbm}. The three distributions exhibit broadly similar shapes, with only minor differences near their modes.
\begin{figure}[!htb]
\center
    \includegraphics[width=0.8\textwidth]{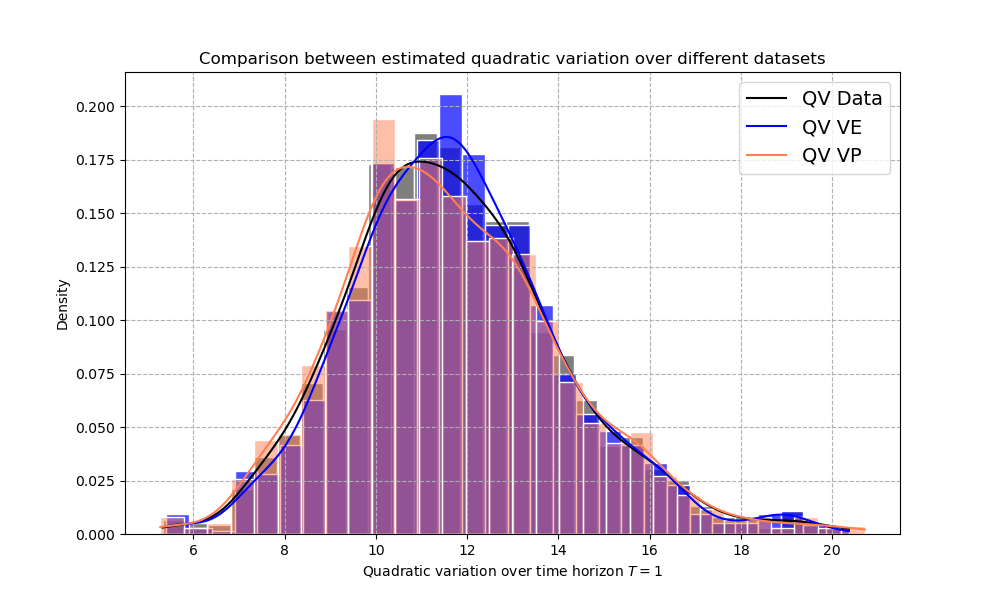}
    \caption{Comparison between quadratic variation}
    \label{fig:QV_fbm}
\end{figure}
\end{example}

\begin{example}
    \textbf{Generation of Black--Scholes sample paths.}

We next use the proposed framework to generate complete sample paths of a stochastic differential equation. Specifically, consider the geometric Brownian motion underlying the Black--Scholes model:
\begin{align}
    \frac{\mathrm{d}S_t}{S_t} = r \mathrm{d}t + \sigma \mathrm{d}W_t, \ \ \ S_0=s_0
\end{align}
where we set $s_0=1.0$, $r=0$, $\sigma=1.0$, and $T=1.0$. 
The exact solution takes the form: 
\begin{align}
    S_T = S_t \exp( (r-\frac{1}{2}\sigma^2) (T-t) + \sigma (W_T -W_t) ).\label{bs_sol_eg}
\end{align}
We first generate $10{,}000$ exact sample paths by evaluating the analytical solution on a uniform temporal grid with $N=100$ time steps. These paths are then used as training data for the SBTS generator, from which we generate $M=10{,}000$ artificial sample paths using the VE and VP reference processes. We subsequently assess and compare the quality of the samples generated by the two approaches.

\subsubsection*{Statistical assessment of the reconstructed Brownian motion}
We first invert the dynamics in \eqref{bs_sol_eg} to recover the Brownian increments $W_{t_{n+1}}-W_{t_n}$ from the corresponding asset-price paths. If the generated sample paths accurately reproduce the statistical properties of the target process, then the recovered Brownian increments should exhibit behavior similar to that of the true Brownian increments.

Specifically, the increments are reconstructed using
\begin{align}
    W_{t_{n+1}}-W_{t_{n}}= \sigma^{-1} \big( \ln \frac{S_{t_{n+1}}}{S_{t_n}} -(r-\frac{1}{2}\sigma^2)\Delta t \big)
\end{align}
We compare the resulting distributions using the sample minimum, maximum, mean, standard deviation, and the $1$st, $5$th, $15$th, $25$th, $50$th, $75$th, $90$th, $95$th, and $99$th percentiles. The recovered increments from the generated paths exhibit noticeable discrepancies from the true increments at the most extreme values, particularly in the sample minima and maxima. In contrast, the remaining summary statistics are closely aligned with those obtained from the true data, indicating that the generated paths reproduce the central part and moderate tails of the Brownian-increment distribution reasonably well.
\begin{figure}[!htb]
\center
    \includegraphics[width=0.6\textwidth]{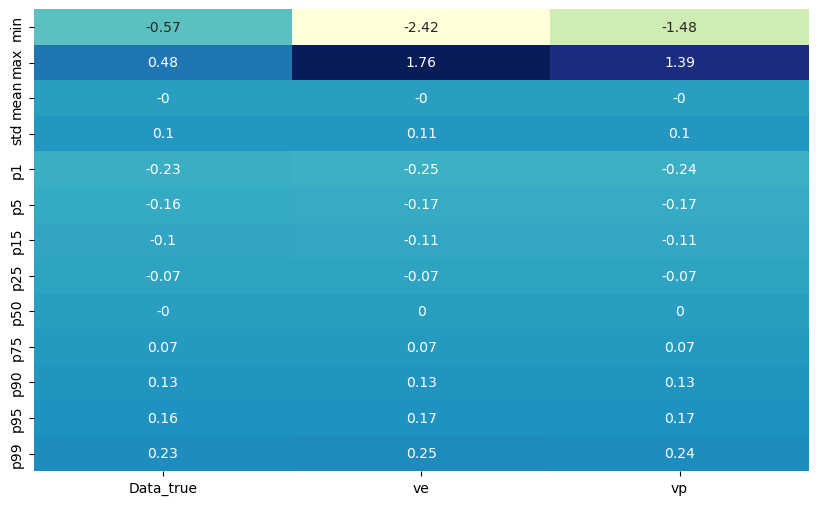}
    \caption{Comparison of summary statistics for the true and reconstructed Brownian increments.}
    \label{fig: bs_sol_stats}
\end{figure}

After recovering the Brownian increments, we reconstruct the corresponding Brownian paths by taking their cumulative sums. To assess whether the reconstructed processes exhibit the statistical properties of Brownian motion, we estimate the Hurst exponent using \eqref{estimate_Hhat}, following \cite{hurst}, and examine their temporal correlation structures:  
\begin{align} \label{estimate_Hhat}
    \hat{H} = \frac{1}{2} \Big( 1- (\log N)^{-1} \log (\sum^{N-1}_{i=0}|W_{t_{i+1}}-W_{t_i}|^2) \Big).
\end{align}

The estimated Hurst exponents of the implied Brownian paths generated by both approaches are close to the theoretical value $H=0.5$. However, the estimated Hurst exponents exhibit greater variability than those obtained from the true Brownian paths.

\begin{table}[h]
\centering
\caption{Comparison of estimated Hurst exponents for the reconstructed Brownian paths.}
\label{tab:sb_BM_index}
\begin{tabular}{|c|c|c|c|c|}  
\hline                       
Estimator &Theoretical H & H Data & H VE  & H VP \\ \hline
Mean &0.5   & 0.50 & 0.49 & 0.49  \\ \hline
Std &--   & 0.015 & 0.026 & 0.020 \\ \hline
\end{tabular}
\end{table}
 We also compare the temporal correlation structures of the true and reconstructed Brownian paths. As shown in Figure~\ref{fig:bs_bm_fig}, the three empirical correlation matrices are visually very similar, indicating that the Brownian paths reconstructed from both the VE and VP generated price processes reproduce the temporal correlation structure of the reference Brownian motion well.
\begin{figure}[htbp]
    \centering
    \begin{subfigure}[b]{0.33\textwidth}
    \includegraphics[width=1.2\textwidth]{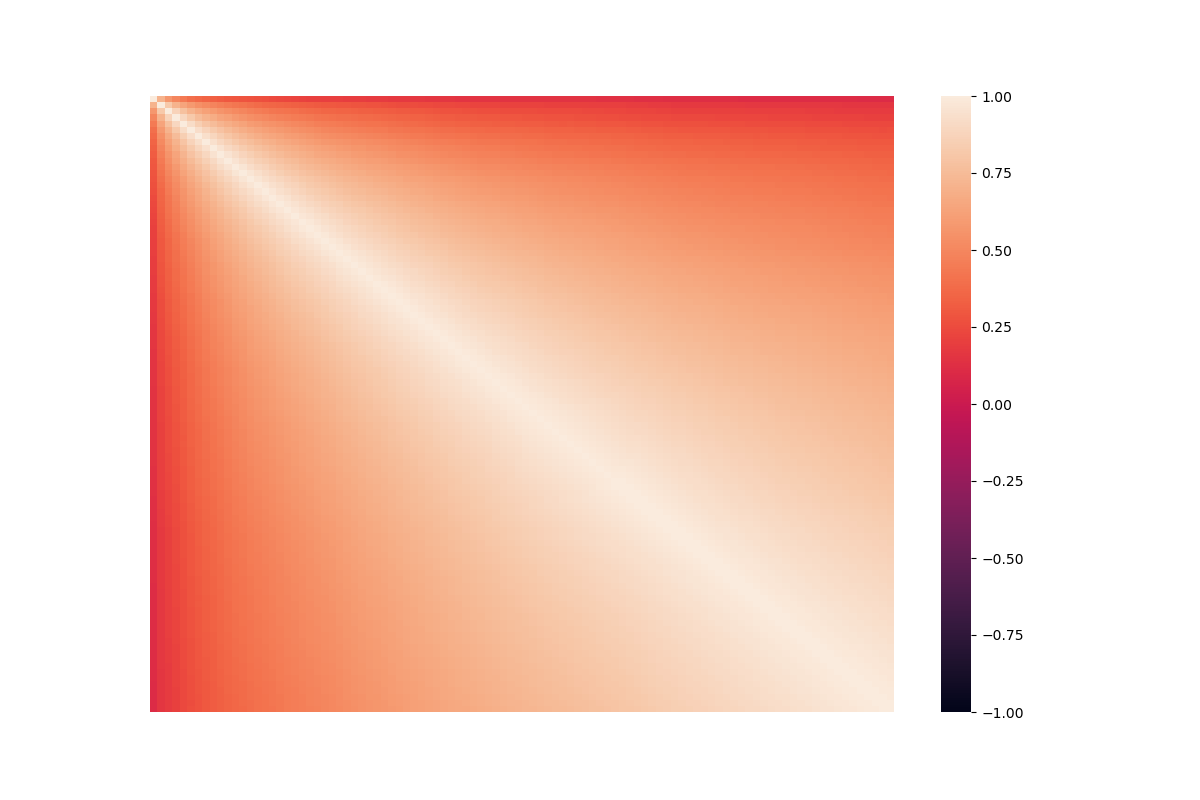}
    \caption{Reference data}
    \label{fig:bs_bm_sub0}
    \end{subfigure}
    \hfill
    \hspace{-0.5cm}
    \begin{subfigure}[b]{0.33\textwidth}
        \includegraphics[width=1.2\textwidth]{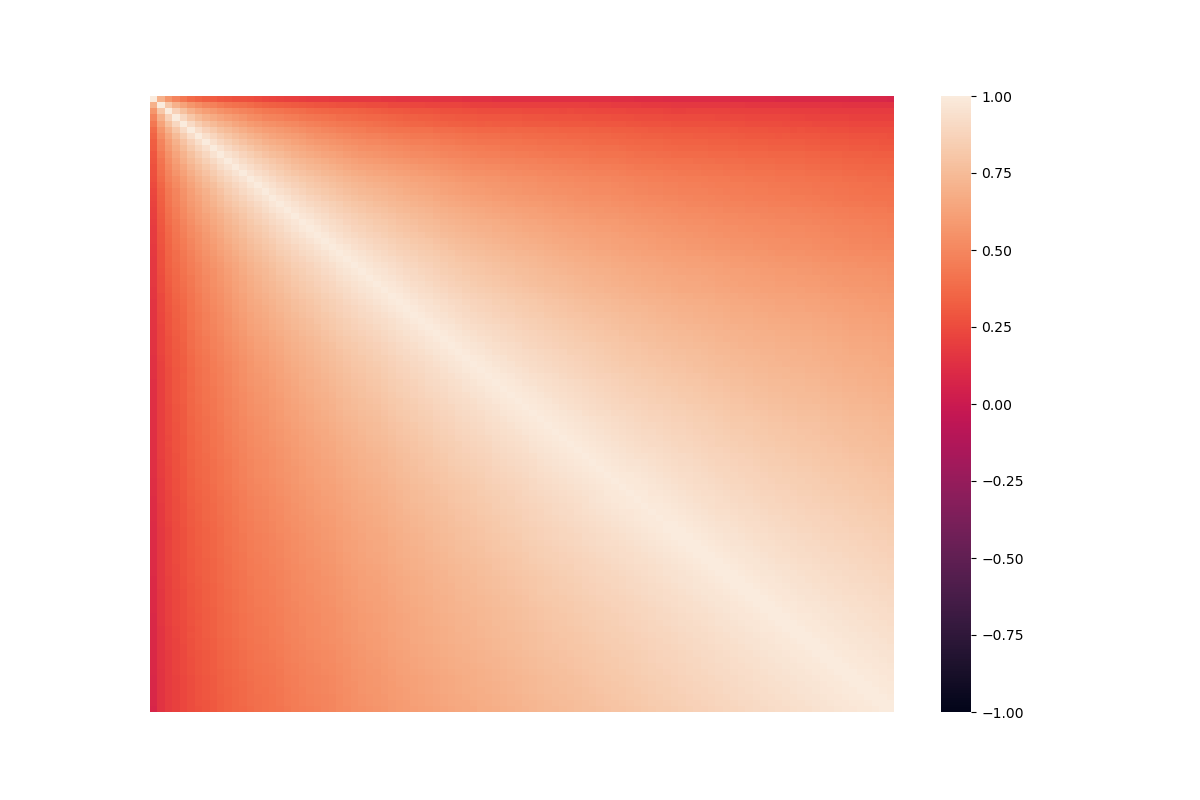}
        \caption{VE SBTS}
        \label{fig:bs_bm_sub2}
    \end{subfigure}
    \hfill
    \hspace{-0.5cm}
    \begin{subfigure}[b]{0.33\textwidth}
        \includegraphics[width=1.2\textwidth]{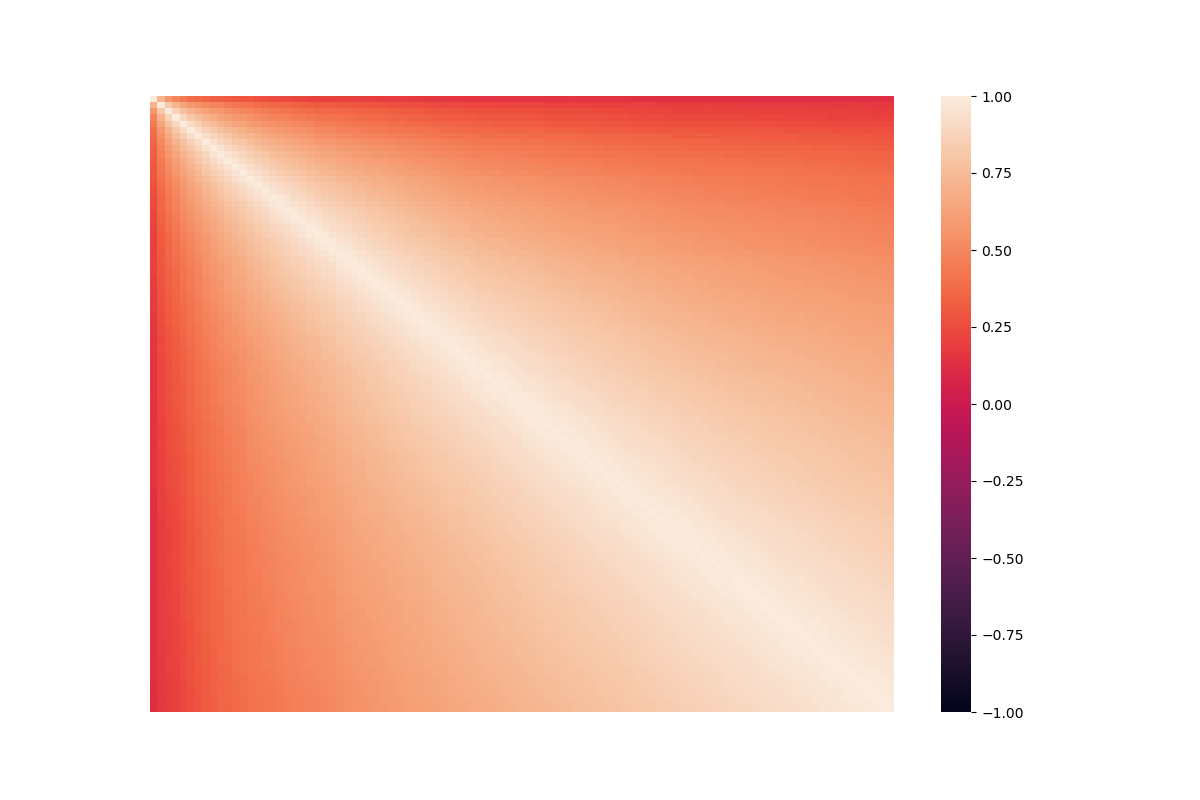}
        \caption{VP SBTS}
        \label{fig:bs_bm_sub3}
    \end{subfigure}
    \caption{Comparison of the empirical temporal correlation matrices of the reference Brownian paths and the Brownian paths reconstructed from the VE and VP generated price processes.}
    \label{fig:bs_bm_fig}
\end{figure}
\end{example}
We next examine the generated price process $S_t$ through the path-dependent quantity
\begin{align}
    Y_T:= \max_{t_n \in \lbrace t_0, t_1,...,t_N\rbrace} S_{t_n} - S_T.
\end{align}

Since $Y_T$ measures the difference between the maximum value attained along the discretized path and its terminal value, it provides a useful diagnostic for assessing whether the generated samples preserve this aspect of the path-dependent behavior of the underlying process. Using $M=10{,}000$ sample paths with $N=100$ time steps, we compute the empirical distributions of $Y_T$ from the exact Black--Scholes paths and from the paths generated by the two SBTS schemes. We then compare the generated and reference samples using the two-sample Kolmogorov--Smirnov (KS) test and the Anderson--Darling (AD) test. For both tests, the null hypothesis is that the two samples are drawn from the same underlying distribution.

The results reported in Tables~\ref{tab:yt_ks_index} and \ref{tab:yt_ad_index} show that neither the VE nor the VP samples lead to rejection of the null hypothesis at the $5\%$ significance level. Thus, the tests do not provide statistically significant evidence of a difference between the distributions of $Y_T$ obtained from the generated and exact sample paths. These results suggest that both SBTS schemes reproduce the distribution of this path-dependent quantity reasonably well.

\begin{table}[h]
\centering
\caption{Two-sample Kolmogorov--Smirnov tests comparing the distributions of $Y_T$ obtained from the generated and exact Black--Scholes sample paths.}
\label{tab:yt_ks_index}
\begin{tabular}{|c|c|c|}
\hline
 & VE vs.\ Exact & VP vs.\ Exact \\ \hline
$p$-value & 0.47 & 0.86 \\ \hline
Decision at $5\%$ level & Fail to reject & Fail to reject \\ \hline
\end{tabular}
\end{table}

\begin{table}[h]
\centering
\caption{Anderson--Darling tests comparing the distributions of $Y_T$ obtained from the generated and exact Black--Scholes sample paths.}
\label{tab:yt_ad_index}
\begin{tabular}{|c|c|c|}
\hline
 & VE vs.\ Exact & VP vs.\ Exact \\ \hline
$p$-value & $>0.25$ & $>0.25$ \\ \hline
Decision at $5\%$ level & Fail to reject & Fail to reject \\ \hline
\end{tabular}
\end{table}

\newpage
\subsubsection*{Application in Deep Hedging}

In this section, we apply a deep-hedging procedure using the time-series data generated from the Black--Scholes model \eqref{bs_sol_eg} and compare the resulting hedging performance with that obtained using the original sample paths. This experiment serves two purposes: first, to further assess the reliability of the artificially generated data; and second, to illustrate a practical downstream application of the proposed time-series generation framework.

Under the Black--Scholes setting, the market is complete and the contingent claim is attainable. Then following \cite[Chapter~1]{guyon2014nonlinear}, the discounted hedging relation can be written as
\begin{align}\label{dh_bs}
    z + \int_0^T \Delta^i_s \, \mathrm{d}\widehat{X}_s - D_{0T}F_T = 0, \quad a.s.
\end{align}
where $z$ denotes the option premium, $D_{st}:=\exp(-r(t-s))$ is the discount factor, $F_T$ is the terminal payoff, and $\widehat{X}^i_t:=D_{0t}X^i_t$ is the discounted price process of the risky asset.

The premium $z$ can then be estimated using a deep hedging approach. Specifically, we treat $z$ as a trainable parameter and approximate the hedging strategies $\Delta^i_t$ at the discrete observation times using neural networks. The parameters are obtained by minimizing the $L^2$ residual of the hedging identity \eqref{dh_bs}, in the spirit of the deep hedging framework proposed in \cite{buehler2019deephedging}.

More specifically, in the discrete-time implementation, we minimize the empirical loss
$$ L^M\left(z,\{\theta_i\}_{i=0}^{N-1}\right)
    := \frac{1}{M}\sum_{\ell=1}^{M}
    \left( z + \sum_{i=0}^{N-1} \mathcal{U}_i\left(X^\ell_{t_i};\theta_i\right)
        \left(\widehat X^\ell_{t_{i+1}}-\widehat X^\ell_{t_i}\right) - D_{0T}\left(X^\ell_T-K\right)^+\right)^2. $$
The function $\mathcal{U}_i(\cdot;\theta_i)$ is a neural network approximation of the hedging strategy $\Delta_{t_i}(\cdot)$ at time $t_i$, and $\theta_i$ denotes its trainable parameters. The option premium $z$ is also optimized jointly with the hedging networks. Thus, minimizing $L^M$ corresponds to fitting the discrete-time self-financing hedging relation in an empirical $L^2$ sense.

We train three deep-hedging models separately using $M=10,000$ true Black--Scholes paths, VE-generated paths, and VP-generated paths, respectively. To provide an out-of-sample comparison, we further generate an independent test set consisting of $10{,}000$ new Black--Scholes paths. The three trained models are then evaluated on this same test set. In particular, the terminal hedging PnL is computed as
\begin{align}
    \Pi_T:=z+\sum_{i=0}^{N-1}
\mathcal{U}_i\left(X_{t_i};\theta_i\right) \left(\widehat X_{t_{i+1}}-\widehat X_{t_i}\right) -D_{0T}(X_T-K)^+.
\end{align}
The resulting premium estimates and out-of-sample PnL statistics are reported in Table~\ref{tab:yt_dh_index}, while the corresponding PnL distributions are shown in Figure~\ref{fig:bs_dh_fig}. The main observations are summarized as follows:
\begin{enumerate}
    \item The exact Black--Scholes call option price is $0.3829$. The premium estimates obtained using the true, VE-generated, and VP-generated training data are $0.3830$, $0.3833$, and $0.3862$, respectively. Thus, all three estimates remain close to the analytical benchmark, with the estimate obtained from the VE-generated data particularly close to that obtained from the true data.

    \item When evaluated on the same independent Black--Scholes test set, all three PnL distributions are concentrated around zero. The PnL MSEs are $0.00764$, $0.00735$, and $0.00978$ for the models trained on the true, VE-generated, and VP-generated data, respectively. These results indicate that, in this experiment, training on the VE-generated data produces downstream hedging performance particularly close to that obtained using the true data. 
\end{enumerate}

\begin{table}[h]
\centering
\caption{Comparison of the option premium estimates and out-of-sample hedging performance.}
\label{tab:yt_dh_index}
\begin{tabular}{|c|c|c|c|c|}
\hline
 & Exact price & True data & VE & VP \\ \hline
Option premium & 0.3829 & 0.3830 & 0.3833 & 0.3862 \\ \hline
PnL mean & -- & -0.0007 & -0.0010 & 0.0019 \\ \hline
PnL std & -- & 0.0874 & 0.0857 & 0.0989 \\ \hline
PnL MSE & -- & 0.00764 & 0.00735 & 0.00978 \\ \hline
\end{tabular}
\end{table}

\begin{figure}[htbp]
    \centering
    \begin{subfigure}[b]{0.5\textwidth}
    \includegraphics[width=\textwidth]{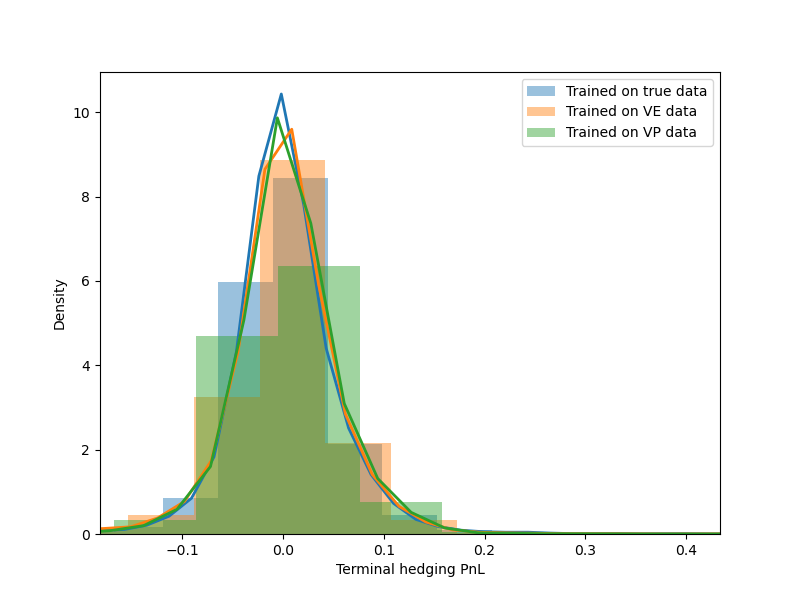}
    \caption{PnL distribution comparison}
    \label{fig:bs_dh_sub0}
    \end{subfigure}
    \hfill
    \hspace{-0.5cm}
    \begin{subfigure}[b]{0.5\textwidth}
        \includegraphics[width=\textwidth]{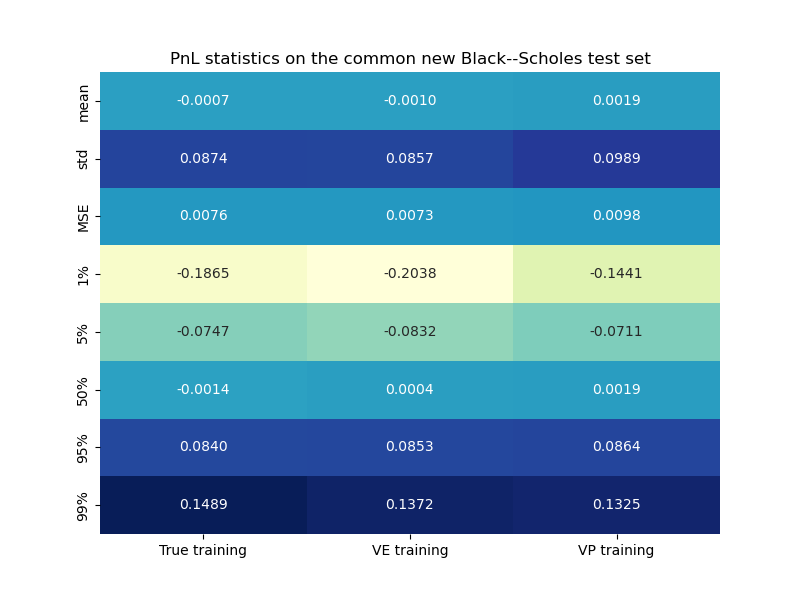}
        \caption{PnL summary statistics}
        \label{fig:bs_dh_sub2}
    \end{subfigure}
    \caption{Comparison of the out-of-sample PnL distributions and summary statistics for the deep-hedging models trained using the true Black--Scholes data and the VE- and VP-generated data. All three models are evaluated on the same independent test set of $10000$ Black--Scholes sample paths.}
    \label{fig:bs_dh_fig}
\end{figure}

\section{Conclusion and future work}
In this work, we provide a convergence analysis for the Schr{\"o}dinger Bridge Time Series data generator based on a distribution-mixture regularization technique. We prove convergence of the regularized scheme and further justify the limiting behavior as the regularization parameter tends to zero. The numerical experiments show that the theoretical convergence rate provides a conservative bound for the empirical decay of the $\mathcal{W}_2$ error. As an additional empirical study, we also test the framework by replacing the Wiener reference measure with the path measure induced by a more general SDE. The resulting generated samples exhibit quality comparable to those produced by the original Brownian-reference formulation, suggesting that the SBTS framework is robust with respect to the choice of reference dynamics. Future work will focus on sharpening the convergence analysis to obtain rates that more closely reflect the behavior observed in numerical experiments.

\end{document}